\documentclass[12pt]{amsart}
\usepackage[margin=2cm]{geometry}

\usepackage{amsmath, amsthm, amssymb}
\input xy
\xyoption{all}
\usepackage{mathrsfs}
\usepackage{enumerate}
\usepackage{caption}
\usepackage{comment}
\usepackage{subcaption}
\usepackage{graphicx}
\usepackage{color}
\usepackage{tikz}
\usepackage{tikz-cd}

\usepackage{tkz-euclide}
\usepackage{bbm}

\usepackage{tikzsymbols}
\usepackage{scalerel}

\usepackage[driverfallback=hypertex]{hyperref}
\usepackage{nameref,zref-xr}                    %to include reference to other papers

\newtheorem{thm}{Theorem}[section]
\newtheorem{prop}[thm]{Proposition}

\newtheorem{conj}[thm]{Conjecture}
\theoremstyle{definition}
\newtheorem{definition}[thm]{Definition}
\newtheorem{nn}[thm]{Notation}

\newtheorem{question}[thm]{Question}
\theoremstyle{remark}
\newtheorem{rmk}[thm]{Remark}

\newtheorem{ex}[thm]{Example}
\newtheorem{obs}[thm]{Observation}

\newcommand{\CM}{{\mathcal{M}}}
\newcommand{\PST}{{PST}}
\newcommand{\BCT}{{BCT}}
\newcommand{\GKT}{{GKT}}
\newcommand{\TZ}{{TZ}}

\newcommand{\MU}
{\mathfrak{m}}

\newcommand{\oCM}
{{\overline{\mathcal{M}}}}

\newcommand{\CL}{{\mathbb{L}}}

\newcommand{\detach}{{\textup{detach}}}

\newcommand{\alt}{{\text{alt} }}

\newcommand{\op}[1]{\operatorname{#1}}
\newcommand{\Aut}{\op{Aut}}

\newcommand{\CS}{{\mathcal{S}}}

\renewcommand{\Aut}{\text{Aut}}

\newcommand{\cS}{\mathcal{S}}

\newcommand{\cW}{\mathcal{W}}

\newcommand{\tw}{\text{tw}}
\newcommand{\NNN}{{\nu}}

\newcommand{\vecd}{\mathbf{d}}

\newcommand{\Detach}{{{\textup{Detach}}}}

\newcommand{\rank}{\operatorname{rank}}

\newcommand{\spinqi}{{\text{spin }\neq i}}

\newcommand{\Z}{\ensuremath{\mathbb{Z}}}
\newcommand{\Q}{\ensuremath{\mathbb{Q}}}
\newcommand{\C}{\ensuremath{\mathbb{C}}}
\newcommand{\R}{\ensuremath{\mathbb{R}}}
\renewcommand{\P}{\ensuremath{\mathbb{P}}}
\newcommand{\M}{\ensuremath{\overline{\mathcal{M}}}}
\renewcommand{\O}{\ensuremath{\mathcal{O}}}

\newcommand{\ev}{\ensuremath{\textrm{ev}}}

\renewcommand{\d}{\ensuremath{\partial}}

\newcommand{\<}{\left<}
\renewcommand{\>}{\right>}

\DeclareMathOperator{\res}{res}

\newcommand{\h}{{\mathfrak{h} }}

\DeclareMathOperator{\re}{Re}

\allowdisplaybreaks[4]

\usepackage{geometry}
\numberwithin{equation}{section}

\begin{document}

\title{Open Enumerative Geometries for Landau-Ginzburg Models}

\author{Mark Gross}
\address{M. Gross: \newline Department of Pure Mathematics and Mathematical Statistics, University of Cambridge, CB4 0WB, United Kingdom}
\email{mgross@dpmms.cam.ac.uk}
\author{Tyler L. Kelly}
\address{T. L. Kelly: \newline School of Mathematical Sciences, Queen Mary University of London, United Kingdom}
\email{t.l.kelly@qmul.ac.uk}

\author{Ran J. Tessler}
\address{R. J. Tessler:\newline Department of Mathematics, Weizmann Institute of Science, Rehovot, Israel}
\email{ran.tessler@weizmann.ac.il}

\maketitle

\begin{abstract}
We survey the recent progress in defining open enumerative theories for Landau-Ginzburg models.  We illustrate the ideas required to develop these new foundations. In particular, we describe how to define the open enumerative invariants as integrals of multisections of certain vector bundles over a moduli space that is a real orbifold with corners, after prescribing boundary conditions for the multisections. We then explain the known situations where the open invariants satisfy certain forms of topological recursion relations, integrable hierarchies, or mirror symmetry. We end with a list of open questions and problems.
\end{abstract}

\setcounter{tocdepth}{1}
\tableofcontents
\section{Introduction}

The moduli space of closed Riemann surfaces $\mathcal{M}_g$ of genus $g$
is a fundamental object of study in geometry. 
Along with this space, one often studies
the closely related spaces $\mathcal{M}_{g,n}$ of genus $g$ Riemann
surfaces with $n$ marked points and its compactification
$\overline{\mathcal{M}}_{g,n}$, the moduli space of stable curves of genus
$g$ with $n$ marked points. Witten's celebrated conjecture 
\cite{Witten2DGravity} considered
intersection numbers on
$\overline{\mathcal{M}}_{g,n}$ of combinations of 
$\psi$-classes (first Chern classes of the conormal bundle at marked points).
The conjecture proposed that if one organized these intersection
numbers
into an appropriate generating function, one obtained a 
$\tau$-function for the KdV
hierarchy. Witten's conjecture was proved by Kontsevich in \cite{Kontsevich},
thereby demonstrating the immensely rich properties that these intersection
numbers possess. 

Motivated by further explorations of Witten \cite{Witten93} and building on the subsequent work on $r$-spin theory, Fan,
Jarvis and Ruan introduced in \cite{FJR} what is now known as FJRW theory.
This theory begins with a choice of finite group $G$ acting on $\C^n$ and
a $G$-invariant polynomial $W\in \C[x_1,\ldots,x_n]$. This provides a
modified kind of moduli space, whose objects are orbicurves (stable curves with
stacky structure at marked points and nodes) along with
certain line bundles with properties determined by $G$ and $W$. While the
resulting moduli spaces are closely related to $\overline{\mathcal{M}}_{g,n}$,
they are not the same, and they carry in addition a virtual fundamental class
which in many nice cases is realised as the top Chern class of a vector
bundle. See \S\ref{Background on (closed) FJRW theory}
below for details. This theory provides new invariants, which often
are related to integrable hierarchies and mirror symmetry.

Another thread in the study of moduli of Riemann surfaces has been the
study of \emph{open} Riemann surfaces. Here, one really means considering
Riemann surfaces with boundary. Such Riemann surfaces have played
a key role in sympletic geometry, and especially Floer theory, for a long
time. Typically, one might be interested in counting $J$-holomorphic
maps from Riemann surfaces with boundary into a symplectic manifold
with the boundary of the Riemann surface mapping to a given Lagrangian,
see e.g., \cite{fukaya2009lagrangian,FOOO_ii,FOOO1,katz2006enumerative,solomon2006intersection}.

More recently, with analogues of Witten's conjecture in mind,
work of Pandharipande, Solomon and Tessler \cite{PST14} began a study
of moduli of open Riemann surfaces analogous to that of the moduli of
closed surfaces discussed above. Already, many basic issues appear. The
moduli spaces in question are now real orbifolds with corners rather than
complex orbifolds, and thus defining intersection numbers becomes more
difficult. Intuitively, one defines the Euler number of a bundle $\mathcal{E}$
on an orbifold with corners $M$ with $\rank \mathcal{E}=\dim M$ by
choosing some boundary conditions for a section $s$ of $\mathcal{E}$
and then counting the number of zeros of such a section. However, to
obtain a set of intersection numbers which satisfies any reasonable
analog of Witten's conjecture requires careful inductive choices for
these boundary conditions. With such choices the generating function of the resulting intersection numbers were shown to yield the KdV wave function \cite{Bur16,BT17, Tes15}.

Work of Buryak, Clader and Tessler \cite{BCT1,BCT2,BCT3} then considered
the $r$-spin case. This is a special case of FJRW theory where
$W=x^r$ and $G$ is the group of $r$th roots of unity \cite{FJRSpin}. They were able
to build an open theory in a similar way to the work of \cite{PST14}, and
so obtained invariants enjoying many nice properties, e.g. a relation to the $r$-KdV wave function \cite{BCT2,TZ3} and mirror symmetry \cite{GKT}.

Going further, Gross, Kelly and Tessler built a theory \cite{GKT} for
rank 2 Fermat polynomials $W=x^r+y^s$, again with full symmetry group 
$G$. While the broad outlines of the construction were the same, 
there were a number of crucial differences. Most importantly, unlike the
previous two cases considered, the possible invariants could not be
uniquely defined. Instead, the set of all invariants could undergo
a wall-crossing as boundary conditions are changed; however, this wall-crossing
is completely controlled. In addition, using these invariants, the
authors of \cite{GKT} were able to prove a mirror symmetry statement
for \emph{closed} FJRW invariants, which showed that oscillatory integrals
involving a generating function for the open invariants yielded the closed
invariants. 

 Tessler and Zhao provided another construction of open $r$-spin and certain Fermat FJRW theories \cite{TZ1,TZ2}. One of them, which corresponds to the Fermat quintic $x_1^5+\cdots +x_5^5$ with minimal symmetry group $\mu_5$, is conjectured to satisfy an open analog of the Landau-Ginzburg/Calabi-Yau correspondence, predicted in \cite{Walcher,Melissa}.

This article is intended to be a survey of the above developments.
We begin with background on closed FJRW theory, and then outline the
construction of the invariants in the various cases described above.
We introduce the key ingredient, open $W$-spin surfaces, and then discuss
their moduli and the analog of the Witten bundle. We describe in some
detail the construction of the relevant moduli spaces and explain how
to define open intersection numbers. Finally, we describe the properties
these various open invariants enjoy, from topological recursion relations,
mirror symmetry, and wall-crossing to analogs of the KdV hierarchy.

\subsection*{Acknowledgements} We thank Shing-Tung Yau for inviting us to write
this survey. MG was supported by the ERC Advanced Grant MSAG. TK was supported by the UKRI Future Leaders Fellowship MR/Y033841/1 and EPSRC Small Grant EP/Y033574/1. TK also thanks the Sydney Mathematical Research Institute where portions of this survey were written.
RT was supported by the ISF grant 1729/23.

\section{Background on (closed) FJRW theory}
\label{Background on (closed) FJRW theory}

In this section, we describe Fan-Jarvis-Ruan-Witten (FJRW) theory. FJRW theory is a closed enumerative theory for Landau-Ginzburg models that has been developed in a sequence of papers \cite{FJR,FJR2,FJR3,polishchuk2011matrix}. This theory generalized Witten's proposal \cite{Witten93} of a closed $r$-spin enumerative theory that had been developed in the meantime \cite{ChiodoStable, ChiodoWitten, Jarvis, JKV, JKV2, FJRSpin}. It gives enumerative invariants for a Landau-Ginzburg model of the form $(X,G,W)$ where $X$ is a complex vector space, $G$ a finite group acting on $X$ and $W: X \to \mathbb{A}^1$ a $G$-invariant regular map. Over the past decade, there has been progress in extending FJRW theory to more general contexts, such as gauged linear sigma models \cite{clader2017landau, FJRGLSM, CFFGKS, FaveroKim}. However, these constructions at present do not have an open counterpart. To streamline our discussion, we will restrict the scope of our discussion of closed FJRW theory to the contexts where there are open FJRW theories to discuss. In (closed) FJRW theory, enumerative invariants are defined as certain characteristic classes over the moduli space of $W$-spin curves. We will start by defining closed $W$-spin Riemann surfaces, describing their moduli space, and introducing important characteristic classes of the moduli. We then define associated enumerative theories and give some results that represent their richness as invariants. We first consider the special
case of an $r$-spin surface, reviewing terminology from \cite{BCT1}.  

\subsection{$W$-spin surfaces}
A \emph{closed marked genus $g$ orbifold Riemann surface $C$}
is a proper, one-dimensional Deligne--Mumford stack whose coarse
space $|C|$ is a genus $g$ Riemann surface with at worst nodal singularities, and such that the morphism $\pi: C \to |C|$ is an
isomorphism away from the finitely many marked points and nodes.
We collectively call these marked points and nodes the \emph{special points}. Any special point may have a non-trivial stabilizer which is required to be a finite cyclic group. Any node must be \emph{balanced}, i.e.,
the local picture at any node is
\begin{equation}\label{eq:balanced}
\{xy=0\}/\mu_d,
\end{equation}
where $\mu_d$ is the group of $d^{th}$ roots of unity, and generator $\zeta$ acts by
\begin{equation}
\label{eq:balanced action}
\zeta\cdot(x,y)=(\zeta x,\zeta^{-1}y).
\end{equation}
The curve $C$ is \emph{$d$-stable} if the isotropy group for all special points is the group $\mu_d$.

\begin{definition}
\label{def:closed normal}
Let $q$ be a node of $C$. We denote by
\[
\NNN_q:{C}'_q\to C
\]
the normalization of $C$ at the node $q$ and by
\[
\NNN:\widehat{C}\to C
\]the normalization of $C$ at all nodes. 
A \emph{half-node} is a point of the normalization $\widehat{C}$
mapping to a node $q$ in $C$. We often refer to a half-node also as a node $q\in C$
along with a choice of branch of $C$ at $q$.
\end{definition}

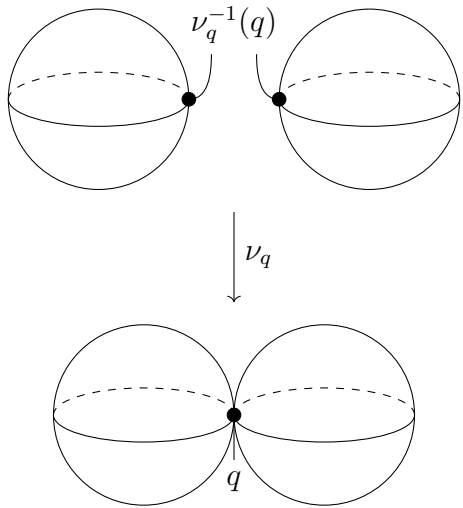
\begin{figure}
  \centering
\begin{tikzpicture}[scale=0.6]

  \draw (-1,7) circle (2cm);
  \draw (-3,7) arc (180:360:2 and 0.6);
  \draw[dashed] (1,7) arc (0:180:2 and 0.6);
    \draw [fill] (1,7) circle [radius=0.15];
    
  \draw (5,7) circle (2cm);
  \draw (3,7) arc (180:360:2 and 0.6);
  \draw[dashed] (7,7) arc (0:180:2 and 0.6);
  \draw [fill] (3,7) circle [radius=0.15];

  \node[above] at (2,8) {$\nu_q^{-1}(q)$};
  \draw (2.5,8) to [out=270,in=180] (3,7);

    \draw (1.5,8) to [out=270,in=0] (1,7);

  \draw [->] (2,4.5) -- (2,2.5);
  \node [right] at (2,3.5) {$\nu_q$};

%The nodal curve

  \draw (0,0) circle (2cm);
  \draw (-2,0) arc (180:360:2 and 0.6);
  \draw[dashed] (2,0) arc (0:180:2 and 0.6);

\draw [fill] (2,0) circle [radius=0.15];
\node [below] at (2,-1) {$q$};
    \draw (2,-1) to (2,0);

%\draw [fill] (0,-1) circle [radius=0.15];

  \draw (4,0) circle (2cm);
  \draw (2,0) arc (180:360:2 and 0.6);
  \draw[dashed] (6,0) arc (0:180:2 and 0.6);

\end{tikzpicture}

\caption{The partial normalization $\NNN_q: C' \rightarrow C$ of $C$ at a node $q$.}
\label{fig:norm_anchors}
\end{figure}

Suppose that $C$ is $d$-stable and $r$ divides $d$.
An \emph{$r$-spin} structure on $C$ is an orbifold line bundle $L\to C$, called the \emph{$r$-spin bundle}, together with an isomorphism
\begin{equation}\label{defn: iso tau L}
\tau:L^{\otimes r}\to\omega_{C,\log}.
\end{equation}
The local structure of $L$ at a marked point $p$ looks like the quotient stack $[\C^2/\mu_d]$ where the generator $\zeta=e^{\frac{2\pi i}{d}}\in\mu_d$ acts by
\begin{equation}\label{eqn: local action}
\zeta\cdot(X,T)=(e^{\frac{2\pi i}{d}} X,e^{m\frac{2\pi i}{d}} T),
\end{equation}
where $X$ is a local coordinate on the curve and $T$ a coordinate on the
fibres of $L$. We define
the \emph{multiplicity} (of $L$) to be the integer $\mathrm{mult}_p(L) \in \{0,\dots, d-1\}$ such that $m=\mathrm{mult}_p(L)$ modulo $d$. Moreover, since the multiplicity of $\omega_{C,\log}$ at any point is $0,$ for each marking $p$ it must hold that
\[{\text{mult}_p(L)\over d/r}\in\Z. \]

An \emph{extended twisted $r$-spin} structure on a closed marked genus $g$ orbifold Riemann surface $C$ is an orbifold bundle of the form
\begin{equation}\label{eq:S}S = L \otimes \O\left(-\sum_{i \in I_0} [z_i] \right),\end{equation}
where $L$ is an $r$-spin bundle and $I_0$ is a subset of the markings of $L$-multiplicity zero, which consists of all such markings, except, possibly, one marking per genus $0$ connected component.\footnote{Such a marking must be what we call an \emph{anchor} in \cite{GKT2}, but we will avoid discussing anchors to focus on the main thrust of the description of our objects.} If $I_0$ consists of all markings of multiplicity $0$ we omit the word ``extended''. The set $I_0$
is a choice and part of the data of the extended $r$-spin structure.
Here, the twist
$\O([z_i])$ is the orbifold line bundle associated to the degree $1$ divisor
$[z_i]$, i.e., $\O([z_i])$ is the pull-back of a line bundle on the
coarse moduli space.

The \emph{twist} of the $i^{th}$ marking (with respect to $S$) is defined to be \[l_i-1+r\mathbf{1}_{i\in I_0}\] where  $l_i$ is the unique solution of\[\text{mult}_{z_i}(L) = \left(\frac{d}{r}\right)l_i,~~l_i\in\{0,\ldots,r-1\}\]
and $\mathbf{1}_{i\in I_0}$ denotes the characteristic function of
$I_0$. 

Using~\eqref{defn: iso tau L}, note that there is an isomorphism $\tau$ so that
\begin{equation}\label{defn: iso tau S}
\tau: S^{\otimes r}\simeq\omega_{C,\log}(-\sum_{i\in I_0}r[z_i]).
\end{equation}

\begin{definition}
A \emph{closed (extended) twisted genus $g$ $r$-spin surface} is a closed genus $g$
$d$-stable orbifold Riemann surface together with a(n extended) twisted $r$-spin structure.
The notion of isomorphism is the standard one.
\end{definition}

Given the normalization $\nu:\widehat C\rightarrow C$,
the orbifold line bundle $\NNN^*L\to\widehat{C}$ is automatically an $r$-spin bundle. However,  $\NNN^*S$ need not be a twisted $r$-spin bundle. Instead,
if we take
\begin{equation}
\label{eq:S Shat def}
\widehat{S}=\NNN^*S\otimes\O\left(-\sum_{q\in\mathcal{R}} [q]\right),
\end{equation}
where
$\mathcal{R}:=\{\hbox{$q$ a half-node of $\widehat{C}$ of
multiplicity $0$}\},$\footnote{The definition here slightly differs from that of \cite{GKT2,BCT2}, since we do not need here the notion of \emph{anchor}, which was used in some constructions and proofs in these papers.}
then $\widehat{S}$ is a twisted $r$-spin bundle, with the set of
marked points where the spin structure is twisted being
\[
\widehat{I}_0=I_0\cup \mathcal{R}.
\]
The \emph{twists} of the half-nodes of $\widehat{C}$ are defined to be the twists of
$\widehat{S}$ at these points. This definition of twist descends to $C$ and yields a notion for the twist of a half-node when
viewed as a nodal point in $C$ along with a choice
of branch of $C$ through this node.

\begin{definition}
A special point whose twist is $-1\mod r$ is called \emph{Ramond}, while the other markings are called \emph{Neveu-Schwarz} (with respect to $S$).
\end{definition}

Consider the coarsification map $\pi: C \rightarrow |C|$ and the pushforward $|L| = \pi_*L$. Here, we can see that locally at the marked point $p$ with $\mathrm{mult}_pL = m$, the sections of this sheaf will be given by $\mu_d$-invariant sections of $\O_{\mathbb{C}}$ with respect to the action in Equation~\eqref{eqn: local action}. The generator acts locally on a section $f$ by taking $\zeta\cdot(X, f(X)) = (\zeta X, \zeta^m f(X))$. For invariance, we need that $\zeta^m f(X) = f(\zeta X)$, which implies that $f(X) = X^m g(X^d)$ for some polynomial $g$. 
We then have the following proposition, see e.g., \cite[Proposition 1.6]{GKT2}:
\begin{prop}\label{obs:twists_on_coarse}
For any connected component $C_l$ of $|\widehat{C}|,$ with markings also denoted by $\{z_i\}_{i\in [l]}$ and half-nodes $\{p_h\}_{h\in N},$ it holds that $|\widehat{S}|$ is a line bundle and
\begin{equation}\label{eq:nodal_curve_closed}\left(|\widehat{S}|\big|_{C_l}\right)^{\otimes r} \cong \omega_{|C_{l}|} \otimes  \O\left(-\sum_{i \in [l]} a_i [z_i] - \sum_{h \in N} c_h [p_h]\right),\end{equation}
\[a_i,c_h\in \{-1,0,\ldots, r-1\},\]
where $a_i,c_h$ are the twists.
\end{prop}
The properties of the twists are summarized in the following well-known observation:
\begin{obs}\label{obs:closed_constraints}
It follows immediately from
Proposition~\ref{obs:twists_on_coarse} that if $C$ is a smooth genus $g$ $r$-spin curve,
then \begin{equation}\label{eq:close_rank1_general}
\frac{\sum a_i +(g-1)(r-2)}{r}\in \Z,
\end{equation}
where $a_i$ are the twists. The same holds for every connected component of the normalized $|\widehat{S}|.$

If $C$ is a stable $r$-spin curve and
$p$ and $p'$ are the two branches of a node, then
\begin{equation}\label{eq:twists_at_half_nodes}c_p + c_{p'} \equiv r-2 \pmod r.\end{equation}
\end{obs}

We now generalize to the $W$-spin case.
Let $W$ be any Fermat polynomial in $n$ variables,
\begin{equation}
\label{eq:Fermat W}
W(x_1,\ldots, x_a)=\sum_{i=1}^a x_i^{r_i},
\end{equation}
and let $d={\mathrm{lcm}}(r_1,\ldots,r_a)$.
Write $G^{\max}=\mu_{r_1}\times\cdots\times\mu_{r_a},$ and let $G^{\min}\simeq \mu_d$ be the subgroup of $G$ generated by the exponential grading operator $(1,\ldots,1)$, as defined in \cite{FJR}. We call $G^{\max}$ the maximal symmetry group of $W,$ and $G^{\min}$ the minimal admissible symmetry group (see \cite[Definition 2.3.2]{FJR} for a definition of admissible).

\begin{definition}
Let $W$ be as above, and $G^{\min}\leq G\leq G^{\max}.$
A \emph{closed genus $g$ $(W,G)$-spin Riemann surface} is a tuple
\[
(C,S_1,\ldots, S_a;\tau_1,\ldots,\tau_a)
\]
where (i) $C$ is $d$-stable marked genus $g$ orbifold Riemann surface, (ii) for all $i\in[a]:=
\{1,\ldots,a\}$, $(S_i,\tau_i)$ is a graded twisted $r_i$-spin structure on $C$, and (iii) for each marked point the multiplicity vector belongs to $G.$
\end{definition}
We similarly define the associated twisted $(W,G)$-spin structure, and extended $(W,G)$-spin structures.

The \emph{twist} of a marking or a half-node in an extended twisted $(W,G)$-spin structure is the $a$-tuple of its twists with respect to each extended twisted spin bundle. We write $\tw(q)$ for the twist of $q,$ and $\tw_i(q)$ for its $i^{th}$ component, $i\in[a]$. The default case is that the extended twisted $(W,G)$-spin structure is just the associated twisted spin structure. A special point $q$ is \emph{broad} if it is Ramond with respect to one of the bundles $S_i,$ otherwise $q$ is \emph{narrow}.
\subsection{Moduli space and intersection theories}
Fan, Jarvis, and Ruan constructed a moduli space $\M_{g,n}^{\text{FJR}, (W,G)}$ consisting of compact stable $W$-spin orbicurves, for which they provide an enumerative
theory. The following theorem is essentially carrying out the proof of Theorem 2.2.6 of \cite{FJR} in our context of
$W=\sum_{i=1}^n x_i^{r_i},~G^{\min}\leq G\leq G^{\max}$, with a minor difference which arises because of different
automorphism groups, as we explain below.
\begin{thm}\label{thm:moduli_closed}
The moduli space $\M_{g,l}^{(W,G)}$ of $(W,G)$-spin curves is a smooth Deligne-Mumford stack with projective coarse moduli. The moduli space $\M_{g,l}^{(W,G)}$ has a universal curve $\mathcal{C}^{(W,G)}_{g,l}$ and universal spin bundles $\CS_1,\ldots,\CS_a$. Moreover, if we let $\operatorname{st}: \M_{g,l}^{(W,G)}\rightarrow \M_{g,l}$ be the morphism given by forgetting the spin bundles and orbifold structure, then $\operatorname{st}$ is a flat, proper, and quasi-finite (but not representable) morphism.

\end{thm}

We point out the differences from our moduli space and the one outlined in \cite{FJR}. First, with our definitions, stabilizer groups of special points are always $\mu_d$, while Fan, Jarvis and Ruan choose a stabilizer group $\mu_e$ at a special point $x$
so that the representation of $\mu_e$ on the direct sum of the fibres of the spin bundles at $x$ is faithful. This changes the universal curve over $\M_{g,n}^W$. In addition, it changes automorphism groups of nodal curves which lie at the boundary of the moduli space, so that our moduli space is a root stack over the moduli space of \cite{FJR}.

The moduli space $\M_{g,n}^{(W,G)}$ decomposes into open and closed substacks $$\M_{g, \{(b_{11},\ldots,b_{1a}) \ldots, (b_{l1},\ldots,b_{la})\}}^{(W,G)},
$$ where the tuples $\{(b_{11},\ldots,b_{1a}) \ldots, (b_{l1},\ldots,b_{la})\}$ specify the twists of markings in the twisted spin structures. These connected components are nonempty precisely if Equation~\eqref{eq:close_rank1_general} holds for all $j\in [a]$.
\begin{rmk}
The moduli space for a given $G$ is an open and closed substack of that for $G^{\max},$ consisting of some of the latter's moduli space connected components.

When $g=0$, the map which forgets the spin structure induces an isomorphism between the coarsification of each connected component of $\M_{g,l}^{(W,G)}$ and $\M_{0,l}.$

The moduli of closed extended $g=0$ curves has additional connected components, thanks to the additional possible twists, but Theorem~\ref{thm:moduli_closed} extends naturally to include them.
\end{rmk}
\subsubsection{Cohomological field theory and the FJRW class}
\cite{FJR3,polishchuk2011matrix} construct a cohomology class $c_{W,G}$, called the \emph{FJRW class} in $H^*(\M_{g,l}^{(G,W)})$ and prove that it satisfies the axioms of a \emph{cohomological field theory (CohFT)} in the sense of Kontsevich-Manin \cite{KontsManin}. In the special $r$-spin case, $(W,G)=(x^r,\mu_r)$, these properties were shown, using different methods and in different scenarios, in \cite{Witten93,PV,Moc06,ChiodoStable,CLL,FJR}. 
In the case of extended twisted $g=0$ theories, \cite{JKV2,GKT2,Nill,BCT_Closed_Extended} constructed an analogous class and studied its properties.

A few important properties of the FJRW class are the following (see \cite[Theorem 1.2.5]{FJR3}, restated in our language):
\begin{itemize}
    \item {\bf Degree.} The degree of the FJRW class is given by the formula
    \[\deg(c_{(W,G)})=
    \sum_{j=1}^a\frac{\sum_{i=1}^l b_{ji}+(g-1)(r_j-1)}{r_j}.\]
    \item {\bf Concavity.} If $\pi_*(\bigoplus_{i\in[a]}\CS_i)=0$ on some connected component which has only narrow insertions, then $R^1\pi_*(\bigoplus_{i\in[a]}\CS_i)$ is a vector bundle, called the \emph{Witten bundle}, and  $c_{(W,G)}$ equals the 
    top Chern class $c^{\op{top}}(R^1\pi_*(\bigoplus_{i\in[a]}\CS_i))$. 
    \item {\bf Index $\mathbf{0}$.} If $\deg(c_{W,G})=0$ on some connected component, and both $$\pi_*(\oplus_{i\in[a]}\CS_i),~R^1\pi_*(\oplus_{i\in[a]}\CS_i)$$ are vector bundles, then $c_{W,G}$ is an explicit number.
\end{itemize}
For non-extended theories there is also the \emph{Ramond vanishing property}, stating that if some twist $b_{ji}=r_j-1$ then $c_{(W,G)}=0$ (see, e.g., \cite[Axiom 4]{JKV}).

Another important characteristic class of the moduli space are the so-called \emph{$\psi$-classes}. For any marked point $z_i$, $i=1,\ldots,l$, we define 
$$ \psi_i := c_1(\CL_i)\in H^2(\M_{g,l}^{(G,W)}),$$ where $\CL_i\to\M_{g,l}^{(G,W)}$ is the relative cotangent line bundle whose fiber at $(C,S_1,\ldots,S_l)$ is $T^*_{z_i}C.$

\begin{definition}\label{def: FJRW int num closed} The \emph{FJRW intersection numbers} are defined via the formula
\[\langle\prod_{i=1}^l\tau_{d_i}^{\vec{b}_i}\rangle^{(W,G),c}_g=C_{G,g}\int_{\M^{(W,G)}_{g,\{\vec{b}_1,\ldots,\vec{b}_l\}}}c_{(W,G)}\prod_{i=1}^l\psi_i^{d_i},\]
where $\vec{b}_i=(b_{1i},\ldots,b_{1a})$ are the vectors of twists, ${\M^{(W,G)}_{g,\{\vec{b}_1,\ldots,\vec{b}_l\}}}$ is the component of $\M_{g,l}^{(W,G)}$ corresponding to orbicurves with these twists, and $C_{G,g}=\frac{|G|^g}{d},$ with $d$ the degree of the map ${\M^{(W,G)}_{g,\{\vec{b}_1,\ldots,\vec{b}_l\}}}\to{\M_{g,l}}$ that forgets the spin structure. This constant's origin is \cite[Definitions 4.2.1, 4.2.6]{FJR}.  The intersection numbers are defined to be $0$ if the cohomological rank of the integrand differs from the dimension of ${\M^{(W,G)}_{g,\{\vec{b}_1,\ldots,\vec{b}_l\}}}$.
\end{definition}

\begin{rmk}
   In Definition~\ref{def: FJRW int num closed}, the superscript ``c'' stands for ``closed'' rather than the open theories we shall consider later on. 
\end{rmk}

\begin{definition}
    The \emph{genus $g$ FJRW potential} $$F^{(W,G),c}_g= \sum_{l\geq 0}\frac{1}{l!}{\sum_{\substack{\vec{b}_1,\ldots,\vec{b}_l}\\d_1,\ldots,d_l}}\langle\prod_{i=1}^l\tau_{d_i}^{\vec{b}_i}\rangle^{(W,G),c}_g,$$where $t^{\vec{b}}_d$ are formal variables, is defined to be the generating function of the genus $g$ open FJRW intersection numbers. We write $$F^{(W,G),c}=\sum_{g}u^{2g-2}F^{(W,G),c}_g$$ for \emph{the all genus FJRW potential}.
\end{definition}

\subsection{Relations with other theories and applications}
We refer the reader to \cite{fjrw_survey} for a more detailed survey of the subject, and more references.
\subsubsection{Mirror symmetry} \label{subsubsec:mirror symmetry}
There has been a great deal of work carried out for LG/LG mirror symmetry. It originates from work of Berglund, Henningson, and H\"ubsch \cite{BerglundHubsch, BerglundHenningson} proposing certain pairs of Landau-Ginzburg models that would mirror each other. This construction was completed by Krawitz, who also provided a mirror symmetry theorem on the level of state spaces \cite{Krawitz}. After the enumerative geometry for LG models was defined by Fan, Jarvis and Ruan, work of Li, Li, Saito and Shen on the enumerative B-model made proving closed mirror symmetry accessible \cite{polyvector, LiLiSaitoShen}, which was proven for closed enumerative invariants by He, Li, Shen, and Webb \cite{HeLiShenWebb}.  

The most relevant work for us is that of
\cite{LiLiSaitoShen} and \cite{HeLiShenWebb}, and in particular
the mirror correspondence
\[
\left(W=\sum_{i=1}^a x_i^{r_i}, G=\prod_i\mu_{r_i}\right)\longleftrightarrow 
\left(W=\sum_{i=1}^a x_i^{r_i}, G=\{e\}\right).
\]
In mirror symmetry, there are two sides, usually called the ``A-model,'' 
and the ``B-model.'' The A-model concerns 
enumerative information, in this case of the Landau-Ginzburg model on
the left. The B-model, which typically
involves period integrals, in this case variants of so-called
oscillatory integrals of the form
\[
\int_{\Gamma} e^{W/\hbar} f(x_1,\ldots,x_a)dx_1\wedge\cdots\wedge dx_a
\]
and similar integrals involving perturbations of $W$. Here $\hbar$ is
an auxilliary variable, $f$ is some holomorphic function on 
$\C^a$, and $\Gamma$ is a cycle on $\C^a$ which may
be unbounded, but such that $\re(W/\hbar)\rightarrow -\infty$ in the unbounded
directions.

A precise description of the $B$-model in this situation
in terms of Saito-Givental theory was given in \cite{LiLiSaitoShen,HeLiShenWebb}. 
Their result demonstrates that one can extract closed FJRW invariants
of the Landau-Ginzburg model with maximal symmetry group by constructing
a Frobenius manifold structure on the universal unfolding of the potential
$W$. In practice, this is a quite difficult procedure to carry out, and
\cite{HeLiShenWebb} only gives the existence of an abstract isomorphism between
$A$-model and $B$-model Frobenius manifolds. This allows for the identification
of closed FJRW $A$- and $B$-model genus $0$ invariants, and
hence the identification of invariants for all genera thanks to 
the reconstruction theorem of \cite{Teleman}. We return to this picture
in \S\ref{marks mirror symmetry section}, where an open version of 
FJRW theory allows an 
effective, explicit proof of this mirror result in the case that
$W=x^r+y^s$.

\subsubsection{Integrable hierarchies}
\label{sub:int_hierarchies_closed}
FJRW theories are also closely related to \emph{integrable hierarchies}, which are certain families of commuting non-linear PDEs~\cite{Witten2DGravity,Witten93,FSZ10,fjrw_and_drinfeld_sokolov,FJR}.
To illustrate this relationship, we now restrict to the $r$-spin case, that is $W=x^r,~r\geq 2$ and $G=\mu_r.$ In this case the relevant hierarchy is the $r$-KdV hierarchy, also known as the Gelfand-Dikii hierarchy, which we now describe.

A {\it pseudo-differential operator}~$O$ is a Laurent series
\[
O=\sum_{n=-\infty}^m o_n(T_1,T_2,\ldots,u)\partial_x^n,
\]
where $m$ is an integer,  $T_i$ for $i\geq 1,$  $\partial_x,u$ are formal variables, and $o_n(T_1,T_2,\ldots,\lambda)\in\C[u^{\pm1}][[T_1,T_2,\ldots]]$.     The space of pseudo-differential operators carries a natural structure of a non-commutative associative algebra, with multiplication $\circ$  induced from the usual product on $\C[u^{\pm1}][[T_1,T_2,\ldots]]$ and from 
\[
\d_x^n\circ f:=\sum_{l=0}^\infty\frac{n(n-1)\ldots(n-l+1)}{l!}\frac{\d^lf}{\d x^l}\d_x^{n-l},
\]
where $n\in\mathbb{Z}$, $f\in\C[u^{\pm1}][[T_*]]$, and $x$ is identified with $T_1$.

We now fix an integer $r\geq 2$.
Let $O$ be a pseudo-differential operator of the form
\[
O=\partial_x^r+\sum_{n=1}^\infty o_n\partial_x^{r-n}.\]
Then $O$ has a unique \emph{$r^{\text{th}}$ root}: a pseudo-differential operator $O^{\frac{1}{r}}$ of the form
\[
O^{\frac{1}{r}}=\partial_x+\sum_{n=0}^\infty p_n\partial_x^{-n}
\]
satisfying $\left(O^{\frac{1}{r}}\right)^r=O$.

Consider an operator $L$ of the form
$$
L:=\partial_x^r+\sum_{i=0}^{r-2}l_i\partial_x^i,\quad l_i\in\C[u^{\pm1}][[T_*]].
$$
If we write $(A)_+$ for the non-negative part of a pseudo-differential operator $A$, meaning the part composed of the non-negative powers of $\partial_x$, then it is straightforward to verify, for all $n\geq 1$, that the commutator $[(L^{n/r})_+,L]$ is of the form $\sum_{i=0}^{r-2}h_i\partial_x^i$ with $h_i\in\C[u^{\pm1}][[T_*]]$. The  $r$-KdV hierarchy is the system of PDEs for $\{l_j\}_{j=0}^{r-2}$ given by:
\begin{equation}
\frac{\partial L}{\partial T_n}=u^{n-1}[(L^{n/r})_+,L],\quad n\geq 1.
\end{equation}

Let $L$ be the solution of the system specified by the initial condition
\begin{equation}
L|_{T_{\geq 2}=0}=\partial_x^r+ru^{-r}x,
\end{equation}
 Witten's $r$-spin conjecture, proven by Faber-Shadrin-Zvonkine \cite{FSZ10} reads
\begin{thm}[Faber-Shadrin-Zvonkine]
    Define a change of  variables from the variables $t_d^a$ of Witten's $r$-spin potential $F^{\frac{1}{r},c}$ to the variables $T_i$ of  the algebra of pseudo-differential operators by
    \begin{gather}
T_m=\frac{1}{(-r)^{\frac{3m}{2(r+1)}-\frac{1}{2}-d}m!_r}t^a_d,\quad 0\leq a\leq r-2,\quad d\geq 0,
\end{gather}
where $m=a+1+rd$ and  $m!_r:=\prod_{i=0}^d(a+1+ri).
$
Then \emph{as a function of $\{T_i\},~i\geq 1$} the $r$-spin potential $F^{(x^r,\mu_r),c}$ satisfies the following:
\begin{itemize}
    \item If $m\ge 1$ is not divisible by $r$, then the coefficient of $\partial_x^{-1}$ in $L^{n/r}$ is
$$
\res L^{m/r}=\lambda^{1-m}\frac{\partial^2 F^{\frac{1}{r},c}}{\partial T_1\partial T_m}.$$
\item If $r$ divides $m$, then $\frac{\partial F^{\frac{1}{r},c}}{\partial T_m}=0$.
\end{itemize}
\end{thm}
The second property in the theorem is called the \emph{Ramond vanishing property} of the $r$-spin potential, and implies that intersection numbers which contain Ramond insertions vanish. The theorem implies that $\exp(F^{(x^r,\mu_r),c})$ becomes, after the above change of variables, a \emph{tau-function} of the $r$-KdV hierarchy.
The case $r=2$ is equivalent Witten's original KdV conjecture \cite{Witten2DGravity}, proven by Kontsevich \cite{Kontsevich}. The equivalence uses the Index Zero axiom on FJRW classes, mentioned above. \cite{fjrw_and_drinfeld_sokolov} have extended this result to more general integrable hierarchies.

\subsubsection{The LG/CY correspondence}
Take $W$ as in \eqref{eq:Fermat W} and $G$ a subgroup of $G^{\max}$ that contains $G^{\min}$ (i.e., is admissible). The polynomial $W$ is quasihomogeneous of degree $d=\mathrm{lcm}(r_1,\ldots,r_a)$, and defines a hypersurface in the $(n-1)$-dimensional (Gorenstein) weighted projective space $\mathbb{P}(q_1, \dots, q_n)$, where $q_i = \tfrac{d}{r_i}$. When $\sum_i \tfrac{1}{r_i} = 1$, then by adjunction the hypersurface $X_W:= Z(W) \subseteq \mathbb{P}(q_1, \dots, q_n)$ is a Calabi-Yau (CY) orbifold. 

One can also take finite quotients of $Z(W)$ and obtain other Calabi-Yau orbifolds. There is a subgroup $SL(W) \leq G^{\max}$ consisting of elements $(k_1, \dots, k_n)$ so that $\sum_i \tfrac{k_i}{r_i} \in \Z$, which corresponds to any $n$-tuple of roots of unity whose product is $1$. Suppose $G^{\min} \leq G \leq SL(W)$ and take the quotient group $\bar G:= G / G^{\min}$. Then one sees that $[X_W / \bar G]$ is also a Calabi-Yau orbifold. We remark that one must consider $\bar G$ instead of $G$ as  $G^{\min}$ is the intersection of $G^{\max}$ with the torus $\C^*$ found when one writes the weighted-projective space as a quotient $\mathbb{P}(q_1, \dots, q_n) = (\C^{n} \setminus\{0\}) / \C^*$. 

The \emph{Landau-Ginzburg/Calabi-Yau (LG/CY)} correspondence predicts that whenever the orbifold $[X_W/\hat{G}]$ is CY, then its GW theory agrees with the FJRW theory of $(W,G)$ after analytic continuation and a symplectic transformation. This correspondence was first proven for the $g=0$ Fermat quintic threefold with minimal symmetry group in \cite{ChiodoRuan}. Various instances, special cases and generalizations have appeared since, see \cite{acosta2014asymptotic,chiodo2011lg,clader2017landau,clader2018sigma,krawitz2011landau,lee2014mirror,milanov2011gromov,milanov2016global,priddis2014proof,priddis2016landau,zhao2022landau}.

\subsubsection{$3$-spin structures and tautological relations on the moduli of curves} We remark that $r$-spin theory has also enjoyed applications to understanding the geometry of $\M_{g,n}$.
The beautiful work \cite{PPZ} applies the Teleman-Givental classification of semi-simple CohFTs \cite{givental2001semisimple,givental2001gromov,Teleman} to the $3$-spin CohFT to prove that Pixton's relations \cite{pixton2012conjectural} hold in the tautological ring of $\M_{g,n}.$ They also obtain an expression for the Witten class (pushed down to the moduli of curves).

\section{Open FJRW theory}
Open FJRW (OFJRW) theory  aims to construct and study the analogue of FJRW theory, but for moduli spaces of (orbifold) Riemann surfaces with boundary.
Developing such an open extension is challenging for several reasons:
\begin{itemize}
\item \textbf{What are the objects?} As seen above, there are many crucial objects at play to define (closed) FJRW intersection numbers. In particular, we need to define the open analogues to $W$-spin curves, the moduli space, the FJRW bundles, $\psi$-classes, and finally intersection numbers. In Definition~\ref{def:open_W_graded}, we start with the definition of a $W$-spin Riemann surface with boundary. Once such an object is defined, new obstacles regarding their moduli spaces and corresponding FJRW bundles present themselves, as seen below.  

\item \textbf{The existence of boundaries.} Unlike its closed cousin, the moduli space of $W$-spin Riemann surfaces with boundary has \emph{real codimension $1$} boundaries, and higher codimension corners. This raises a fundamental difficulty --- while in the closed setting the intersection theory involved intersecting cohomology classes, in the open setting one cannot work in the level of cohomology. Equivalently, defining integration and intersection numbers must involve imposing some boundary conditions, and the resulting theory will depend on the choice of the boundary conditions. A considerable portion of the geometry in OFJRW theory lies in finding geometrically meaningful boundary conditions, and studying \emph{wall crossings} --- the effect of changes in boundary conditions.

\item \textbf{Orientation.} (Closed) FJRW moduli spaces, and the orbifold vector bundles which induce the cohomology classes used in these theories are all \emph{complex} objects, hence possess canonical orientations. It is not the case for their open analogues, which are \emph{real} objects by nature. In order to define the relative Euler class and invariants, one thus has to overcome orientation problems. 

\item \textbf{Computability.} The above features of the open theories make computing their correlators more involved than in the closed case. The existence of wall crossings complicates computations further, and raises additional questions like which invariants are unaffected by wall crossings, and how the collections of correlators behaves under wall crossing transformations.
\item \textbf{Missing foundations.} One of the most important developments in GW and FJRW theory was the development of the virtual fundamental class (see, e.g.,
 \cite{behrend1997intrinsic} for the GW construction in algebraic geometry, and \cite{polishchuk2011matrix,FJR3,Moc06,CLL} for FJRW theory), which is responsible for major advances in the field. Unfortunately, at the moment there is no analogous construction in the open setting. This makes rigorous definitions of theories more convoluted, and, in the case of FJRW theory, generally restricted to concave theories and narrow insertions, or mild weakenings of these conditions.
\item \textbf{Scarcity of physics predictions.} Finally, unlike the closed setting, and partially because of the difficulties listed above, there are only few predictions from physics regarding the potential existence and properties of FJRW theory.
See, for example, \cite{Walcher,Melissa} which concern properties of the Fermat quintic FJRW and \cite{Hori,Horiprivate} regarding possible boundary twists for open $r$-spin theories (see \textsection\ref{subsec: point insertion}).
\end{itemize} 
Despite the above impediments, open FJRW theories have been constructed and analyzed in several important cases, as we detail below.

\subsection{Graded $W$-spin surfaces}
\begin{definition}\label{def: open RS bdry}
Fix $d\in\mathbb{N}.$ A {\it (connected)  $d$-stable marked genus $g$ orbifold Riemann surface with boundary} is a tuple
\[(C, \phi, \Sigma, \{z_i\}_{i \in [l]}, \{x_j\}_{j \in [k]})\]
in which:
\begin{enumerate}[(i)]
\item $C$ is a (possibly connected, nodal) $d$-stable genus $g$ orbifold Riemann surface.
\item $\phi: C \rightarrow C$ is an anti-holomorphic involution ({\it conjugation}), such that the induced action on the set of connected components of $C\setminus C^\phi$ is fixed-point free. We also write $\phi:|C|\rightarrow |C|$ for this induced
conjugation.
\item $\Sigma$ is a fundamental domain for the induced conjugation
on $|C|$. Here $\Sigma$ is a Riemann surface, usually, but not
always, with boundary.
\item The $z_i \in C$ are a collection of distinct points
 (the {\it internal marked points}) labeled by the set $I$, whose images in $|C|$ lie in $\Sigma \setminus \d\Sigma$, with {\it conjugate marked points} $\bar{z}_i:= \phi(z_i)$.
\item The $x_j \in \text{Fix}(\phi)$ are a collection of distinct points (the {\it boundary marked points}) labeled by the set $B$, whose images in $|C|$ lie in $\d\Sigma$.
\end{enumerate}
 We will often denote the object only by $C$ or even by $\Sigma$. When $C$ is connected and stable with genus $g=0$ we refer to $C$ as a stable marked disk.  Figure \ref{fig:smooth_orbRS_boundary} provides an example of such a disk. Here, in the genus $g=0$ and smooth case, one can think of the boundary as $\R\cup\{\infty\}$ and interior of $\Sigma$ as the upper-half plane of $\C$. 
\begin{figure}
  \centering

\begin{tikzpicture}[scale=0.8]
  \draw (0,0) circle (2cm);
  \draw (-2,0) arc (180:360:2 and 0.6);
  \draw[dashed] (2,0) arc (0:180:2 and 0.6);

 \node at (2, 2) {$C$};
 \node at (2.35,0) {$\partial\Sigma$};
\node (c) at (-1,.95) {$\bullet$};
\node[right] at (-1,.95) {$z_1$};

\node at (-1,-1.25) {$\bullet$};
\node[right] at (-1,-1.25) {$\bar z_1$};

\node at (.25, -.6) {$\bullet$};
\node[above] at (.25,-.6) {$x_1$};

\node (a) at (.5, 1.2) { $\Sigma$};
\node at (0,2) {$\bullet$};
\node[above] at (0,2) {$z_2$};
\node (c) at (0, -2) {$\bullet$};
\node[below] at (0,-2) {$\bar z_2$};
\end{tikzpicture}

  \caption{A smooth connected marked genus 0 orbifold Riemann surface with boundary with two internal markings $z_1, z_2$ and a boundary marking $x_1$.}
  \label{fig:smooth_orbRS_boundary}
\end{figure}
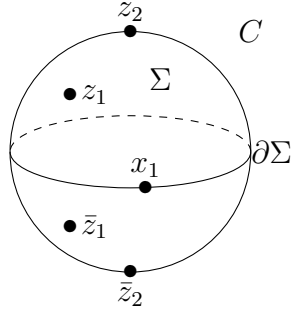
An isolated $\phi$-invariant point is called a \emph{contracted boundary node}. Those $\phi$-invariant nodes which are not a contracted boundary are called \emph{boundary nodes}. The remaining nodes whose image under the coarsification map falls in the interior of $\Sigma$ are the \emph{internal nodes}. Their conjugate points are the \emph{conjugate nodes}. We call $\partial\Sigma$ the \emph{boundary of the open surface}. It is oriented as the boundary of $\Sigma$. 
The surface is \emph{smooth} if all connected components are smooth. The stability assumption implies that each connected component of the normalization has finite automorphism group. The collection of marked points, contracted boundaries, nodes and their conjugates are collectively called \emph{special points}.

Two $d$-stable marked genus $g$ orbifold Riemann surfaces with boundary are isomorphic if there exists an isomorphism between them which preserves all the extra data.
\end{definition}

Note that one can write
\[|C| = \left(\Sigma \coprod_{\d \Sigma} \overline{\Sigma}\right)/\sim_{CB},\]
where $\bar{\Sigma} = \phi(\Sigma)$ is obtained from $\Sigma$ by reversing the complex structure, and $\sim_{CB}$ is the equivalence relation on $\Sigma \cup_{\d \Sigma} \overline{\Sigma}$ defined by $x\sim_{CB}y$ precisely if $x=\phi(y)$ and one of them is a contracted boundary node.

A node of a nodal marked surface can be internal, boundary or contracted boundary, both internal and boundary nodes can be separating or nonseparating,\footnote{A node $n$ is separating if the result of partial normalization of $C$ at $n$ is disconnected. In $g=0$ all nodes are separating.} so there are five types of nodes, as illustrated in Figure~\ref{fig:node type} by shading $\Sigma \subseteq |C|$ in each case.  Note that ${\partial}\Sigma\subset\text{Fix}(|\phi|)$ is a union of cycles, and $\text{Fix}(|\phi|)\setminus{\partial}\Sigma$ is the union of the contracted boundaries.

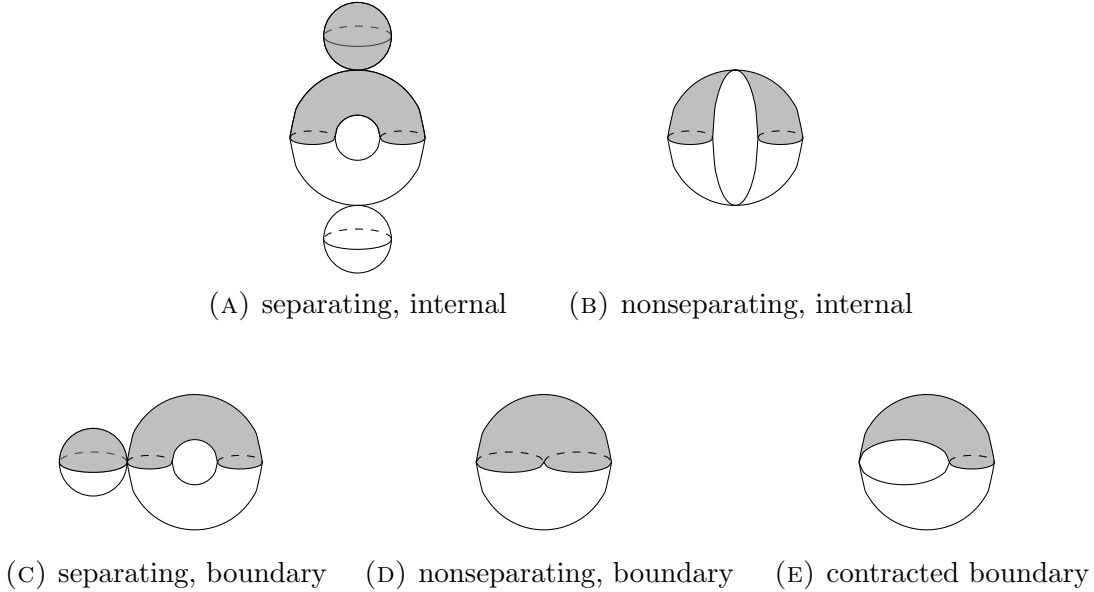
\begin{figure}[h]
\centering

  \begin{subfigure}{.3\textwidth}
  \centering

\begin{tikzpicture}[scale=0.3]
   \draw[ domain=-1:1,  smooth, variable=\x] plot ({\x}, {sqrt(1-\x*\x)});
   \draw[ domain=-3:3,   smooth, variable=\x] plot ({\x}, {sqrt(9-\x*\x)});

   \draw[ domain=-1:1, smooth, variable=\x] plot ({\x}, {-sqrt(1-\x*\x)});
   \draw[ domain=-3:3,  smooth, variable=\x] plot ({\x}, {-sqrt(9-\x*\x)});

    \draw [fill=gray,fill opacity=0.5] plot [domain=-1:1,  smooth]({\x}, {sqrt(1-\x*\x)}) --plot [domain=1:3,  smooth]({\x}, {-sqrt(1-(\x-2)*(\x-2))*0.3}) -- plot [domain=3:-3,  smooth ] ({\x}, {sqrt(9-\x*\x)})-- plot [domain=-3:-1,  smooth]({\x}, {-sqrt(1-(\x+2)*(\x+2))*0.3});

    \draw  [dashed, domain=1:3,  smooth] plot ({\x}, {sqrt(1-(\x-2)*(\x-2))*0.3});

    \draw  [dashed, domain=-3:-1,  smooth] plot ({\x}, {sqrt(1-(\x+2)*(\x+2))*0.3});

    \draw (0,4.5) circle (1.5);
    \draw (-1.5,4.5) arc (180:360:1.5 and 0.45);
    \draw[dashed] (1.5,4.5) arc (0:180:1.5 and 0.45);

     \draw (0,-4.5) circle (1.5);
    \draw (-1.5,-4.5) arc (180:360:1.5 and 0.45);
    \draw[dashed] (1.5,-4.5) arc (0:180:1.5 and 0.45);

    \draw [fill=gray,fill opacity=0.5] (0,4.5) circle (1.5);
   
\end{tikzpicture}

  \caption{separating, internal }
\end{subfigure}
  \begin{subfigure}{.3\textwidth}
  \centering

\begin{tikzpicture}[scale=0.3]

   \draw[ domain=-3:3,   smooth, variable=\x] plot ({\x}, {-sqrt(9-\x*\x)});
  \draw [fill=gray,fill opacity=0.5] plot [domain=-1:1,  smooth]({\x}, {sqrt(1-\x*\x)*3}) --plot [domain=1:3,  smooth]({\x}, {-sqrt(1-(\x-2)*(\x-2))*0.3}) -- plot [domain=3:-3,  smooth ] ({\x}, {sqrt(9-\x*\x)})-- plot [domain=-3:-1,  smooth]({\x}, {-sqrt(1-(\x+2)*(\x+2))*0.3});

  \draw  [dashed, domain=1:3,  smooth] plot ({\x}, {sqrt(1-(\x-2)*(\x-2))*0.3});

    \draw  [dashed, domain=-3:-1,  smooth] plot ({\x}, {sqrt(1-(\x+2)*(\x+2))*0.3});

    \draw[ domain=-1:1,   smooth, variable=\x] plot ({\x}, {-sqrt(1-\x*\x)*3});
\end{tikzpicture}
\vspace{0.9cm}
  \caption{nonseparating, internal }
\end{subfigure}

\begin{subfigure}{.3\textwidth}
  \centering
\vspace{0.9cm}
\begin{tikzpicture}[scale=0.3]
  \draw[ domain=-1:1, smooth, variable=\x] plot ({\x}, {-sqrt(1-\x*\x)});
   \draw[ domain=-3:3,  smooth, variable=\x] plot ({\x}, {-sqrt(9-\x*\x)});

    \draw [fill=gray,fill opacity=0.5] plot [domain=-1:1,  smooth]({\x}, {sqrt(1-\x*\x)}) --plot [domain=1:3,  smooth]({\x}, {-sqrt(1-(\x-2)*(\x-2))*0.3}) -- plot [domain=3:-3,  smooth ] ({\x}, {sqrt(9-\x*\x)})-- plot [domain=-3:-1,  smooth]({\x}, {-sqrt(1-(\x+2)*(\x+2))*0.3});

    \draw  [dashed, domain=1:3,  smooth] plot ({\x}, {sqrt(1-(\x-2)*(\x-2))*0.3});

    \draw  [dashed, domain=-3:-1,  smooth] plot ({\x}, {sqrt(1-(\x+2)*(\x+2))*0.3});

    \draw (-4.5,0) circle (1.5);
    \draw (-6,0) arc (180:360:1.5 and 0.45);
    \draw[dashed] (-3,0) arc (0:180:1.5 and 0.45);
    \draw[fill = gray, opacity = 0.5] (-6,0) arc (180:360:1.5 and 0.45) arc (0:180:1.5);
\end{tikzpicture}
\vspace{0.15cm}

  \caption{separating, boundary }
\end{subfigure}
\begin{subfigure}{.3\textwidth}
  \centering

\begin{tikzpicture}[scale=0.3]
\vspace{0.15cm}

   \draw[ domain=-3:3,  smooth, variable=\x] plot ({\x}, {-sqrt(9-\x*\x)});

    \draw [fill=gray,fill opacity=0.5] plot [domain=0:3,  smooth]({\x}, {-sqrt(2.25-(\x-1.5)*(\x-1.5))*0.3}) -- plot [domain=3:-3,  smooth ] ({\x}, {sqrt(9-\x*\x)})-- plot [domain=-3:0,  smooth]({\x}, {-sqrt(2.25-(\x+1.5)*(\x+1.5))*0.3});

    \draw  [dashed, domain=0:3,  smooth] plot ({\x}, {sqrt(2.25-(\x-1.5)*(\x-1.5))*0.3});

    \draw  [dashed, domain=-3:0,  smooth] plot ({\x}, {sqrt(2.25-(\x+1.5)*(\x+1.5))*0.3});
  
\end{tikzpicture}
\vspace{0.15cm}

  \caption{nonseparating, boundary}
\end{subfigure}
\begin{subfigure}{.3\textwidth}
  \centering

\begin{tikzpicture}[scale=0.3]
\vspace{0.15cm}
  \draw[ domain=-3:3,  smooth, variable=\x] plot ({\x}, {-sqrt(9-\x*\x)});

    \draw [fill=gray,fill opacity=0.5] plot [domain=1:3,  smooth]({\x}, {-sqrt(1-(\x-2)*(\x-2))*0.3}) -- plot [domain=3:-3,  smooth ] ({\x}, {sqrt(9-\x*\x)})-- plot [domain=-3:1,  smooth]({\x}, {sqrt(4-(\x+1)*(\x+1))*0.5});

    \draw  [dashed, domain=1:3,  smooth] plot ({\x}, {sqrt(1-(\x-2)*(\x-2))*0.3});

    \draw  [ domain=-3:1,  smooth] plot ({\x}, {-sqrt(4-(\x+1)*(\x+1))*0.5});
\end{tikzpicture}
\vspace{0.15cm}

  \caption{contracted boundary }
\end{subfigure}

\caption{The five types of nodes on a nodal marked surface.}
\label{fig:node type}
\end{figure}

\begin{definition}
Let $C$ be a $d$-stable marked genus $g$ orbifold Riemann surface with boundary, and $r\geq 2$ an integer which divides $d$. 
An \emph{$r$-spin} (resp. \emph{twisted $r$-spin}) structure on $C$ is an $r$-spin structure $L$ (resp. twisted $r$-spin structure $S$) on $C$ together with an involution  $\widetilde{\phi}$ on $L$ (resp. $S$) lifting
$\phi$ and such that $\widetilde{\phi}^{\otimes r}$ agrees under $\tau$
with the involution of $\omega_{C,\log}$ induced by $\phi$. We write $J=S^\vee\otimes\omega_C.$
\end{definition}

Twists of marked points and half-nodes relative to the twisted $r$-spin structure $S$ are defined as in the closed case. Twists at marked points and half-nodes are the same as the corresponding twists of their conjugates. As in the closed case, twists of the two branch points of a node satisfy \eqref{eq:twists_at_half_nodes}. We refer to the twist at an internal (boundary) marked point $z_i$ ($x_j$) as the \emph{internal (boundary) twist} $a_i~(b_j)$.

Degree considerations impose the following arithmetic constraint on the existence of an $r$-spin structure:
\begin{obs}
\label{obs:open_rank_general}
Let $S$ be a twisted $r$-spin structure on a connected marked genus $g$ orbifold Riemann surface with boundary, with internal markings $i\in [l]$ with twists $a_i$ and boundary markings $j \in [k]$ with twists $b_j$. Then we have a constraint on the twists given by the condition that 
\begin{equation}\label{eq:open_rank1_general}
\frac{2\sum_{i\in [l]} a_i + \sum_{j\in [k]} b_j+(g-1)(r-2)}{r}\in \Z.
\end{equation}
\end{obs}

If $n$ is a node of $C$ then the half-nodes also inherit twists, and they sum to $r-2$ modulo $r$, as in~\eqref{eq:twists_at_half_nodes}. Combined with \eqref{eq:open_rank1_general} this fixes the twists of \emph{separating} half-nodes.

\subsubsection{Lifting and alternations}
A completely new geometric ingredient which is foundational for open FJRW theories is the notion of \emph{liftings}. This is an additional structure, coming on top of the spin structures, which can be thought of as \emph{boundary conditions} for the spin structure on the boundary of the surface. The liftings can be thought of as picking a positive direction, or orienting, the real subspin bundle on the boundary minus the special points, and they play vital roles in defining a canonical relative orientation for the Witten bundle, and for defining boundary conditions for sections of the Witten bundle. A \emph{grading} is a lifting which satisfies a few additional properties at special points. At first reading the reader may skip this section, bearing in mind the idea of a lifting and grading as picking positive directions of the real spin line on the boundary.

\begin{definition}\label{def:lifting_compatible}
Let $S$ be a twisted $r$-spin structure on a connected marked genus $g$ orbifold Riemann surface with boundary $C$.

\begin{enumerate}

\item Let $A$ be the complement in $C^\phi$ of the set of special points. A \emph{lifting of $S$} is a choice (if it exists) of an orientation for the real line bundle
\[S^{\widetilde\phi}\to A\]
 satisfying the following property.  For any vector $v \in S^{\widetilde\phi}_a$, at any $a\in A,$ such that $v$ is positive with respect to the orientation specified by the choice of lifting, its image under the map $v \mapsto \tau(v^{\otimes r})$ is positive with respect to the natural orientation of $\omega_{C}^{\phi}|_{A}$. 

A {\it lifting of $J$ over $A$} is an orientation of the real line bundle
\[  J^{\widetilde\phi}\to A,\]
 satisfying the following property. If we take any vector $w\in J^{\widetilde\phi}_a$ at any $a\in A$ such that $w$ is positive with respect to the orientation, then there exists a vector $v\in S^{\widetilde\phi}_a$ with $\tau(v^{\otimes r})$ and $\langle w,v\rangle$ positive with respect to the natural orientation of $\omega_{C}^{\phi}|_{A}$, where $\langle-,-\rangle$ is the natural pairing between fibers of $S,J$ which takes values at the corresponding fiber of $\omega_C.$

 \item Suppose that $C$ has a contracted boundary Ramond node $q$. In this case a \emph{lifting of $S$} is a choice (if it exists) of an orientation for the real line $(S_q)^{\tilde{\phi}}$
 satisfying a different set of properties as follows. 
The fiber $\omega_{C,q}$ is canonically identified with $\C$ via the residue, and the involution $\phi$ is sent, under this identification, to the involution $z\to -\bar{z},$ whose fixed points are the purely imaginary numbers.\footnote{The residue of a conjugation invariant form $\zeta$ can be calculated as $\frac{1}{2\pi i} \oint_L\zeta $ where $L\subset\Sigma$ is a small loop surrounding $q$ whose orientation is such that $q$ is to the left of $L$. The behaviour under conjugation shows that the residue of an invariant section is purely imaginary.} 
The choice of orientation must satisfy the property that for any vector $v\in S_q^{\tilde\phi}$ positive with respect to the orientation, the image of $v^{\otimes r}$ under the map
\[
S^{\otimes r}_q \rightarrow \omega_{C,q}
\]
is \emph{positive imaginary}, meaning that it lies in $i\R_+$. We then write a lifting of $S$ at $q$ as an equivalence class $[v]$ of such positive elements $v$ under the equivalence relation of multiplication by a positive real number. In the contracted boundary case,
there always exists a $\widetilde{\phi}$-invariant $w \in J|_q$ such that $\langle w , v\rangle$ is positive imaginary, and we refer to such $w$ as a {\it grading at $q$}.
 \end{enumerate}

We say that a twisted $r$-spin structure $S$ on a marked genus $g$ orbifold Riemann surface with boundary $C$ is \emph{compatible} if each connected component of $C$ has a lifting, and every contracted boundary is Ramond. A \emph{lifting on $C$} is a choice of lifting on each connected component.
\end{definition}
\begin{rmk}
The grading at a contracted boundary, defined in Definition \ref{def:lifting_compatible}, is a limit case of 
a lifting over a boundary component with no special points which contracts to a single point. It can be shown that this limit must be a Ramond point.
\end{rmk}

Note that a lifting of $S\rightarrow C$ induces a lifting on the induced $r$-spin structure on the normalization of $C.$ 
\begin{definition}[Alternating nodes or marked boundary points; grading]
\label{def:alt}
Let $S$ be a compatible twisted $r$-spin structure on a connected marked genus $0$ orbifold Riemann surface with boundary. Suppose further that $C$ has no contracted boundary nodes. Take $q$ to be a boundary marked point, or boundary half-node of the normalization $\NNN:\widehat{C}\to C,$ and take $U_q$ to be some neighborhood of $|\NNN^{-1}(q)|$ in the boundary of the normalized surface. We say that the \emph{lifting alternates at $q$} and that $q$ is \emph{alternating} with respect to the lifting if the orientation \emph{cannot be extended} to an orientation of $|\widehat{J}|^{\widetilde\phi}\big|_{U_q}$. Otherwise the lifting \emph{does not alternate} or is \emph{non-alternating} and $q$ is \emph{non-alternating}.\footnote{
Alternating and non-alternating are termed \emph{legal} and \emph{illegal} resp. in \cite{BCT1}.}
We write $\alt(q)=1$ if the lifting alternates in $q$ and $\alt(q)=0$ if it does not alternate.

A \emph{grading} of a twisted $r$-spin structure on a connected marked genus $g$ orbifold Riemann surface with boundary, is a lifting with the additional property that for every Ramond boundary node both half-nodes are non-alternating, and otherwise exactly one half-node is alternating.  
\end{definition}

We summarize the properties of $r$-spin structures with a lifting in the following proposition. It was proven in \cite[Proposition 2.5]{BCT1} for $g=0,$ but the arguments generalize to $g>0.$

\begin{prop}\label{prop:graded_r_spin_prop}
Suppose that $(C,S)$ is a smooth connected genus $g$ twisted $r$-spin Riemann surface with boundary.
\begin{enumerate}
\item\label{it:compatibility_odd} When $r$ is odd, any twisted $r$-spin structure is compatible, and there is a unique choice of a lifting.
\item\label{it:compatibility_even} When $r$ is even, the boundary twists $b_j$ in a compatible twisted $r$-spin structure must be even.  Whenever the boundary twists are even, either the $r$-spin structure is compatible or it becomes compatible after replacing $\tilde\phi$ by $\xi\circ \tilde \phi \circ \xi^{-1}$ for $\xi$ an $r^{\text{th}}$ root of $-1$, which yields an isomorphich $r$-spin structure. 
\item\label{it:lifting and parity_odd} Suppose $r$ is odd and $v$ is a lifting. Then $x_j$ is alternating if and only if its twist is odd.
\item\label{it:lifting and parity_even}
Suppose $r$ is even.  If a lifting alternates precisely at a subset $D \subset \{x_j\}_{j \in [k]}$, then
\begin{equation}
\label{eq:parity}
\frac{2\sum a_i + \sum b_j+(g-1)(r-2)}{r} \equiv |D| +g+1\mod 2.
\end{equation} If~\eqref{eq:parity} holds, then there exist exactly two liftings per connected component of $|C|^\phi$, one the negative of the other, which alternate precisely at $D \subset \{x_j\}_{j \in B}$.
\end{enumerate}
\end{prop}
\begin{rmk}\label{rmk:degenerated_structure}
Our condition that contracted boundaries are Ramond and carry gradings and the condition concerning the alternations of half-nodes of a node are conditions required for smoothing the graded structure. These are required in addition to the smoothing condition \eqref{eq:twists_at_half_nodes} which is a smoothing condition for spin structures.
In other words, these are the constraints obtained when degenerating graded $W$-structures on smooth surfaces to nodal ones.
\end{rmk}
In practice this proposition can be applied also to connected components of the normalization of $C,$ if it is not smooth, since a lifting lifts to the normalization.  

\begin{definition}\label{def:open_W_graded}
Consider a Fermat polynomial
\begin{equation}\label{eq:Fermat W}W=x_1^{r_1}+\cdots+x_a^{r_a}.\end{equation} Write $d=\mathrm{lcm}(r_1,\ldots,r_a)$.
A \emph{$W$-spin Riemann surface with boundary} is a tuple
\[(C,S_1,\ldots, S_a;\tau_1,\ldots,\tau_a;\phi_1,\ldots,\phi_a)\]
where
\begin{enumerate}
\item $C$ is a stable orbifold marked Riemann surface with boundary that is $d$-stable.
\item each $(S_i,\tau_i,\phi_i)$ are twisted $r_i$-spin structures.
\end{enumerate}
The $W$-spin structure on $C$ is \emph{compatible} if each $r_i$-spin structure is.

A \emph{lifting} for a $W$-spin Riemann surface with boundary is an $a$-tuple of liftings, one for each $J_i=\omega_C\otimes S_i^\vee$. For any boundary special marking or half-node $q$ we write $\alt(q)$ for the $a$-tuple whose $i^{th}$ component is $\alt_i(q),$ the alternation with respect to the $i^{th}$ lifting.
\end{definition}

\subsubsection{Point-insertion surfaces}\label{subsec: point insertion}
Hori \cite{Hori,Horiprivate} predicts the existence of open $r$-spin theories with $\lfloor\frac{r}{2}\rfloor$ types of boundary twists, rather than only allowing the twists to be $r-2$ as in \cite{BCT1,BCT2}. A potential application of this construction is to the conjectural open LG/CY correspondence \cite{Melissa,Walcher}. The works \cite{TZ1,TZ2,TZ3} suggest such a generalization for $r$-spin surfaces, and also for $W$-spin surfaces for special $W$.

To study this theory one needs to work with $(W,\mathfrak{h})$-surfaces, which are genus $g$ \emph{disconnected stable graded $W$-surfaces} $((C_\alpha)_{\alpha\in A},\MU)$ where $A$ is a finite set, each $C_\alpha$ is a connected stable graded $W$-spin surface with twists and alternations as above, and $\MU$ is a partial matching between boundary marked points and internal marked points of the surfaces. The matching $\MU$ satisfies:
\begin{enumerate}
    \item a boundary point of twist $(t,\ldots,t)$ may be matched only to an internal point of twist $(\frac{r-2-t}{2},\ldots,\frac{r-2-t}{2})$.
    \item Consider a matching between a boundary marking and an internal marking and take the connected components containing each of them. If we glue an arbitrary boundary point of each of them with the other, then we require that the resulting nodal surface is of genus $g$. Note that in $g=0$ it implies, in particular, all components to be of genus $0,$ and $\mathfrak{m}$ may match only points from different component.  
\end{enumerate}
A \TZ-surface is $\MU$-connected if  the resulting nodal surface above is connected. 
Two \TZ-surfaces are isomorphic if there is an isomorphism under the expected notion of isomorphism for the disconnect components, which also respects the matching $\MU$.

We also define an additional notion of equivalence, generated by the
following operation.
Let $((C_{\alpha})_{\alpha\in A},\mathfrak{m})$ be a \TZ-surface,
and suppose some $C_\alpha$ has a boundary node $n$ with the property
that if we normalize $C_{\alpha}$ at $n$, we obtain a component
$C'_{\alpha}$ with the normalized half-node $n_a$ and a smooth
disk $C^*$ containing only the other normalized half-node $n_b$ and
an internal marked point $z_{\xi}$. Suppose further that this internal
marked point is matched under $\MU$ to a boundary point $x_{\zeta}$
of some $C_{\beta}$. Then we can form a new $\MU$-connected surface
by removing $C^*$ from $C_{\alpha}$ to get $C'_{\alpha}$
and gluing $C'_{\alpha}$ and $C_{\beta}$ at the marked boundary points $n_a,
x_{\zeta}$ respectively. See Figure~\ref{fig:rh surface}. The equivalence
relation then identifies two $\MU$-connected surfaces if one can get from
the one surface to the other by performing this operation and its inverse
repeatedly.

 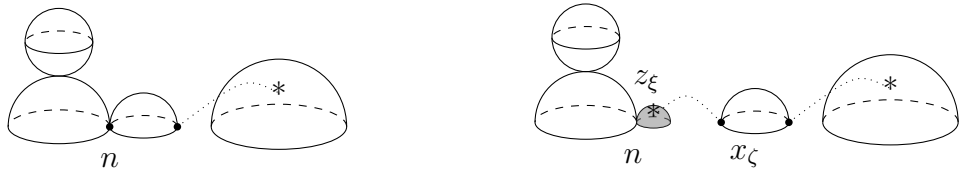
\begin{figure}[h]
         \centering

         \begin{subfigure}{.45\textwidth}
  \centering
\begin{tikzpicture}[scale=0.45]
\draw (0,-0.5) circle (1);
\draw (-1,-0.5) arc (180:360:1 and 0.333);
\draw[dashed](1,-0.5) arc (0:180:1 and 0.333);

\draw (-1.5,-3) arc (180:360:1.5 and 0.5);
\draw[dashed](1.5,-3) arc (0:180:1.5 and 0.5);
\draw(1.5,-3) arc (0:180:1.5);

\draw (1.5,-3) arc (180:360:1 and 0.333);
\draw[dashed](3.5,-3) arc (0:180:1 and 0.333);
\draw(3.5,-3) arc (0:180:1);

\node at (1.5,-3) [circle,fill,inner sep=1pt]{};

\draw (4.5,-3) arc (180:360:2 and 0.667);
\draw[dashed](8.5,-3) arc (0:180:2 and 0.667);
\draw(8.5,-3) arc (0:180:2);

\node at (3.5,-3) [circle,fill,inner sep=1pt]{};
\node at (6.5,-2){*};
\draw[dotted] (3.5,-3) .. controls (5.5,-1.5) ..(6.5,-2);

\node at (1.5,-4){$n$}
;

\end{tikzpicture} 
\end{subfigure}
\begin{subfigure}{.45\textwidth}
  \centering
\begin{tikzpicture}[scale=0.45]
\draw (0,-0.5) circle (1);
\draw (-1,-0.5) arc (180:360:1 and 0.333);
\draw[dashed](1,-0.5) arc (0:180:1 and 0.333);

\draw (-1.5,-3) arc (180:360:1.5 and 0.5);
\draw[dashed](1.5,-3) arc (0:180:1.5 and 0.5);
\draw(1.5,-3) arc (0:180:1.5);

\draw (1.5,-3) arc (180:360:0.5 and 0.167);
\draw[dashed](2.5,-3) arc (0:180:0.5 and 0.167);
\draw(2.5,-3) arc (0:180:0.5);

\fill[color = gray, opacity = 0.5] (1.5,-3) arc (180:360:0.5 and 0.167) arc (0:180:0.5);

\draw (4,-3) arc (180:360:1 and 0.333);
\draw[dashed](6,-3) arc (0:180:1 and 0.333);
\draw(6,-3) arc (0:180:1);

\draw (7,-3) arc (180:360:2 and 0.667);
\draw[dashed](11,-3) arc (0:180:2 and 0.667);
\draw(11,-3) arc (0:180:2);

\node at (6,-3) [circle,fill,inner sep=1pt]{};
\node at (9,-2){*};
\draw[dotted] (6,-3) .. controls (8,-1.5) ..(9,-2);

\node at (4,-3) [circle,fill,inner sep=1pt]{};

\node at (2,-2.8){*};

\draw[dotted] (4,-3) .. controls (3,-2) ..(2,-2.8);

\node at (1.4,-4){$ n$};
\node at (1.8,-1.8){$ z_\xi$};
\node at (4.7,-4){$ x_\zeta$};

\end{tikzpicture} 
\end{subfigure}

        \caption{An example of two equivalent \TZ~$(x^r,\mathfrak{h})$-surfaces $\MU$-connected surfaces.  Here $x_\zeta$ is alternating and can have twist $r-2m-2,$ for $0\leq m\leq\mathfrak{h},$ while $z_\xi$ has twist $m.$ Pointed arcs correspond to $\MU$-edges.}
        \label{fig:rh surface}
    \end{figure}

\subsubsection{Different theories}\label{subsub:different_theories}
At the moment, there are several constructions of open $W$-spin theories. They differ in what types of  $W$-spin Riemann surfaces with boundary they allow and the possible twists at each boundary marking. We summarize them as follows.

\begin{enumerate}[(i)]
 \item\label{it:GKT} \GKT-$W$ surfaces are genus zero $W$-spin Riemann surfaces with boundary where $W=x_1^{r_1}+ \dots+ x_a^{r_a} $ , and, for every $i=1,\ldots,a$ every boundary marked point either has $\tw_i=r_i-2,\alt_i=1$ or $\tw_i=\alt_i=0$. 
 \item \label{it:BCT}\BCT-$r$-spin surfaces are genus 0 or 1 $W$-spin Riemann surfaces with boundary where $W=x^r$ and all boundary markings alternate and have twist $r-2$.
 \item \label{it:PST} \PST-surfaces are $W$-spin Riemann surfaces with boundary of any genus where $W=x^2$, all twists are $0$, and all boundary marked points are alternating.
 \item \label{it:TZ} \TZ-$(W,\mathfrak{h})$ surfaces are genus zero $W$-spin Riemann surfaces with boundary where $W= x_1^r + \cdots + x_a^r$ where $a$ is odd, all boundary markings alternate, and the twist vector for any marking is of the form $\tw = (t,\dots, t)$, where $t \equiv r \pmod 2$ and $r-2 \ge t \ge r- 2 \mathfrak{h}-2$, for some $\mathfrak{h} \in \{0, \dots, \lfloor \tfrac{r}{2}\rfloor -1\}$.
\end{enumerate}

Each of these theories has a corresponding literature.  \GKT-$W$ surfaces are studied in  \cite{GKT2,GKT3}. \BCT-$r$-spin surfaces are studied in \cite{BCT1,BCT2,BCT3,GKT,TZ3}. \PST-surfaces are studied in \cite{PST14,Tes15,Bur15,Bur16,BT17,alexandrov2017refined,dijkgraaf2018developments}. \TZ-$W$ surfaces are studied in \cite{TZ1,TZ2}.

\begin{rmk}\label{rmk:labeling_issues}
One may enrich the data of graded spin surfaces by adding an additional labeling to marked points, e.g., by identifying the set $I$ of internal markings with the set $[|I|]=[l]$, and/or the set of boundary markings $B$ with $[|B|]=[k],$ and require isomorphisms to respect these identifications. The effect of such additional identifications is mainly combinatorial, e.g. the intersection numbers which will be defined below will just be scaled by a simple combinatorial factor. 
On the level of moduli spaces, the effect of adding this extra decoration is obtaining a (possibly orbifold) covering. The conventions for the cases \PST,\BCT,\TZ~chosen in \cite{PST14,BCT1,TZ1} respectively was to make these identifications for both sets $I,B,$ while in the case \GKT~\cite{GKT2} only $I$ was identified with $[l],$ and no such identification was made for $B.$
\end{rmk}

\subsubsection{$W$-spin dual graphs}The combinatorial data of (closed) $W$-spin surfaces, or (open) $W$-spin surfaces with a lifting can be encoded using decorated dual graphs. These are called \emph{$W$-spin graphs}, or \emph{$W$-spin graphs with a lifting}, respectively, in the literature. Their definitions are straightforward, but lengthy, so we will not provide it and instead give the following dictionary:
\medskip
\begin{center}
\begin{tabular}{c|c}
$W$-spin graphs & $W$-spin surfaces \\ \hline
closed/open vertices     & closed/open irreducible components     \\
boundary/internal edges      & boundary/internal nodes \\
boundary/internal tails & boundary/internal marked points \\
contracted boundary half-edges & contracted boundary half-nodes \\
\end{tabular}
\end{center}
\medskip
By open components we mean surfaces with boundary.
A half-edge (including a tail) has \emph{twists}
\[\tw_i(h)\in\{-1,\ldots,r_i-2, r_i -1\},~~i\in[a]\]
and \emph{alternations} \[\alt_i\in\Z_2,~~i\in[a]\]
which correspond to the twists and alternations of the corresponding markings. Thus, half-edges can be classified as broad/narrow/Ramond/Neveu-Schwarz. Vertices have notions of genus and, if open, also number of boundary components. There are constraints on twist and alternations arising from \eqref{eq:close_rank1_general}, \eqref{eq:open_rank1_general}, Proposition \ref{prop:graded_r_spin_prop}, Definition \ref{def:lifting_compatible}, and Definition~\ref{def:alt}.
We refer the reader to \cite[Section 1]{GKT2}.

In the \TZ-$W$-spin case, the graphs include additional $\MU$-edges, corresponding to the matching $\MU$ of Subsection~\ref{subsec: point insertion}, visually depicted in Figure~\ref{fig:rh surface}.
See \cite{TZ1} for more details.

We write $\Gamma(\Sigma)$ for the graded $W$-spin graph corresponding to the surface $\Sigma,$ and $\Sigma(\Gamma)$ for the \emph{topological graded $W$-spin surface} corresponding to the graph $\Gamma.$ 
\begin{ex}
In Figure~\ref{fig:graded_disk_and_graph}, we provide a graded $W$-spin graph and its corresponding dual graph $\Gamma$. Here, the open vertices are drawn as white vertices and closed as black vertices. Internal markings are denoted by $z_i$ and boundary markings by $x_i$ with their corresponding tails being solid and dashed, respectively.
\end{ex}

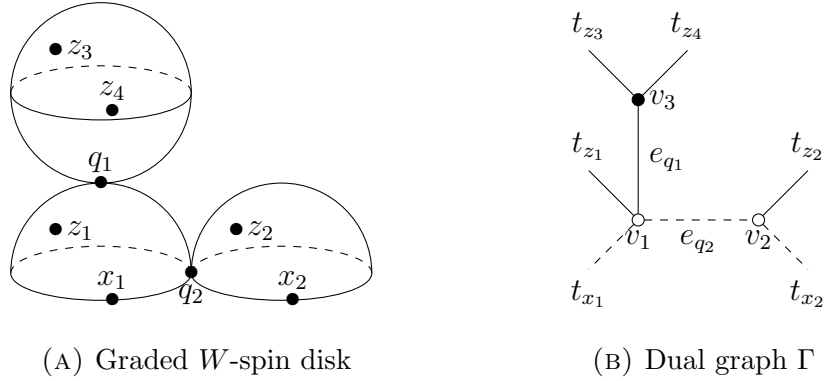
\begin{figure}
\begin{subfigure}{.4\textwidth}
  \centering

\begin{tikzpicture}[scale=0.6]
\vspace{0.15cm}

  \draw (0,4) circle (2cm);
   \draw (-2,4) arc (180:360:2 and 0.6);
  \draw[dashed] (2,4) arc (0:180:2 and 0.6);
  \draw (-2,0) arc (180:360:2 and 0.6);
  \draw[dashed] (2,0) arc (0:180:2 and 0.6);
  \draw (2,0) arc (0:180:2 and 2);
    \draw (2,0) arc (180:360:2 and 0.6);
  \draw[dashed] (6,0) arc (0:180:2 and 0.6);
  \draw (6,0) arc (0:180:2 and 2);

  \node (c) at (-1,.95) {$\bullet$};
\node[right] at (-1,.95) {$z_1$};

  \node (c) at (2,0) {$\bullet$};
\node[below] at (2,0) {$q_2$};

\node at (.25, -.6) {$\bullet$};
\node[above] at (.25,-.6) {$x_1$};

\node at (4.25, -.6) {$\bullet$};
\node[above] at (4.25,-.6) {$x_2$};

\node at (0,2) {$\bullet$};
\node[above] at (0,2) {$q_1$};
  \node (c) at (3,.95) {$\bullet$};
\node[right] at (3,.95) {$z_2$};

\node at (.25, 3.6) {$\bullet$};
\node[above] at (.25,3.6) {$z_4$};
  \node (c) at (-1,4.95) {$\bullet$};
\node[right] at (-1,4.95) {$z_3$};

\end{tikzpicture}
\vspace{0.15cm}

  \caption{Graded $W$-spin disk}
\end{subfigure}
\begin{subfigure}{.4\textwidth}
  \centering
\vspace{.24cm}
\begin{tikzpicture}[scale=0.4]

	\draw (0,0) circle (.2cm);
	\draw (0,4)[black, fill = black] circle (.2cm);
	\draw (4,0) circle (.2cm);
	
	\draw (0,0.2) -- (0,4);
	\draw[dashed] (0.2,0) -- (3.8,0);
	\node[below] at (2,0) {$e_{q_2}$};
	\node[right] at (0,2) {$e_{q_1}$};
	
	\draw (4.1414,.1414) -- (5.6414, 1.6414);
	\draw[dashed] (4.1414,-.1414) -- (5.6414, -1.6414);
	
	\draw (-.1414,.1414) -- (-1.6414, 1.6414);
	\draw[dashed] (-.1414,-.1414) -- (-1.6414, -1.6414);
	
	\node[above] at (5.6414, 1.6414) {$t_{z_2}$};
	\node[below] at (5.6414, -1.6414) {$t_{x_2}$};
	
	\node[above] at (-1.6414, 1.6414) {$t_{z_1}$};
	\node[below] at (-1.6414, -1.6414) {$t_{x_1}$};
	
	\draw (-.1414,4.1414) -- (-1.6414, 5.6414);
	\draw (.1414,4.1414) -- (1.6414, 5.6414);
	\node[above] at (-1.6414, 5.6414) {$t_{z_3}$};
	\node[above] at (1.6414, 5.6414) {$t_{z_4}$};
	
	\node[below] at (0,0) {$v_1$};
	\node[below] at (4,0) {$v_2$};
	\node[right] at (0,4) {$v_3$};
\end{tikzpicture}
\vspace{0.15cm}

  \caption{Dual graph $\Gamma$}
\end{subfigure}
\caption{A graded $W$-spin disk and its corresponding dual graph $\Gamma$}
\label{fig:graded_disk_and_graph}
\end{figure}
Geometric operations like normalizing or smoothing have their combinatorial counterpart. For example, the counterpart of partial normalization at a node is the detaching map at an edge. See Figure \ref{fig:smooth_and_detach}.

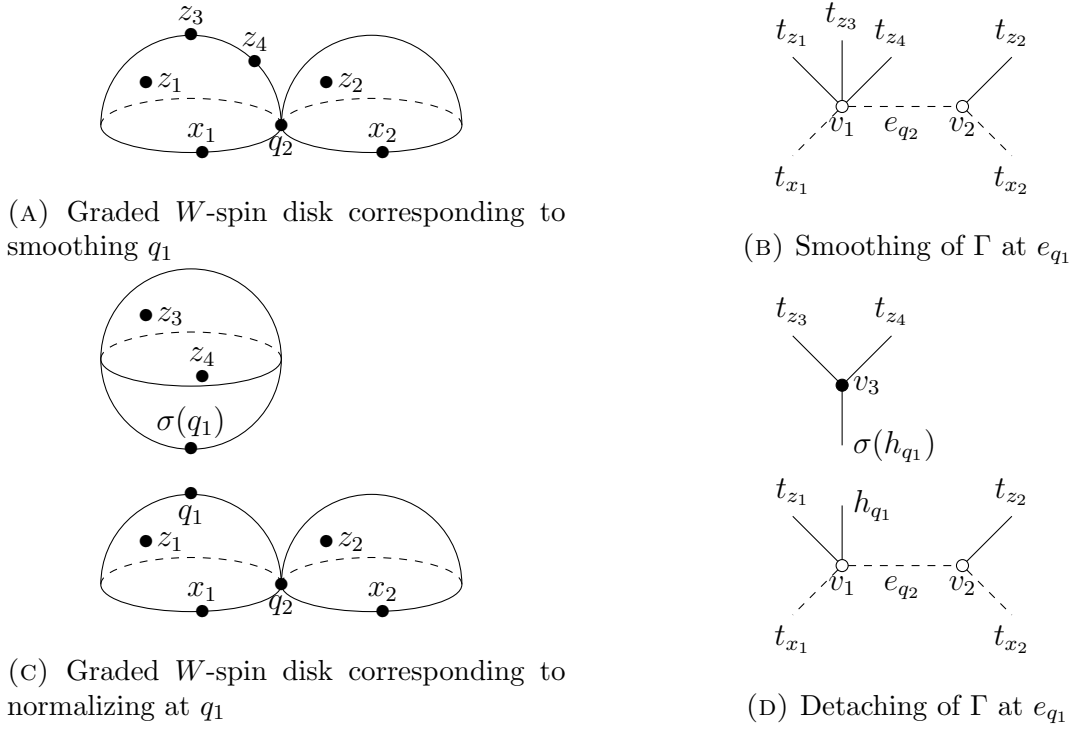
\begin{figure}
\begin{subfigure}{.45\textwidth}
  \centering

\begin{tikzpicture}[scale=0.6]
\vspace{0.15cm}

  \draw (-2,0) arc (180:360:2 and 0.6);
  \draw[dashed] (2,0) arc (0:180:2 and 0.6);
  \draw (2,0) arc (0:180:2 and 2);
    \draw (2,0) arc (180:360:2 and 0.6);
  \draw[dashed] (6,0) arc (0:180:2 and 0.6);
  \draw (6,0) arc (0:180:2 and 2);

  \node (c) at (-1,.95) {$\bullet$};
\node[right] at (-1,.95) {$z_1$};

  \node (c) at (2,0) {$\bullet$};
\node[below] at (2,0) {$q_2$};

\node at (.25, -.6) {$\bullet$};
\node[above] at (.25,-.6) {$x_1$};

\node at (4.25, -.6) {$\bullet$};
\node[above] at (4.25,-.6) {$x_2$};

  \node (c) at (3,.95) {$\bullet$};
\node[right] at (3,.95) {$z_2$};

\node at (1.414, 1.414) {$\bullet$};
\node[above] at (1.414,1.414) {$z_4$};
  \node (c) at (0,2) {$\bullet$};
\node[above] at (0,2) {$z_3$};
\end{tikzpicture}
\vspace{0.15cm}
  \caption{Graded $W$-spin disk corresponding to smoothing $q_1$}
\end{subfigure}\qquad
\begin{subfigure}{.45\textwidth}
  \centering
\vspace{.24cm}
\begin{tikzpicture}[scale=.4]

	\draw (0,0) circle (.2cm);
	\draw (4,0) circle (.2cm);
	
	\draw[dashed] (0.2,0) -- (3.8,0);
	\node[below] at (2,0) {$e_{q_2}$};
	
	\draw (4.1414,.1414) -- (5.6414, 1.6414);
	\draw[dashed] (4.1414,-.1414) -- (5.6414, -1.6414);
	
	\draw (-.1414,.1414) -- (-1.6414, 1.6414);
	\draw[dashed] (-.1414,-.1414) -- (-1.6414, -1.6414);
	
	\node[above] at (5.6414, 1.6414) {$t_{z_2}$};
	\node[below] at (5.6414, -1.6414) {$t_{x_2}$};
	
	\node[above] at (-1.6414, 1.6414) {$t_{z_1}$};
	\node[below] at (-1.6414, -1.6414) {$t_{x_1}$};
	
	\draw (0,.2) -- (0, 2.2);
	\draw (.1414,.1414) -- (1.6414, 1.6414);
	\node[above] at (0, 2.2) {$t_{z_3}$};
	\node[above] at (1.6414, 1.6414) {$t_{z_4}$};
	
	\node[below] at (0,0) {$v_1$};
	\node[below] at (4,0) {$v_2$};
\end{tikzpicture}
\vspace{0.15cm}

  \caption{Smoothing of $\Gamma$ at $e_{q_1}$}
\end{subfigure}

\begin{subfigure}{.45\textwidth}
  \centering

\begin{tikzpicture}[scale=0.6]
\vspace{0.15cm}

  \draw (0,5) circle (2cm);
   \draw (-2,5) arc (180:360:2 and 0.6);
  \draw[dashed] (2,5) arc (0:180:2 and 0.6);
  \draw (-2,0) arc (180:360:2 and 0.6);
  \draw[dashed] (2,0) arc (0:180:2 and 0.6);
  \draw (2,0) arc (0:180:2 and 2);
    \draw (2,0) arc (180:360:2 and 0.6);
  \draw[dashed] (6,0) arc (0:180:2 and 0.6);
  \draw (6,0) arc (0:180:2 and 2);

  \node (c) at (-1,.95) {$\bullet$};
\node[right] at (-1,.95) {$z_1$};

  \node (c) at (2,0) {$\bullet$};
\node[below] at (2,0) {$q_2$};

\node at (.25, -.6) {$\bullet$};
\node[above] at (.25,-.6) {$x_1$};

\node at (4.25, -.6) {$\bullet$};
\node[above] at (4.25,-.6) {$x_2$};

\node at (0,2) {$\bullet$};
\node[below] at (0,2) {$q_1$};

\node at (0,3) {$\bullet$};
\node[above] at (0,3) {$\sigma(q_1)$};
  \node (c) at (3,.95) {$\bullet$};
\node[right] at (3,.95) {$z_2$};

\node at (.25, 4.6) {$\bullet$};
\node[above] at (.25,4.6) {$z_4$};
  \node (c) at (-1,5.95) {$\bullet$};
\node[right] at (-1,5.95) {$z_3$};

\end{tikzpicture}
\vspace{0.15cm}

  \caption{Graded $W$-spin disk corresponding to normalizing at $q_1$}
\end{subfigure}\qquad
\begin{subfigure}{.45\textwidth}
  \centering
\vspace{.24cm}
\begin{tikzpicture}[scale=0.4]

	\draw (0,0) circle (.2cm);
	\draw (0,6)[black, fill = black] circle (.2cm);
	\draw (4,0) circle (.2cm);
	
	\draw (0,0.2) -- (0,2);
	\draw (0,4) -- (0,6);
	\draw[dashed] (0.2,0) -- (3.8,0);
	\node[below] at (2,0) {$e_{q_2}$};
	\node[right] at (0,2) {$h_{q_1}$};
	\node[right] at (0,4) {$\sigma(h_{q_1})$};

	\draw (4.1414,.1414) -- (5.6414, 1.6414);
	\draw[dashed] (4.1414,-.1414) -- (5.6414, -1.6414);
	
	\draw (-.1414,.1414) -- (-1.6414, 1.6414);
	\draw[dashed] (-.1414,-.1414) -- (-1.6414, -1.6414);
	
	\node[above] at (5.6414, 1.6414) {$t_{z_2}$};
	\node[below] at (5.6414, -1.6414) {$t_{x_2}$};
	
	\node[above] at (-1.6414, 1.6414) {$t_{z_1}$};
	\node[below] at (-1.6414, -1.6414) {$t_{x_1}$};
	
	\draw (-.1414,6.1414) -- (-1.6414, 7.6414);
	\draw (.1414,6.1414) -- (1.6414, 7.6414);
	\node[above] at (-1.6414, 7.6414) {$t_{z_3}$};
	\node[above] at (1.6414, 7.6414) {$t_{z_4}$};
	
	\node[below] at (0,0) {$v_1$};
	\node[below] at (4,0) {$v_2$};
	\node[right] at (0,6) {$v_3$};
\end{tikzpicture}
\vspace{0.15cm}

  \caption{Detaching of $\Gamma$ at $e_{q_1}$}
\end{subfigure}

\caption{Smoothings and detachings of the dual graph $\Gamma$ from Figure~\ref{fig:graded_disk_and_graph}}
\label{fig:smooth_and_detach}
\end{figure}
\begin{nn}
When $\Gamma=\Gamma(\Sigma),$ we write $n_h=n_h(\Sigma)$ for the half-node in the normalization of $\Sigma$ that corresponds to the half-edge corresponding to a half-node that is not a contracted boundary. We write $n_e$ for the node in $\Sigma$ that corresponds to the edge $e$.  When $h$ is an internal half-edge or a contracted boundary tail, we sometimes write $z_h$ instead of $n_h$.  When $h$ is a boundary half-edge, we sometimes write $x_h$ instead of $n_h$.\end{nn}

\subsection{The moduli spaces}
Fan, Jarvis, and Ruan constructed a moduli space $\M_{g,n}^{\text{FJRW}, W}$ consisting of compact stable $W$-spin orbicurves, for which they then provide an enumerative
theory.  

Denote by $\M_{g,k,l}^W,$ the moduli spaces of stable connected genus $g$ $W$-spin Riemann surfaces with boundary having $k$ boundary marked points and $l$ internal marked points, together with a grading. Analogously, we write $\M_{g,k,l}^{W,\mathfrak{h}}$ for the moduli spaces of stable $\MU$-connected genus $g$ \TZ-$(W,\mathfrak{h})$-surfaces with boundary having  $k$ boundary marked points and $l$ internal marked points, together with a grading.
There is a set-theoretic decomposition of these moduli spaces into compact connected components according to the different twists and topologies of the underlying surfaces.
The following results are proven in the cases $g=0$ \PST\, in \cite{PST14}, $g>0$ \PST\, in \cite{ST_unpublished}, $g=0$~\BCT\, in \cite{BCT1}, $g=0$ 
\GKT\, in \cite{GKT2}, $g=0,1$ \TZ, and  $g=1$ \BCT, in\cite{TZ1}.
\begin{prop}\label{prop:moduli_orbi_corners}The aforementioned moduli spaces are smooth real orbifolds with corners in the sense of \cite[Section 3]{Zernik}, have universal curves carrying $r_i$-spin bundle for all $i\in[a],$ and a grading. Moreover, we have that
\begin{equation}\label{real dim of open Moduli of discs}
\dim_{\R}(\M_{g,k,l}^{W})=k+2l+3g-3.
\end{equation}
The results for $\M_{g,k,l}^{W,\mathfrak{h}}$ are the same.
\end{prop}
The proof essentially follows the procedure performed in \cite[Section 2]{Zernik} (and also sketched in Section 1 of \cite{Zernik}), which we briefly sketch.  For more details, see, for example \cite{BCT1}, \S3.2 for the case of $g=0$ \BCT~$r$-spin theory.
\begin{proof}[Sketch of proof]
We first consider the case where the sets of internal and boundary markings are identified with $[l],[k]$ respectively, see Remark \ref{rmk:labeling_issues}.
    We have the following sequence of maps, with arrows
labeled by the step in which they are defined. Starting from (B) each step preserves the structure of orbifold with corners. Each step allows pulling back the universal curve and spin structures, and the last step adds the grading:
\begin{equation}
\label{eq:OWCsequence}
\begin{tikzcd}[row sep=scriptsize, column sep = small]
\M_{g,k,l}^{W} \arrow{r}{(E)}
&	\widehat{\mathcal{M}}_{g,k,l}^{W}\arrow{r}{(D)}  & 	\widetilde{\mathcal{M}}_{g,k,l}^{W} 	\arrow{r}{(C)}  & \widetilde{\mathcal{M}}_{g,k,l}^{W;\Z_2}\arrow{r}{(B)}  &
\overline{\mathcal{M}}_{g,k+2l}^{W;\Z_2}\arrow{r}{(A)} &\overline{\mathcal{M}}_{g,k+2l}^{'W} 
\end{tikzcd}
\end{equation}
 Let us now briefly define the moduli spaces and maps appearing in \eqref{eq:OWCsequence}.

{\bf Step (A):} The morphism $\overline{\mathcal{M}}_{g,k+2l}^{W;\Z_2}\rightarrow\overline{\mathcal{M}}_{g,k+2l}^{'W}$. The space $\overline{\mathcal{M}}_{g,k+2l}^{'W}$ is the (open and closed) sub-orbifold of $\M_{g,k+2l}^{W}$ given by the conditions on twists imposed by the specific model, see Subsection \ref{subsub:different_theories}.
Inside this space, $\M_{g,k+2l}^{W, \Z_2}$ is the fixed locus of the involution defined by
\[(C;w_1, \ldots, w_{k+2l}, \{S_i\}) \mapsto (\overline{C}; w_1, \ldots, w_k, w_{k+l+1},\ldots, w_{k+2l},w_{k+1},\ldots,w_{k+l}, \{\overline{S}_i\}),\]
where $\overline{C}$, $\overline{S}_i,~i\in[a]$ are the same as $C$, $S_i,~i\in[a]$, respectively, but with the conjugate complex structure.
As a  fixed locus of the above anti-holomorphic involution, this moduli
 has the structure of a real orbifold. A point in the fixed locus comes equipped with an involution $\phi: C \rightarrow C$ given by conjugation which is covered by involutions $\tilde \phi_i: S_i \rightarrow S_i$ for all $i$.
This moduli space parameterizes isomorphism types of marked spin spheres with a real structure, involutions $\phi$, $\tilde\phi_i,~i\in[a]$, and the prescribed twists. It also inherits a universal curve via pullback.

{\bf Step (B):} The morphism $\widetilde{\mathcal{M}}_{g,k,l}^{W;\Z_2}\rightarrow\overline{\mathcal{M}}_{g,k+2l}^{W;\Z_2}$. 
Take $N\hookrightarrow \overline{\mathcal{M}}_{g,k+2l}^{W;\Z_2}$ to be the real simple normal crossing divisor consisting of
 curves with at least one boundary node.
Via the real hyperplane blowup of
\cite{Zernik}, we cut $\overline{\mathcal{M}}_{g,k+2l}^{W;\Z_2}$ along $N$, yielding an orbifold with corners $\widetilde{\mathcal{M}}_{g,k,l}^{W, \mu_2}$. The morphism in this step is then constructed by gluing the cuts described here.

{\bf Step (C):} The morphism $\widetilde{\mathcal{M}}_{g,k,l}^{W}
\rightarrow\widetilde{\mathcal{M}}_{g,k,l}^{W;\Z_2}$. From here, we define $\widetilde{\mathcal{M}}_{g,k,l}^{W}$ to be the disconnected $2$-to-$1$ cover of $\widetilde{\mathcal{M}}_{g,k,l}^{W;\Z_2}$. The generic point of the moduli space $\widetilde{\mathcal{M}}_{g,k,l}^{W}$ corresponds to a smooth marked real spin curve with a choice of a distinguished connected component of $C\setminus C^\phi$.  This induces an orientation on $C^{\phi}$.  

{\bf Step (D):} The morphism $\widehat{\mathcal{M}}_{g,k,l}^{W}\hookrightarrow\widetilde{\mathcal{M}}_{g,k,l}^{W}$. Inside $\widetilde{\mathcal{M}}_{0,k,l}$, we denote by $\widehat{\mathcal{M}}_{g,k,l}^{W}$ the union of connected components such that the marked points $w_{k+1},\ldots, w_{k+l}$ lie in the distinguished stable disk and, for even $r_i,$ the $i^{th}$ spin structure is compatible in the sense of Definition~\ref{def:lifting_compatible}. The morphism here is inclusion.

{\bf Step (E):} The morphism $\M_{g,k,l}^{W}\rightarrow\widehat{\mathcal{M}}_{g,k,l}^{W}$.  Here, $\M_{g,k,l}^{W}$ is the orbifold with corners cover of $\widehat{\mathcal{M}}_{g,k,l}^{W}$ given by a choice of grading. 

When $B$ is not identified with $[k],$ as in the convention of \cite{GKT} the mere difference is that in $\overline{\mathcal{M}}_{g,k+2l}^{'W}$ the last $k$ points are unlabeled. This space is again a smooth orbifold, which is a quotient of its labeled cousin.

The construction in the case of \TZ~theory is similar, but in order to treat moduli spaces of $\MU$-connected surfaces we need to further consider quotients of products of moduli spaces as above, by a discrete group of combinatorial automorphisms of the corresponding $\MU$-connected graphs.
\end{proof}

For any connected stable genus $g$ $W$-spin graph $\Gamma$ with a grading, there is a corresponding moduli space $\M_\Gamma^{W}$ parameterizing the $W$-spin Riemann surfaces with boundary that have $\Gamma$ as its dual graph, along with their degenerations. It is also a real orbifold with corners carrying a universal curve, universal spin lines, and universal gradings. 

There is an embedding $\iota_\Gamma:\M_\Gamma^{W}\hookrightarrow\M^{W}_{g,k,l}$ where $k$ is the number of boundary tails, and $l$ the number of internal tails. The real codimension of $\M_\Gamma^{W}$ in $\M^{W}_{g,k,l}$ is the sum of the number of boundary edges, contracted boundary edges and twice the number the internal edges.
More generally, $\Gamma$ is a partial smoothing of $\Lambda$ if a disk $\Sigma(\Gamma)$ associated with $\Gamma$ is obtained from some disk $\Sigma(\Lambda)$ associated with $\Lambda$ by smoothing some nodes. In this case we have
\[\iota_{\Lambda,\Gamma}:\M_\Lambda^{W}\hookrightarrow\M_\Gamma^{W}.\]

The above paragraph also holds for the  moduli spaces $\M_{g,k,l}^{(W,\mathfrak{h})}$ of \TZ-$(W, \mathfrak{h})$-surfaces with boundary, after redefining $k$ to be the number of boundary tails which are not connected by $\MU$-edges, and $l$ the number of internal tails which are not connected by $\MU$-edges.

\smallskip

There are forgetful maps between some of the $\oCM_{\Gamma}^W$:
We can forget a marked point if its twist is $0,$ and, in the case it is a boundary marked point it is additionally non-alternating. However, this may create unstable irreducible components. We repeatedly contract any unstable irreducible component. 
This process might create new boundary marked points which were formerly boundary half-nodes. These new boundary marked points may be non-alternating with $\tw=\vec{0}$. In this case, we repeat the process. If the process terminates with some unstable connected components, then we remove them. Graph-theoretically, this will result in a new dual graph $\Gamma'$. Using this process, we can define
\[\text{For}_{\text{non-alt}}: \M_{\Gamma}^W \rightarrow \M_{\Gamma'}^W.\]

Let $\Gamma$
be a pre-graded $W$-spin graph. The graded $r_i$-spin graph $\text{for}_{\text{spin}\neq i}\Gamma$ is an $r_i$-spin graph with a lifting obtained from $\Gamma$ by forgetting the additional $\tw_j,\alt_j$ data, for $j\neq i$ and from forgetting all tails $t$ with $(\tw_i(t),\alt_i(t))=(0,0)$.
We obtain a map
\[\text{For}_{\text{spin}\neq i}:\oCM_{\Gamma}^W\to\oCM^{1/r_i}_{\text{for}_{\spinqi}\Gamma}\]
induced by the map. 

If $\Gamma'$ is obtained from $\Gamma$ by detaching the boundary edges $E'\subseteq E^B,$ then we have a dimension preserving projection $\op{Detach}_{E'}:\M_\Gamma^{W}\to \M_{\Gamma'}^{W}$ (or $\op{Detach}_{E'}:\M_{\Gamma}^{(W,\mathfrak{h})}\to \M_{\Gamma'}^{(W,\mathfrak{h})}$), which is a diffeomorphism if the automorphism groups of $\Gamma,\Gamma'$ are the same.

If $\Gamma$ is an $\MU$-connected graph, and $\Gamma'$ is obtained from $\Gamma$ by detaching some of the $\MU$-edges then we have a dimension preserving projection $\M_\Gamma^{(W,\mathfrak{h})}\to \M_{\Gamma'}^{(W,\mathfrak{h})}$, which is a diffeomorphism if the automorphism groups of $\Gamma,\Gamma'$ are the same.

\begin{ex}\label{ex: hexagon}
Consider the case $W=x^2,$ moduli space of disks with one internal point $z_1$ of twist $0$ and three non-alternating marked boundary points $x_1, x_2, x_3$ with twist $0$. By \eqref{real dim of open Moduli of discs}, we know that the real dimension of $\oCM^{x^2}_{g=0,\{0,0,0\},\{0\}}$
is $2$. The moduli space has two connected components, depending on the
cyclic order of the boundary marked points. Each connected
component is a hexagon and has
 six codimension 1 boundary strata and six codimension 2 boundary strata. The real dimension of each stratum is again $3$ minus the number of irreducible components. We draw one of the connected
components of the moduli space in Figure~\ref{fig:moduliExample}. In this figure, each stratum is accompanied by a depiction of a representative stable marked disk $\Sigma$. 

\begin{figure}

  \centering
  \begin{tikzpicture}[scale=0.7]

\newdimen\Rad
   \Rad=5cm
   \draw[line width=0.25mm] (0:\Rad) \foreach \x in {60,120,...,360} {  -- (\x:\Rad) };
   \draw (5,0)[black, fill = black] circle (.1cm);
   \draw (-5,0)[black, fill = black] circle (.1cm);
   \draw (2.5, 4.33)[black, fill = black] circle (.1cm);
      \draw (-2.5, 4.33)[black, fill = black] circle (.1cm);
         \draw (2.5, -4.33)[black, fill = black] circle (.1cm);
            \draw (-2.5, -4.33)[black, fill = black] circle (.1cm);

 %Generic
  \draw (0,0) circle (1cm);
  \node at (0, 0) {$\bullet$};
\node[above] at (0,0) {$z_1$};
  \node at (1, 0) {$\bullet$};
\node[right] at (1,0) {$x_1$};
\node at (-.5, .866) {$\bullet$};
\node[above] at (-.5,.866) {$x_2$};
\node at (-.5, -.866) {$\bullet$};
\node[below] at (-.5,-.866) {$x_3$};

%First codim 1
  \draw (0,6) circle (1cm);
    \draw (0,8) circle (1cm);

  \node at (0, 6) {$\bullet$};
\node[above] at (0,6) {$z_1$};
  \node at (0, 5) {$\bullet$};
\node[below] at (0,5) {$x_1$};
\node at (.866, 8.5) {$\bullet$};
\node[right] at (.866,8.5) {$x_2$};
\node at (-.866, 8.5) {$\bullet$};
\node[left] at (-.866,8.5) {$x_3$};

%Second codim 1

  \draw (5.196,3) circle (1cm);
    \draw (6.928,4) circle (1cm);

  \node at (5.196,3) {$\bullet$};
\node[above] at (5.196,3) {$z_1$};
  \node at (6.428, 4.866) {$\bullet$};
\node[above] at (6.428, 4.866) {$x_3$};
\node at (7.79423, 4.5) {$\bullet$};
\node[right] at (7.79423, 4.5) {$x_2$};
\node at (7.428, 3.134) {$\bullet$};
\node[below] at (7.428, 3.134) {$x_1$};

%Third codim 1

  \draw (5.196,-3) circle (1cm);
    \draw (6.928,-4) circle (1cm);

  \node at (5.196,-3) {$\bullet$};
\node[above] at (5.196,-3) {$z_1$};
  \node at (7.794, -3.5) {$\bullet$};
\node[right] at (7.794, -3.5) {$x_2$};
  \node at (4.330, -2.5) {$\bullet$};
\node[above] at (4.330, -2.5) {$x_3$};
\node at (7, -5) {$\bullet$};
\node[below] at (7, -5) {$x_1$};

%Fourth codim 1

  \draw (0,-6) circle (1cm);
    \draw (0,-8) circle (1cm);

  \node at (0, -6) {$\bullet$};
\node[above] at (0,-6) {$z_1$};
  \node at (1, -8) {$\bullet$};
\node[right] at (1,-8) {$x_2$};
\node at (0, -9) {$\bullet$};
\node[below] at (0,-9) {$x_1$};
\node at (-1, -8) {$\bullet$};
\node[left] at (-1,-8) {$x_3$};
%Fifth codim 1

  \draw (-5.196,-3) circle (1cm);
    \draw (-6.928,-4) circle (1cm);

  \node at (-5.196,-3) {$\bullet$};
\node[above] at (-5.196,-3) {$z_1$};
  \node at (-7.794, -3.5) {$\bullet$};
\node[left] at (-7.794, -3.5) {$x_3$};
  \node at (-4.330, -2.5) {$\bullet$};
\node[above] at (-4.330, -2.5) {$x_2$};
\node at (-7, -5) {$\bullet$};
\node[below] at (-7, -5) {$x_1$};

%Sixth codim 1

  \draw (-5.196,3) circle (1cm);
    \draw (-6.928,4) circle (1cm);

  \node at (-5.196,3) {$\bullet$};
\node[above] at (-5.196,3) {$z_1$};
  \node at (-6.428, 4.866) {$\bullet$};
\node[above] at (-6.428, 4.866) {$x_2$};
\node at (-7.79423, 4.5) {$\bullet$};
\node[left] at (-7.79423, 4.5) {$x_3$};
\node at (-7.428, 3.134) {$\bullet$};
\node[below] at (-7.428, 3.134) {$x_1$};

%First Codim 2 %This will be between 1 and 2 above and continue clockwise

  \draw (3.25,5.629) circle (1cm);
    \draw (4.25,7.361) circle (1cm);
        \draw (5.25,9.093) circle (1cm);

            \node at (3.25,5.629) {$\bullet$};
\node[below] at (3.25,5.629) {$z_1$};

  \node at (5.116,6.861) {$\bullet$};
\node[right] at (5.116,6.861) {$x_1$};
\node at (6.25,9.093 ) {$\bullet$};
\node[right] at (6.25,9.093) {$x_2$};
\node at (4.75, 9.959) {$\bullet$};
\node[left] at (4.75, 9.959) {$x_3$};

%Second Codim 2

  \draw (6.5,0) circle (1cm);
    \draw (8.5,0) circle (1cm);
        \draw (10.5,0) circle (1cm);

          \node at (6.5,0) {$\bullet$};
\node[above] at (6.5,0) {$z_1$};
  \node at (8.5,1) {$\bullet$};
\node[above] at (8.5,1) {$x_3$};
\node at (11,.866 ) {$\bullet$};
\node[right] at (11,.866) {$x_2$};
\node at (11,-.866) {$\bullet$};
\node[right] at (11,-.866) {$x_1$};

%Third Codim 2

  \draw (3.25,-5.629) circle (1cm);
    \draw (4.25,-7.361) circle (1cm);
        \draw (5.25,-9.093) circle (1cm);

            \node at (3.25,-5.629) {$\bullet$};
\node[below] at (3.25,-5.629) {$z_1$};

  \node at (3.384,-7.861) {$\bullet$};
\node[left] at (3.384,-7.861) {$x_3$};
\node at (6.25,-9.093 ) {$\bullet$};
\node[right] at (6.25,-9.093) {$x_2$};
\node at (4.75, -9.959) {$\bullet$};
\node[left] at (4.75, -9.959) {$x_1$};

%Fourth Codim 2

  \draw (-3.25,-5.629) circle (1cm);
    \draw (-4.25,-7.361) circle (1cm);
        \draw (-5.25,-9.093) circle (1cm);

            \node at (-3.25,-5.629) {$\bullet$};
\node[below] at (-3.25,-5.629) {$z_1$};

  \node at (-3.384,-7.861) {$\bullet$};
\node[right] at (-3.384,-7.861) {$x_2$};
\node at (-6.25,-9.093 ) {$\bullet$};
\node[left] at (-6.25,-9.093) {$x_3$};
\node at (-4.75, -9.959) {$\bullet$};
\node[right] at (-4.75, -9.959) {$x_1$};

%Fifth Codim 2

  \draw (-6.5,0) circle (1cm);
    \draw (-8.5,0) circle (1cm);
        \draw (-10.5,0) circle (1cm);

                  \node at (-6.5,0) {$\bullet$};
\node[above] at (-6.5,0) {$z_1$};
  \node at (-8.5,1) {$\bullet$};
\node[above] at (-8.5,1) {$x_2$};
\node at (-11,.866 ) {$\bullet$};
\node[left] at (-11,.866) {$x_3$};
\node at (-11,-.866) {$\bullet$};
\node[left] at (-11,-.866) {$x_1$};

%Sixth Codim 2

  \draw ( -3.25,5.629) circle (1cm);
    \draw (-4.25,7.361) circle (1cm);
        \draw (-5.25,9.093) circle (1cm);

            \node at (-3.25,5.629) {$\bullet$};
\node[below] at (-3.25,5.629) {$z_1$};

  \node at (-5.116,6.861) {$\bullet$};
\node[left] at (-5.116,6.861) {$x_1$};
\node at (-6.25,9.093 ) {$\bullet$};
\node[left] at (-6.25,9.093) {$x_3$};
\node at (-4.75, 9.959) {$\bullet$};
\node[right] at (-4.75, 9.959) {$x_2$};

\end{tikzpicture}
 \caption{One connected component of the moduli space of stable disks with 3 boundary markings and one internal marking.}
\label{fig:moduliExample}
\end{figure}
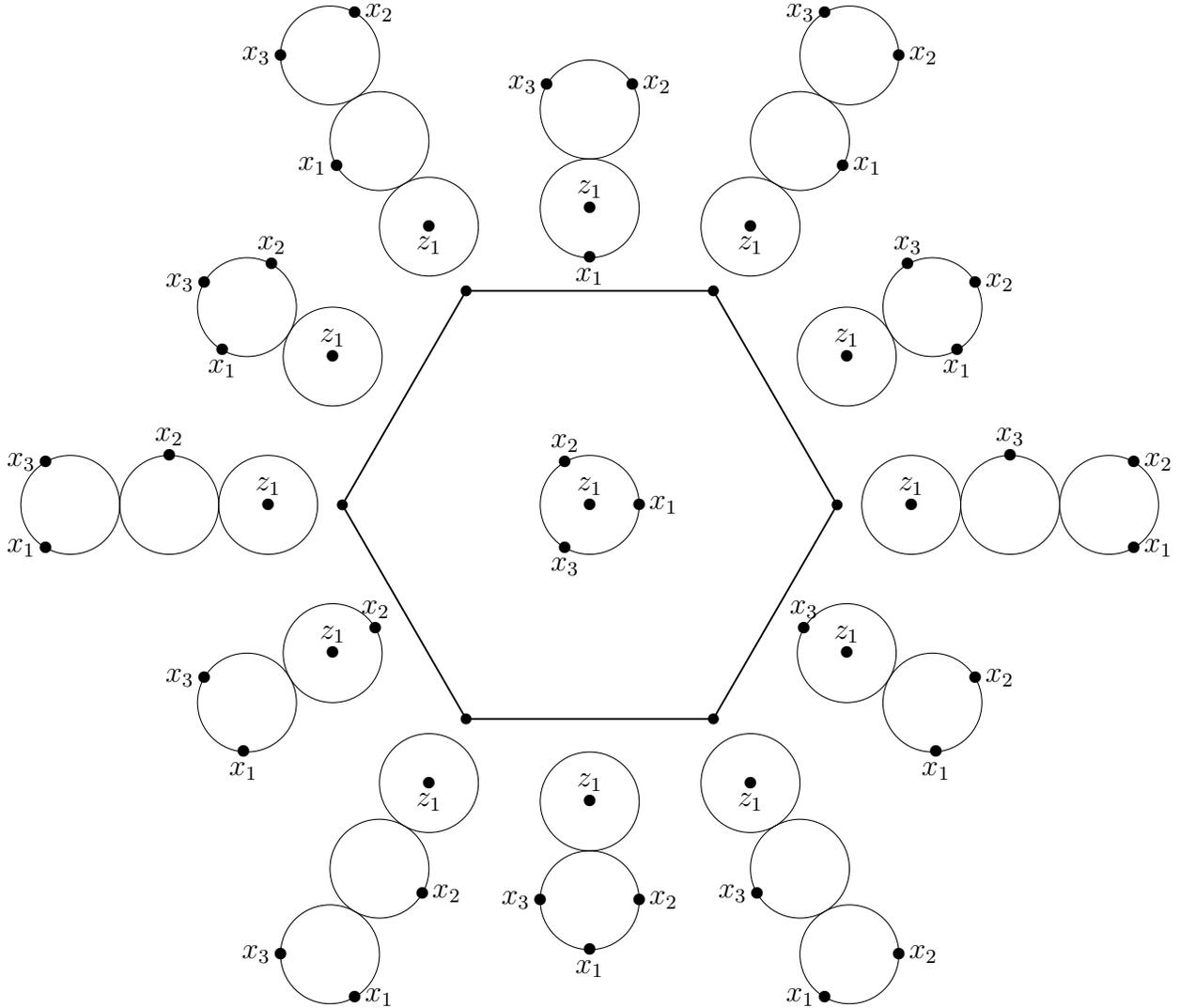
\end{ex}
\begin{ex}
    We now consider a \TZ~type example. Let $r=2,~k=3,1$ with all twist being $0$ again, and $\mathfrak{h}=0.$ Figure \ref{fig:point insertion for r=2 k=3 l=1} shows the moduli space  $\M^{(x^2,\mathfrak{h}=0)}_{g=0,{\{0,0,0\}},\{0\}}.$ For each original cyclic order of $x_1,x_2,x_3$ we obtain a copy of $\M^{x^2}_{0,\{0,0,0\},\{0\}},$ which is enlarged by adding new connected components of $\MU$-connected surfaces. The resulting space, per cyclic order, consists of two copies of $\M^{x^2}_{0,\{0,0,0\},\{0\}},$ which are combinatorially hexagons, and six copies of $\M^{x^2}_{0,\{0\},\{0,0\}},$ which are combinatorially bigons. Edges of different components can be identified via the equivalence class as in the figure. If we glue boundaries according to these identifications the resulting space is a topological sphere, per cyclic order.
    
    \begin{figure}[h]
         \centering
        \begin{tikzpicture}[scale=0.8,every node/.style={font=\fontsize{8}{10}\selectfont}]
    \draw (-1+1.05,0) -- (1+1.05,0) -- (2+1.05,1.4) -- (1+1.05,2.8) -- (-1+1.05,2.8) -- (-2+1.05,1.4) -- cycle;
    
    \draw (0.9+1.05,1.3) arc (0:360:0.9);
    \node at (0+1.05,1.3)[circle,fill,inner sep=1pt]{};
    \node at (-0.7+1.05,0.73)[circle,fill,inner sep=1pt]{};
    \node at (-0.75+1.05,0.5) {$0_{1}$};
    \node at (0.7+1.05,0.73)[circle,fill,inner sep=1pt]{};
    \node at (0.9+1.05,0.5) {$0_{2}$};
    \node at (0+1.05,2.2)[circle,fill,inner sep=1pt]{};
    \node at (0+1.05,2.45) {$0_{3}$};

    \draw (0.6+1.05,-1.2) arc (0:360:0.6);
    \node at (-0.45+1.05,-1.6)[circle,fill,inner sep=1pt]{};
    \node at (-0.6+1.05,-1.7) {$0_{1}$};
    \node at (0.45+1.05,-1.6)[circle,fill,inner sep=1pt]{};
    \node at (0.8+1.05,-1.7) {$0_{2}$};
    \node at (0+1.05,-0.6)[circle,fill,inner sep=1pt]{};
    \node at (0+1.05,-0.35) {$0_{3}$};
    
    \draw (0.5+1.05,-2.3) arc (0:360:0.5);
    \node at (0+1.05,-2.3)[circle,fill,inner sep=1pt]{};

    \draw (8,0) -- (10,0) -- (11,1.4) -- (10,2.8) -- (8,2.8) -- (7,1.4) -- cycle;
    
    \draw (9.6,1.3) arc (0:360:0.9);
    \node at (8.7,1.3){$*$};
    \node at (8,0.73)[circle,fill,inner sep=1pt]{};
    \node at (7.95,0.5) {$0_{1}$};
    \node at (9.4,0.73)[circle,fill,inner sep=1pt]{};
    \node at (9.6,0.5) {$0_{2}$};
    \node at (8.7,2.2)[circle,fill,inner sep=1pt]{};
    \node at (8.7,2.45) {$0_{3}$};

    \draw (10.75,1.4) arc (0:360:0.5);
    \node at (10.25,1.4)[circle,fill,inner sep=1pt]{};
    \node at (9.75,1.4) {$*$};

    \draw[dotted] (9.75,1.4) .. controls (9.2,1.5) .. (8.7,1.3);

    \draw (9.4,-1.2) arc (0:360:0.6);
    \node at (8.35,-1.6)[circle,fill,inner sep=1pt]{};
    \node at (8.2,-1.7) {$0_{1}$};
    \node at (9.25,-1.6)[circle,fill,inner sep=1pt]{};
    \node at (9.6,-1.7) {$0_{2}$};
    \node at (8.8,-0.6)[circle,fill,inner sep=1pt]{};
    \node at (8.8,-0.35) {$0_{3}$};
    
    \draw (9.3,-2.3) arc (0:360:0.5);
    \node at (8.8,-2.3){$*$};

    \draw (10.95,-2) arc (0:360:0.5);
    \node at (10.45,-2)[circle,fill,inner sep=1pt]{};
    \node at (9.95,-2) {$*$};
    \draw[dotted] (9.95,-2) .. controls (9.5,-2.5) .. (8.8,-2.3);

    \draw[dotted,<->] (1.5,0) .. controls (1.25,-1) and (8.25,-1) .. (8.5,0);

    \draw (9.9,3.8) arc (0:360:0.6);
    \node at (8.85,3.4)[circle,fill,inner sep=1pt]{};
    \node at (8.7,3.3) {$0_{1}$};
    \node at (9.75,3.4)[circle,fill,inner sep=1pt]{};
    \node at (10.1,3.3) {$0_{2}$};
    
    \draw (9.8,4.9) arc (0:360:0.5);
    \node at (9.3,4.9){$*$};
    \node at (9.3,5.4)[circle,fill,inner sep=1pt]{};
    \node at (9.3,5.6) {$0_{3}$};

    \draw (11.25,4.5) arc (0:360:0.5);
    \node at (10.75,4.5)[circle,fill,inner sep=1pt]{};
    \node at (10.25,4.5) {$*$};
    \draw[dotted] (10.25,4.5) .. controls (9.9,4.8) .. (9.3,4.9);

    \draw (9.9+0.11, 6.8) arc (0:360:0.6);
    \node at (8.85+0.11, 6.4)[circle,fill,inner sep=1pt]{};
    \node at (8.7+0.11, 6.3) {$0_{1}$};
    \node at (9.75+0.11, 6.4)[circle,fill,inner sep=1pt]{};
    \node at (10.1+0.11, 6.3) {$0_{2}$};
    \node at (9.3+0.11, 7.4) {$*$};
    \draw[dotted] (9.3+0.11, 7.4) .. controls (9.3+0.11, 7.8) .. (9.5+0.11, 8);
    
    \draw (10+0.11, 8) arc (0:360:0.5);
    \node at (9.5+0.11, 8){$*$};

    \draw (11.25+0.11, 6.8) arc (0:360:0.5);
    \node at (10.75+0.11, 6.8)[circle,fill,inner sep=1pt]{};
    \node at (10.25+0.11, 6.8) {$*$};
    \draw[dotted] (10.25+0.11, 6.8) .. controls (10.3+0.11, 7.55) .. (10.5+0.11, 8);

    \draw (11+0.11, 8) arc (0:360:0.5);
    \node at (10.5+0.11, 8){$*$};
    \node at (10.5+0.11, 8.5)[circle,fill,inner sep=1pt]{};
    \node at (10.5+0.11, 8.8) {$0_{3}$};

   \draw (4.4+0.2+0.59, 5.6-0.5) arc (0:360:0.6);
    \node at (3.35+0.2+0.59, 5.2-0.5)[circle,fill,inner sep=1pt]{};
    \node at (3.2+0.2+0.59, 5.1-0.5) {$0_{1}$};
    \node at (4.25+0.2+0.59, 5.2-0.5)[circle,fill,inner sep=1pt]{};
    \node at (4.6+0.2+0.59, 5.1-0.5) {$0_{2}$};

    \node at (3.8+0.2+0.59, 6.2-0.5) {$*$};
    \draw[dotted] (3.8+0.2+0.59, 6.2-0.5) .. controls (3.8+0.2+0.59, 6.6-0.5) .. (4+0.2+0.59, 6.8-0.5);
    
    \draw (4.5+0.2+0.59, 6.8-0.5) arc (0:360:0.5);
    \node at (4+0.2+0.59, 6.8-0.5){$*$};
    \node at (4+0.2+0.59, 7.3-0.5)[circle,fill,inner sep=1pt]{};
    \node at (4+0.2+0.59, 7.5-0.5) {$0_{3}$};

    \draw (5.75+0.2+0.59, 5.6-0.5) arc (0:360:0.5);
    \node at (5.25+0.2+0.59, 5.6-0.5)[circle,fill,inner sep=1pt]{};
    \node at (4.75+0.2+0.59, 5.6-0.5) {$*$};
    \draw[dotted] (4.75+0.2+0.59, 5.6-0.5) .. controls (4.8+0.2+0.59, 6.35-0.5) .. (5+0.2+0.59, 6.8-0.5);
    
    \draw (5.5+0.2+0.59, 6.8-0.5) arc (0:360:0.5);
    \node at (5+0.2+0.59, 6.8-0.5){$*$};

    \draw (7.4,7.6) arc (0:360:0.6);
\node at (6.35,7.2)[circle,fill,inner sep=1pt]{};
\node at (6.2,7.1) {$0_{1}$};
\node at (7.25,7.2)[circle,fill,inner sep=1pt]{};
\node at (7.6,7.1) {$0_{2}$};

\node at (6.8,8.2) {$*$};
\draw[dotted] (6.8,8.2) .. controls (6.8,8.6) .. (7.25,8.8);

\node at (7.25,8.8){$*$};
\node at (7.4,9.4)[circle,fill,inner sep=1pt]{};
\node at (7.4,9.6) {$0_{3}$};

\draw (8.75,7.6) arc (0:360:0.5);
\node at (8.25,7.6)[circle,fill,inner sep=1pt]{};

\node at (7.75,7.6) {$*$};
\draw[dotted] (7.75,7.6) .. controls (7.5,8.35) .. (7.5,8.8);

\draw (8,8.8) arc (0:360:0.6);
\node at (7.5,8.8){$*$};

\draw (7.3,10.1) .. controls (9.4,8.2) and (9.4,7.5) .. (7.3,5.6);
\draw (7.3,10.1) .. controls (5.2,8.2) and (5.2,7.5) .. (7.3,5.6);

\draw (2.1,10.1) .. controls (4.2,8.2) and (4.2,7.5) .. (2.1,5.6);
\draw (2.1,10.1) .. controls (0.0,8.2) and (0.0,7.5) .. (2.1,5.6);

  \draw (4.4+0.2-2.01, 5.6+1.91) arc (0:360:0.6);
    \node at (3.35+0.2-2.01, 5.2+1.91)[circle,fill,inner sep=1pt]{};
    \node at (3.2+0.2-2.01, 5.1+1.91) {$0_{1}$};
    \node at (4.25+0.2-2.01, 5.2+1.91)[circle,fill,inner sep=1pt]{};
    \node at (4.6+0.2-2.01, 5.1+1.91) {$0_{2}$};

    \node at (3.8+0.2-2.01, 6.2+1.91) {$*$};
    \draw[dotted] (3.8+0.2-2.01, 6.2+1.91) .. controls (3.5+0.2-2.01, 6.6+1.91) .. (3.8-2.01, 7+1.91);
    
    \draw (4.6-2.01, 7+1.91) arc (0:360:0.6);
    \node at (3.8-2.01, 7+1.91){$*$};
    \node at (4-2.01, 7.6+1.91)[circle,fill,inner sep=1pt]{};
    \node at (4-2.01, 7.8+1.91) {$0_{3}$};

    \node at (4.2-2.01, 7+1.91)[circle,fill,inner sep=1pt]{};

        \draw (4.4+0.2+0.22, 5.6+3.61) arc (0:360:0.6);
    \node at (3.35+0.2+0.22, 5.2+3.61)[circle,fill,inner sep=1pt]{};
    \node at (3.2+0.2+0.22, 5.1+3.61) {$0_{1}$};
    \node at (4.25+0.2+0.22, 5.2+3.61)[circle,fill,inner sep=1pt]{};
    \node at (4.6+0.2+0.22, 5.1+3.61) {$0_{2}$};

    \node at (3.8+0.2+0.22, 6.2+3.61) {$*$};
    \draw[dotted] (3.8+0.2+0.22, 6.2+3.61) .. controls (3.5+0.2+0.22, 6.4+3.61) .. (3.5+0.22, 6.7+3.61);
    
    \draw (4.0+0.22, 6.7+3.61) arc (0:360:0.5);
    \node at (3.5+0.22, 6.7+3.61){$*$};
    \node at (3.5+0.22, 7.2+3.61)[circle,fill,inner sep=1pt]{};
    \node at (3.5+0.22, 7.4+3.61) {$0_{3}$};

    \node at (4.5+0.22, 6.7+3.61)[circle,fill,inner sep=1pt]{};

    \draw (5.0+0.22, 6.7+3.61) arc (0:360:0.5);

    \draw (0 +1.03, 3.8) arc (0:360:0.6);
\node at (-1.05 +1.03,3.4)[circle,fill,inner sep=1pt]{};
\node at (-1.2 +1.03,3.3) {$0_{1}$};
\node at (-0.15 +1.03,3.4)[circle,fill,inner sep=1pt]{};
\node at (0.1 +1.03,3.3) {$0_{2}$};

\draw (-0.1 +1.03,4.9) arc (0:360:0.5);
\node at (-0.6 +1.03,4.9)[circle,fill,inner sep=1pt]{};
\node at (-0.6 +1.03,5.4)[circle,fill,inner sep=1pt]{};
\node at (-0.6 +1.03,5.6) {$0_{3}$};

   \draw (4.4+0.2-4.43, 5.6+2.02) arc (0:360:0.6);
    \node at (3.35+0.2-4.43, 5.2+2.02)[circle,fill,inner sep=1pt]{};
    \node at (3.2+0.2-4.43, 5.1+2.02) {$0_{1}$};
    \node at (4.25+0.2-4.43, 5.2+2.02)[circle,fill,inner sep=1pt]{};
    \node at (4.6+0.2-4.43, 5.1+2.02) {$0_{2}$};

    \node at (3.8+0.2-4.43, 6.2+2.02) {$*$};
    \draw[dotted] (3.8+0.2-4.43, 6.2+2.02) .. controls (3.5+0.2-4.43, 6.4+2.02) .. (3.5-4.43, 6.7+2.02);
    
    \draw (4.0-4.43, 6.7+2.02) arc (0:360:0.5);
    \node at (3.5-4.43, 6.7+2.02){$*$};
    \node at (4.5-4.43, 7.2+2.02)[circle,fill,inner sep=1pt]{};
    \node at (4.5-4.43, 7.4+2.02) {$0_{3}$};

    \node at (4.5-4.43, 6.7+2.02)[circle,fill,inner sep=1pt]{};

    \draw (5.0-4.43, 6.7+2.02) arc (0:360:0.5);

    \draw[dotted, <->] (0.7+0.75,2.8) .. controls (1,5) .. (1.4,6.3);

    \draw[dotted, <->] (-0.7+9.2,2.8) .. controls (-1+9.2,5) .. (-1.4+9.2,6.1);

    \draw[dotted, <->] (3.7,7.8) -- (5.7,7.8);

    \draw (3.65+1.05-7.11, 5.05-3.41) .. controls ({3.65 + 0.866*(4.7-3.65) - 0.500*(4.1-5.05)+1.05-7.11}, {5.05 + 0.500*(4.7-3.65) + 0.866*(4.1-5.05)-3.41}) and ({3.65 + 0.866*(4.7-3.65) - 0.500*(3.75-5.05)+1.05-7.11}, {5.05 + 0.500*(4.7-3.65) + 0.866*(3.75-5.05)-3.41}) .. ({3.65 + 0.866*(3.65-3.65) - 0.500*(2.8-5.05)+1.05-7.11}, {5.05 + 0.500*(3.65-3.65) + 0.866*(2.8-5.05)-3.41});
    \draw (3.65+1.05-7.11, 5.05-3.41) .. controls ({3.65 + 0.866*(2.6-3.65) - 0.500*(4.1-5.05)+1.05-7.11}, {5.05 + 0.500*(2.6-3.65) + 0.866*(4.1-5.05)-3.41}) and ({3.65 + 0.866*(2.6-3.65) - 0.500*(3.75-5.05)+1.05-7.11}, {5.05 + 0.500*(2.6-3.65) + 0.866*(3.75-5.05)-3.41}) .. ({3.65 + 0.866*(3.65-3.65) - 0.500*(2.8-5.05)+1.05-7.11}, {5.05 + 0.500*(3.65-3.65) + 0.866*(2.8-5.05)-3.41});

    \draw (3.65+8.84, 5.05-3.41) .. controls ({3.65 + 0.866*(4.7-3.65) + 0.500*(4.1-5.05)+8.84}, {5.05 - 0.500*(4.7-3.65) + 0.866*(4.1-5.05)-3.41}) and ({3.65 + 0.866*(4.7-3.65) + 0.500*(3.75-5.05)+8.84}, {5.05 - 0.500*(4.7-3.65) + 0.866*(3.75-5.05)-3.41}) .. ({3.65 + 0.866*(3.65-3.65) + 0.500*(2.8-5.05)+8.84}, {5.05 - 0.500*(3.65-3.65) + 0.866*(2.8-5.05)-3.41});
    \draw (3.65+8.84, 5.05-3.41) .. controls ({3.65 + 0.866*(2.6-3.65) + 0.500*(4.1-5.05)+8.84}, {5.05 - 0.500*(2.6-3.65) + 0.866*(4.1-5.05)-3.41}) and ({3.65 + 0.866*(2.6-3.65) + 0.500*(3.75-5.05)+8.84}, {5.05 - 0.500*(2.6-3.65) + 0.866*(3.75-5.05)-3.41}) .. ({3.65 + 0.866*(3.65-3.65) + 0.500*(2.8-5.05)+8.84}, {5.05 - 0.500*(3.65-3.65) + 0.866*(2.8-5.05)-3.41});

     \draw (3.65+1.90, 5.05-3.41) .. controls ({3.65 + 0.866*(4.7-3.65) - 0.500*(4.1-5.05)+1.90}, {5.05 + 0.500*(4.7-3.65) + 0.866*(4.1-5.05)-3.41}) and ({3.65 + 0.866*(4.7-3.65) - 0.500*(3.75-5.05)+1.90}, {5.05 + 0.500*(4.7-3.65) + 0.866*(3.75-5.05)-3.41}) .. ({3.65 + 0.866*(3.65-3.65) - 0.500*(2.8-5.05)+1.90}, {5.05 + 0.500*(3.65-3.65) + 0.866*(2.8-5.05)-3.41});
    \draw (3.65+1.90, 5.05-3.41) .. controls ({3.65 + 0.866*(2.6-3.65) - 0.500*(4.1-5.05)+1.90}, {5.05 + 0.500*(2.6-3.65) + 0.866*(4.1-5.05)-3.41}) and ({3.65 + 0.866*(2.6-3.65) - 0.500*(3.75-5.05)+1.90}, {5.05 + 0.500*(2.6-3.65) + 0.866*(3.75-5.05)-3.41}) .. ({3.65 + 0.866*(3.65-3.65) - 0.500*(2.8-5.05)+1.90}, {5.05 + 0.500*(3.65-3.65) + 0.866*(2.8-5.05)-3.41});

    \draw (3.65+1.05-0.16, 5.05-3.41) .. controls ({3.65 + 0.866*(4.7-3.65) + 0.500*(4.1-5.05)+1.05-0.16}, {5.05 - 0.500*(4.7-3.65) + 0.866*(4.1-5.05)-3.41}) and ({3.65 + 0.866*(4.7-3.65) + 0.500*(3.75-5.05)+1.05-0.16}, {5.05 - 0.500*(4.7-3.65) + 0.866*(3.75-5.05)-3.41}) .. ({3.65 + 0.866*(3.65-3.65) + 0.500*(2.8-5.05)+1.05-0.16}, {5.05 - 0.500*(3.65-3.65) + 0.866*(2.8-5.05)-3.41});
    \draw (3.65+1.05-0.16, 5.05-3.41) .. controls ({3.65 + 0.866*(2.6-3.65) + 0.500*(4.1-5.05)+1.05-0.16}, {5.05 - 0.500*(2.6-3.65) + 0.866*(4.1-5.05)-3.41}) and ({3.65 + 0.866*(2.6-3.65) + 0.500*(3.75-5.05)+1.05-0.16}, {5.05 - 0.500*(2.6-3.65) + 0.866*(3.75-5.05)-3.41}) .. ({3.65 + 0.866*(3.65-3.65) + 0.500*(2.8-5.05)+1.05-0.16}, {5.05 - 0.500*(3.65-3.65) + 0.866*(2.8-5.05)-3.41});

    \draw[dotted, <->] (2.5,0.6) -- (3.2,0.6);

    \draw[dotted, <->] (4.7,0.6) -- (5.4,0.6);

    \draw[dotted, <->] (6.9,0.6) -- (7.6,0.6);

     \draw[dotted, <->] (10.45,0.6) -- (11.15,0.6);

    \draw[dotted, <->] (-1.1,0.6) -- (-0.4,0.6);

    \draw[dotted, <->] (2.6,2) .. controls (5,2.5) ..  (7.4,2);

    \draw[dotted, <->] (-2.4,0.1) .. controls (-4.4,-4.5) and (14.5,-4.5) .. (12.5,0.1);

    \draw[dotted, <->] (-0.4, 2.2) .. controls (-4,5) and (14.05,5) .. (10.45,2.2);

\end{tikzpicture} 
        \caption{An example of gluing $2$-dimensional moduli spaces by point insertion. We omit some twists of markings because they are all equal to $0$. %Interestingly, if we glue along identified boundaries the resulting space is a sphere.
        }
        \label{fig:point insertion for r=2 k=3 l=1}
    \end{figure}
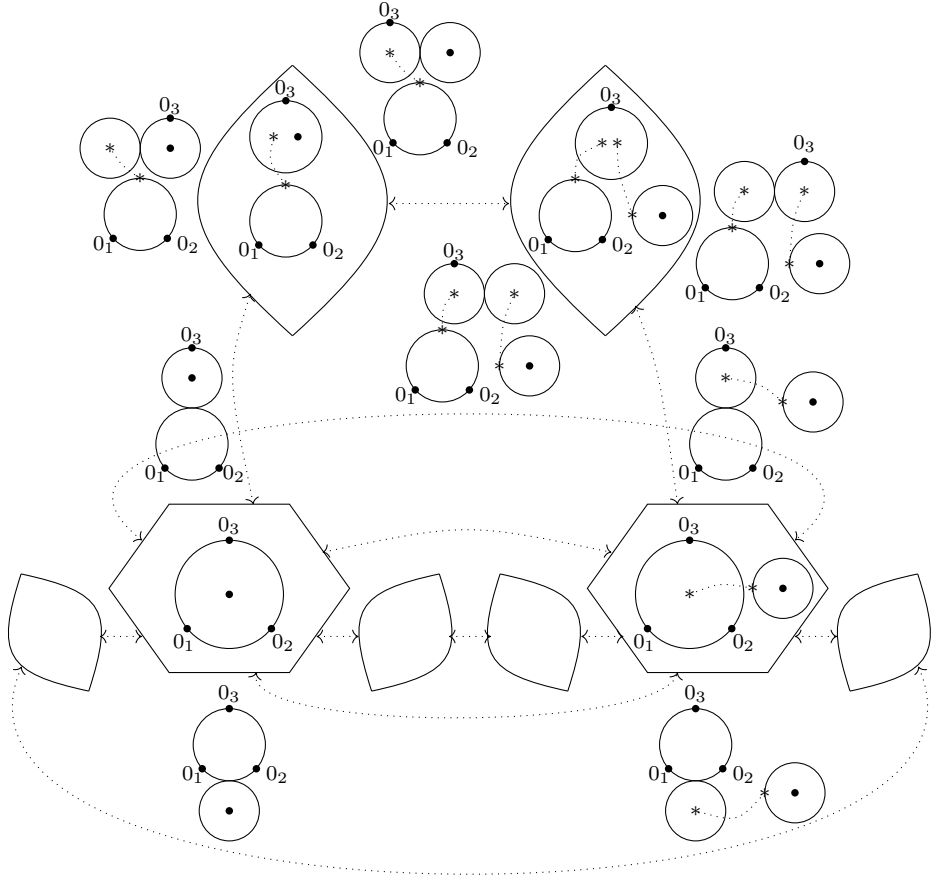

\end{ex} 

\subsection{The open FJRW bundles and relative cotangent lines}
\subsubsection{The Witten bundles and their properties}\label{subsec:WittenBundle}
We will first define the Witten bundle in the closed case, and then use this to define the open Witten bundle. Fix $g,n,W$ and possibly $\mathfrak{h}.$
Define the \emph{dimension jump locus} as the subspace of the corresponding moduli space consisting of points contained in the support of at least one of the sheaves $R^0\pi_*\mathcal{S}_i$.
On the complement to this locus we define the Witten and FJRW bundles by\[\cW^c_{i} = (R^1\pi_*\mathcal{S}_i)^{\vee},~\cW^c=\bigoplus_{i\in[a]}\cW^c_i\] where $\pi: \mathcal{C} \rightarrow \M_{g,n}^{W}$ is the universal family and $\mathcal{S}_i \rightarrow \mathcal{C}$ are the universal twisted spin bundles.
\begin{rmk}\label{rmk:when_bundle}
When $g=0$, a simple degree computation shows that whenever at most one twist equals $-1$ and all other twists are in $\{0,\ldots,r_i-1\}$, then the dimension jump locus is empty and $\cW_i$ is a
complex vector bundle.
In higher genus $\cW_i$ is a vector bundle outside of the dimension jump locus, and in $g=1$ this locus is easily characterized.
In all cases the fiber of the Witten bundle is
\[H^1(C, S_i)^{\vee} \cong H^0(C, J_i),\]
and is of complex rank \begin{equation}\label{eq:graded_closed_rank}\frac{\sum_{j=1}^n \tw_i(z_j) +(g-1) (r_i-2)}{r_i}.\end{equation}
\end{rmk}

Moving to the open case, recall the universal involutions
\[\phi: \mathcal{C} \rightarrow \mathcal{C},~~\widetilde{\phi}_i: \cS_i \rightarrow \cS_i.\]
They induce involutions on $\mathcal{J}_i:= \mathcal{S}^{\vee}_i \otimes \omega_{\pi}$ hence also on $R^1\pi_*\mathcal{S}_i^\vee\simeq R^0\pi_*\mathcal{J}_i$.
Define
\[\mathcal{W}^{o}_i := (R^0\pi_*\mathcal{J}_i)_+ =(R^1\pi_*\mathcal{S}_i)^{\vee}_-\]
to be the real vector bundle of $\widetilde{\phi}_i$-invariant sections of $J_i$. Note here the second equality uses Serre duality, under which invariant sections become anti-invariant.
The \emph{real} rank of $\mathcal{W}^{o}_i$ is seen to be
\begin{equation}\label{eq:graded_open_rank}
\frac{2\sum_{j=1}^{l} \tw_i(z_j)+\sum_{j=1}^{l} \tw_i(x_j)+(g-1)(r_i-2)}{r_i}.
\end{equation}
Indeed, if we define $(R^0\pi_*\mathcal{J}_i)_-$ as the space of anti-invariant sections, then 
\[\dim_\R(R^0\pi_*\mathcal{J}_i)_+=\dim_\R(R^0\pi_*\mathcal{J}_i)_-,\]since multiplication by $i$ takes one space isomorphically onto the other.
And \[\dim_\R(R^0\pi_*\mathcal{J}_i)_++\dim_\R(R^0\pi_*\mathcal{J}_i)_-=\dim_\R(R^0\pi_*\mathcal{J}_i),\]since $R^0\pi_*\mathcal{J}_i$ decomposes as the direct sum of the two eigenspaces for eigenvalues $\pm1$ of the lifted involution.

Write $\cW^o=\bigoplus\cW^o_i$. We refer to $\cW_i^o$ as the \emph{open Witten bundle} and to $\cW^o$ as the \emph{open FJRW bundle}.
We will usually omit the superscripts $c,o$ and understand from context if we are dealing with closed or open surfaces.

We now describe some properties of the Witten bundles.
\begin{itemize}
\item Let $\Gamma$
be a pre-graded $W$-spin graph. If $\Gamma'$ is obtained from $\Gamma$ by forgetting all tails $t$ with $(\tw(t),\alt(t))=(\vec{0},\vec{0}),$
then 
\begin{equation}\label{eq:forgetful_witten_0}\text{For}_{\text{non-alt}}^*((\cW_i)_{\Gamma'})\backsimeq (\cW_i)_{\Gamma}.\end{equation}
canonically for every $i\in[a].$
\item The Witten bundles satisfy decomposition properties along nodal strata. We describe them in an informal way, referring the reader to \cite[Proposition 2.14]{GKT2}, for a more accurate description of the $g=0$ case, as well as the treatment of internal edges, and to \cite{TZ3} for the $g=1$ statement.
Let $\Gamma$ be a $W$-graph with a grading, $e$ an edge of $\Gamma,$ and $i\in[a].$ Write $\Gamma_1,\Gamma_2$ for the two graphs with a lifting obtained by detaching $\Gamma$ along $e.$ Then if $e$ is a boundary NS node for the $i$th spin structure, then the $i$th Witten bundle $\cW_i$ on $\M_\Gamma^{W}$ 
satisfies
\begin{equation}
    \label{eq:NS_bdry_decomp}
    \cW_i|_{\M_\Gamma^{W}}\simeq\cW_i^1\boxplus\cW_i^2
\end{equation}
where $\cW_i^j$ for $j=1,2$ are the corresponding Witten bundles on the moduli spaces of $\Gamma_1,\Gamma_2.$ If $e$ is a Ramond boundary edge then, with the same notations,
\begin{equation}
    \label{eq:Ramond_bdry_decomp}
    0\to\cW_i|_{\M_\Gamma^{W}}\to\cW_i^1\boxplus\cW_i^2\to \mathcal{T}_+ \to 0,
\end{equation}is exact, where $\mathcal{T}_+$ is a trivial real line bundle, the maps to $\cW_i^1\boxplus\cW_i^2$ are projections, and the map to $\mathcal{T}_+$ is the difference of evaluation maps at the Ramond half-edges. Moreover, both ~\eqref{eq:NS_bdry_decomp} and ~\eqref{eq:Ramond_bdry_decomp} hold for $\M_{\Gamma}^{(W,\mathfrak{h})}$.

\item If $\Gamma$ is an $\MU$-connected graph, then the FJRW bundle on the corresponding moduli space is naturally isomorphic to the direct sum of the FJRW bundles on the moduli spaces corresponding to the connected components of $\Gamma$ after detaching all $\MU$-edges. 
\item Fix $i\in[a]$, $C\in\M^{W}_{g,k,l}$ (or $C\in\M^{W,\mathfrak{h}}_{g,k,l}$)  and $p\in C.$ Then there is an \emph{evaluation map} $(\ev_i)_p:\cW_i|_C\to (J_i)_p.$ Moreover, if $p\in C^\phi$ then the image of this map lies in $(J_i)_p^{\widetilde{\phi}}.$ 
\end{itemize}

\subsubsection{The relative cotangent lines}
As in the closed case, we can define $\psi_i$ classes for internal marked points. To do so, for each $i \in [l]$, we consider the cotangent line bundle $\mathbb{L}_i$ that is defined on the moduli space of stable marked surfaces with boundary, whose fiber over $(C, \phi, \Sigma, \{z_i\}, \{x_j\})$ is the cotangent line $T^*_{z_i}\Sigma$ at the interior marked point $z_i$. This definition extends to the FJRW case either by pull-back via the maps forgetting the extra data, or directly with the same definition.

\begin{obs}\label{isomorphism forgetting non-alt for descendents}
Let $\Gamma$ be a $W$ or $(W,\mathfrak{h})$-spin graph with a lifting.
\begin{enumerate}[(i)]
\item Let $i$ be a marking of an internal tail of a connected component $\Lambda$ in the graph  obtained by detaching all edges of $\Gamma$ (and possibly $\MU$-edges of $\Gamma$). Then
\begin{equation}\label{eq: pi pullback for descendent}
    \mathbb{L}_i^{\Gamma} = \Pi^*\mathbb{L}_i^{\Lambda},
\end{equation}
where $\Pi: \oCM^{W}_{\Gamma} \rightarrow \oCM^{W}_{\Lambda}$ is the
composition of the detaching map $\Detach_{E(\Gamma)}:\oCM^{W}_{\Gamma}
\rightarrow \oCM^{W}_{\detach(\Gamma)}$ 
with the projection to the factor $\oCM^{W}_{\Lambda}$. The similar statement holds for \TZ-$(W, \mathfrak{h})$-surfaces with boundary.
\item If $\Gamma' = \text{for}_{\text{non-alt}}(\Gamma)$, then there exists a canonical morphism
\[t_{\Gamma} : \text{For}_{\text{non-alt}}^*\mathbb{L}_i^{\Gamma'} \rightarrow \mathbb{L}_i^{\Gamma}.\]
This morphism vanishes identically on the strata
where the component containing $z_i$ is contracted by the forgetful map. Away
from these strata, $t_\Gamma$ is an isomorphism. 
\item $\mathbb{L}_i^{\Gamma}$ is canonically oriented as a complex orbifold line bundle.
\end{enumerate}
\end{obs}
\noindent See \cite[Section 3.5]{PST14} for further discussion. While $\mathbb{L}_i^{\Gamma}$ is canonically oriented, we are not so lucky for the FJRW bundle and the moduli space. However, we have the following theorem:

\begin{thm}\label{thm:or}
Let $\M\subseteq\M_{g,k,l}^{W}$ (or $\M\subseteq\M_{g,k,l}^{(W,\mathfrak{h})}$)  be the complement of the dimension jump locus inside one of the moduli spaces defined in Proposition~\ref{prop:moduli_orbi_corners}, let $\cW$ be the corresponding open FJRW bundle and let $E$ be a direct sum of $\cW$ and some copies of relative cotangent lines. Then $E\rightarrow \M$ carries a canonical relative orientation.
\end{thm}

The case of $g>0$ \PST-theory was proven in \cite{Tes15}. In this case there is no FJRW-bundle and the claim is that the moduli space of graded $2$-spin surfaces is canonically oriented. The orientation was constructed by decomposing the space into a union of cells indexed by ribbon graphs, writing an orientation expression for each, and showing that the orientations of the different chambers glue.

The proof for $r$-spin theories appeared in \cite{BCT1,TZ1} and was shown by constructing explicit frames for the Witten bundle. The higher Fermat cases studied in \cite{GKT2,GKT3,TZ1} involved careful analysis of the orientations of their $r$-spin components.

All these works also included a detailed study of how these orientations behave under degenerations, which was crucial for the study of relations between intersection numbers.

\section{Open intersection numbers}
As explained above, to define intersection theories one has to impose boundary conditions on sections of the vector bundles. Intuitively, the intersection numbers are the weighted signed zero counts of a transverse extension of the section on the boundary to the interior. By boundary conditions we mean putting restrictions on the behaviour of such sections on topological boundaries of real codimension $1$ and more. If the section on the boundary is nowhere vanishing, then the section defines an intersection number by integrating the \emph{relative Euler class}, see the appendices of \cite{PST14,BCT2}. 

In order for the resulting intersection numbers to be interesting, and computable, these boundary conditions must be geometrically meaningful. Interestingly, the geometry of $W$-spin surfaces with a grading leads us to such \emph{canonical} boundary conditions.
While the precise definitions are somehow technical, describing them for real codimension $1$ boundaries is more intuitive, and captures the main ideas, so we shall restrict to that, and add comments on higher codimension corners.
It will be easiest to start with \PST~theory, which only involves relative cotangent lines, and introduces boundary conditions on them. We next move to the \BCT~theory which adds the Witten bundle and boundary conditions on it and then to \GKT, generalizing \BCT. Finally, we discuss to \TZ~theory, which uses some of the boundary conditions of \BCT, but replaces others.
We will only consider boundary conditions which are invariant under re-labeling marked points. This imposition makes us work with multi-valued sections (multisections), but we will still usually use the word ``section'', unless we want to stress the multi-valuedness.
\subsection{\PST}
In this case, one can compute that the rank of any Witten bundle is zero, thus we only consider $\psi_i$-classes (at internal markings). There are two types of real codimension $1$ boundaries, those which correspond to strata with a single boundary node, and those which correspond to strata with a contracted boundary node.
Let $\Gamma$ be a graded graph which corresponds to a boundary of the former type, and let $e$ be the unique boundary edge.
Let $\Gamma_1,\Gamma_2$ be obtained from $\Gamma$ by \emph{first detaching} the edge $e$ and \emph{then forgetting the resulting non-alternating half-edge}.
Let $i$ be an internal tail which belongs to $\Gamma_j.$ Then a section $s$ of $\CL_i$ is \emph{canonical at $\Gamma$} if there exists a section $s'$ of $\CL_i$ on $\M_{\Gamma_j}$ such that\begin{equation}\label{eq:forgetful_bc_psi}s|_{\M_\Gamma}=\Pi^*s',\end{equation}
where $\Pi$ is the projection from \eqref{eq: pi pullback for descendent} induced from detaching and projecting to $\M_{\Gamma_j}.$ Note that we have used Observation~\ref{isomorphism forgetting non-alt for descendents} and the fact, which can be checked directly, that under the above process $\Gamma_j$ cannot lose stability. We call these boundary conditions \emph{forgetful boundary conditions}.

Suppose now that $\Gamma$ is a graded graph with a single contracted boundary tail $h.$ This may happen, in the \PST~case, only at $g>0.$ The section is \emph{canonical at $\Gamma$} if it is independent of the choice of grading at the contracted boundary which corresponds to $h.$
The boundary conditions are canonical if they are canonical at each codimension $1$ boundary stratum.
\begin{rmk}
A few comments are in order:
\begin{itemize}
    \item Continuity implies that also $s'$ in the above definition must be canonical.
    \item A notion of canonicity holds in higher codimension, that is, in the corners of the moduli space. To define this, we use the notion of a \emph{base}. The base is the moduli space corresponding to the graph obtained by taking any graph $\Gamma$ and detaching all boundary edges, forgetting the non-alternating half-edges,\footnote{We forget only non-alternating half-edges with trivial twists in the \BCT~ and \GKT~ versions of the base.} forgetting gradings at contracted boundaries and removing all connected components but the one which contains the $i$th marking. We impose that the section $s|_{\M_\Gamma}$ is pulled back from a section defined on the base.
    \item One can show (see \cite[Remark 3.5]{PST14}) that usually there are no non-zero sections which are canonical, even at $g=0,$ and one has to work with multi-sections.
\end{itemize}
\end{rmk}
\begin{thm}\label{thm:pst_nums_well_def}
   Assume $d_1,\ldots,d_l\in\Z_{\geq 0}$ satisfy $\sum d_i=2l+k+3g-3$. Then there exists global multisections for $\bigoplus_{i\in[l]}\CL_i^{\oplus d_i}\to\M_{g,k,l}^{x^2}$ that satisfy canonical boundary conditions and are nowhere vanishing on the boundary of $\M_{g,k,l}^{x^2}$.
     Moreover, every choice of such boundary conditions induces the same intersection number, which we denote by $\langle\tau^0_{d_1}\ldots\tau^0_{d_l}\sigma^k\rangle^{(x^2,\mu_2),o}.$
\end{thm}

This theorem was proven for $g=0$ in \cite{PST14}, in $g>0$ the proof was found in \cite{ST_unpublished}, and an alternative proof appears in \cite{Tes15}.

The idea behind the proof is as follows. For the existence part, one constructs multisections inductively, based on the recursive nature of the boundary conditions. The intuitive reason for the well-definedness of the intersection numbers is that the forgetful boundary conditions effectively mean the boundary section is pulled back from the base which is dimension at least two less than the original moduli space. One can connect any two such choices of boundary section by a transverse homotopy which also satisfies the same forgetful boundary conditions, and is effectively pulled back from a space of codimension $1$. Lastly, it can be shown that there exists such a transverse homotopy between sections of a rank $m$ bundle over a base space of dimension $\leq m-1$ that is nowhere vanishing. 

It is a general fact that, given a transverse homotopy between two multisections, the difference of the intersection numbers with respect to each multisection can be computed via the intersection numbers of the homotopy on the boundary (see \cite[Lemma 3.55]{PST14}). Hence, from the above argument, the two boundary sections induce the same intersection number.

\subsection{\BCT}
Again we start with $\Gamma$ which has a single boundary edge $e,$ and let $\Gamma_1,\Gamma_2$ be as above. We will only describe the $g=0$ case.
First, if one half-edge of $e$ has $\tw=\alt=0,$ then for $\psi_i$ we use the same forgetful boundary conditions as in \eqref{eq:forgetful_bc_psi}.
A section $s$ of the Witten bundle is canonical at $\Gamma$ if
\begin{equation}\label{eq:forgetful_bc_Witten}s|_{\M_\Gamma}=\Pi_1^*s_1\boxplus
\Pi^*_2s_2,\end{equation}
where $s_j$ is a section of the Witten bundle on $\cW\to\M_{\Gamma_j},$ and $\Pi_j$ is the projection, where we use the decomposability properties of the Witten
bundle mentioned above, and \eqref{eq:forgetful_witten_0}. Again these boundary conditions are called forgetful boundary conditions.

However, there exists boundary edges where neither half-edge has the property that $\tw = \alt = 0$, so we need additional boundary conditions. The new ingredient of \BCT\, theory is the \emph{positivity boundary conditions} which we now informally describe. See \cite[Section 3]{BCT2} for a precise definition. We say such a boundary edge is \emph{positive} with respect to a spin bundle when one half-edge has $\tw>0$ and $\alt = 0$.
Let $\Gamma$ be a graph with either a single positive boundary edge or a contracted boundary tail. 
In the former case, let $\Gamma_1,\Gamma_2$ be the two components of the detaching of $\Gamma$ at $e,$ and $h_1,h_2$ the resulting half-edges. At least one of them is non-alternating, without loss of generality $h_1.$ For any $C\in\M_{\Gamma_1}^{W}$ the grading can be extended to the special point $n_1$ corresponding to $h_1.$ We say that $s$ is positive at $C$ if 
$\ev_{n_1}(s|_C)\in J_{n_1}^{\tilde\phi}$
is \emph{positive with respect to the grading}.

Note that if $e$ is Ramond then the exact sequence \eqref{eq:Ramond_bdry_decomp} and properties of the grading show it does not matter on which side we evaluate.
Similarly, if $e$ is a contracted boundary and $C\in\M_{\Gamma_1}^{W}$ then $s$ is positive at $C$ if $\ev_n(s|_C)$ is positive with respect to the grading, where $n$ is the contracted boundary node.
Note that in these two cases we do not impose boundary conditions on the $\psi_i$ classes.

Finally, if $\Gamma$ has no boundary half-edges, a section $s$ is said to be positive if for any $C\in\M_{\Gamma_1}^{W},~p\in C^\phi,$ $\ev_p(s|_C)$ is positive with respect to the grading.
Sections satisfying all these forgetful and positivity conditions are said to be \emph{canonical}.
\begin{rmk}
    Again certain comments are in place.
    \begin{itemize}
        \item First, in the $g=0$ \BCT~$r$-spin case the forgetful maps we use in order to define boundary conditions may lead to stability loss, and by Observation \ref{isomorphism forgetting non-alt for descendents} it will imply that canonical sections of $\CL_i$ vanish there.
        However, it can be easily shown that in all these cases the section of the Witten bundle is positive, and we will not need to worry about this vanishing.         \item Proving that positive boundary conditions exist for moduli points which belong to higher codimensional corners is the heart of this construction, and follows from a real algebraic geometry argument, see \cite[Proposition 3.20]{BCT2}. Moving from a pointwise existence to a global existence imposes additional difficulties, and in fact requires a relaxation of the definition of positivity, see \cite[Example 3.23]{BCT2}. There are various ways to perform this relaxation, and the road taken in \cite{BCT2,GKT2} involves removing certain boundary strata and imposing the positivity \emph{near} these boundary strata, where the evaluation at the nodes is replaced by evaluations at boundary intervals which converge to the node.
        \item At $g=1$, some of the positivity boundary conditions break. The remedy is obtained by adding new boundary conditions for the sections of $\CL_i$ in addition to the forgetful ones. See \cite{TZ3} for a construction. 
    \end{itemize}
\end{rmk}

\begin{thm}\label{thm:r-spin_nums_well_def}[\cite{BCT2} for $g=0,$ \cite{TZ3} for $g=1$]
   Consider $g\in\{0,1\}$, $a_1,\ldots, a_l,k\in\Z_{\geq 0}$ which satisfy \eqref{eq:graded_open_rank},~\eqref{eq:parity}, and  $d_1,\ldots,d_l\in\Z_{\geq 0}$ for which \[\sum d_i+\frac{2\sum a_i+(k+g-1)(r-2)}{r}=2l+k+3g-3.\]
    Then there exists a family of global multisections for $\cW\oplus\bigoplus_{i\in[l]}\CL_i^{\oplus d_i}\to\M_{g,k,\{a_1,\ldots,a_l\}}^{x^r}$ that satisfy the canonical boundary conditions and are nowhere vanishing on the boundary of $\M_{g,k,\{a_1,\ldots,a_l\}}^{x^r}$. Moreover, every choice of such a family of canonical multisections induces the same intersection number, which we denote by $\langle\tau^{a_1}_{d_1}\ldots\tau^{a_l}_{d_l}\sigma^k\rangle^{(x^r,\mu_r),o}.$
\end{thm}

\begin{definition}
The intersection numbers $\langle\tau^{a_1}_{d_1}\ldots\tau^{a_l}_{d_l}\sigma^k\rangle^{(x^r,\mu_r),o}$ from Theorem~\ref{thm:r-spin_nums_well_def} are called \emph{open $r$-spin invariants}.
\end{definition}

As explained above, the existence of canonical sections is non-trivial and is based on a delicate real algebraic geometry argument. For the well-definedness of intersection numbers, one uses the same homotopy argument sketched at the end of the previous subsection. Again one constructs homotopies which satisfy the same conditions as the sections --- forgetfulness and positivity. Boundary strata which are subject to the forgetful boundary conditions do not contribute zeroes to the homotopy by transversality and rank versus dimension arguments as above. Strata which satisfy positivity do not contribute to the zero count thanks to their positivity. 
\subsection{\GKT}
The $g=0$ sector of both of the above constructions is a special case of the more general \GKT ~construction. 
\GKT-theory studies mainly intersection numbers which correspond to \emph{rooted} $W$-spin disks, that is $W$-spin disks which have a single boundary marking, called the \emph{root}, which is alternating w.r.t every $r_i$-spin structure, and whose $i^{th}$ twist is $r_i-2.$ The remaining boundary markings are \emph{(type $i$) singly twisted}, that is, for every such point $x_j,$ for a single $i\in[a]$ the $x_j$ is alternating and its $i^{th}$ twist is $r_i-2.$ For the remaining $l\in[a]$ its $l^{th}$ twist is $0$ and it is non-alternating. These intersection numbers are those that encode the mirror symmetry. 

In \cite{GKT2} the case $a=2,~W=x_1^{r_1}+x_2^{r_2}$ is treated. In this case the boundary conditions are exactly the same as in the \BCT~ case. But there is a crucial difference. Graphs can have boundary edges where no half-edge has $\tw=\alt = (0,0)$ or is positive for either spin bundle. These are precisely the boundaries which correspond to a single boundary edge $e,$ which consists of two half-edges $h_1,h_2$ so that
\[\tw_j(h_i)=\delta_{ij}(r_j-2),\quad \alt_j(h_i)=\delta_{ij}.\]
Here, no half-edge can be forgotten. Moreover, no half-edge allows the $\cW_1$ or $\cW_2$ positivity boundary condition. 
For such boundaries we do impose that the sections of the Witten bundles and relative cotangent lines are direct sums of pullbacks of corresponding sections on the two components of the normalized moduli space, but this requirement is not enough to fix the boundary behaviour.

Because of these problematic boundaries the above homotopy argument fails, and intersection numbers induced by canonical boundary conditions \emph{do depend on the choice of the boundary conditions}; however, the resulting intersection numbers are still constrained, as the theorem below shows. We will see below that this dependence on choices is essential for mirror symmetry.

\begin{nn}\label{not: GKT moduli}
    For $W=x_1^{r_1}+x_2^{r_2}$ denote by  $\M_{0,k_1,k_2,1,\mathbf{A}}^{W,o} $ the moduli space of rooted disks with $k_i$ singly twisted points of type $i,$ and internal twists $\mathbf{A}= \{(a_i,b_i)\}_{i \in [l]}$.
\end{nn}

We denote these rooted intersection numbers by 
\begin{equation}\label{eq:rooted intersection number}
\langle\tau^{a_{1},b_{1}}_{d_1}\ldots\tau^{a_{l},b_{l}}_{d_l}\sigma_1^{k_1}\sigma_2^{k_2}\sigma_{12};s\rangle^{(W,\mu_r\times\mu_s),o},
\end{equation}
where $a_{i}$ (respectively, $b_i$) is the twist with respect to the $1$st (respectively, $2$nd) spin structure on the $i$th internal marked point, $k_i$ is the number of singly twisted points of type $i$, $d_i$ is the multiplicity of summand $\CL_i,$ and the rank of the resulting bundle equals the dimension of the moduli space. As usual,  the twists are subject to constraints given in \eqref{eq:graded_open_rank} and ~\eqref{eq:parity} to obtain a nonzero intersection number. Finally, $s$ is the canonical boundary condition, which now must appear in the notation.

For the rooted intersection numbers in ~\eqref{eq:rooted intersection number} to have interesting structure (e.g., exhibit mirror symmetry) we require the canonical boundary conditions to satisfy a so-called \emph{compatibility} condition. For example, suppose we have two moduli spaces $\M_{\Gamma}$ and $\M_{\Gamma'}$ which each have a boundary stratum, associated to graphs $\Lambda$ and $\Lambda'$ respectively, such that $\Lambda,\Lambda'$ each only a have single boundary edge. As done above, write $\Lambda_1, \Lambda_2$ (respectively, $\Lambda'_1, \Lambda_2'$) for the graphs obtained from detaching at the unique boundary edge. Since the global multisections $s_{\Gamma}$ and $s_{\Gamma'}$ on their respective FJRW bundles must pull back from their respective bases, we have $$s|_{\M_\Lambda} = \Pi^*_1 s_{\Lambda_1} \boxplus \Pi_2^* s_{\Lambda_2} \text{ and }s|_{\M_{\Lambda'}} = \Pi^*_1 s_{\Lambda'_1} \boxplus \Pi_2^* s_{\Lambda_2'},$$ just as in ~\eqref{eq:forgetful_bc_Witten}. Compatibility imposes that if the graphs $\Lambda_2$ and $\Lambda_2'$ are isomorphic, then $s_{\Lambda_2}=s_{\Lambda_2'}$.
 See \cite[Definition 3.27]{GKT2} for a precise formulation, where we absorb this compatibility into the definition of a \emph{canonical family of multisections}. Example~\ref{GKT Example of Wall Crossing} below illustrates how this situation manifests and its consequences. By a slight abuse of notation, if $\mathbf{s}^\bullet$ is a compatible family we may write $\mathbf{s}^\bullet$ instead of $s$ in the notation for intersection number.

The analog in this theory to Theorem \ref{thm:pst_nums_well_def} and Theorem \ref{thm:r-spin_nums_well_def} is

\begin{thm}\label{thm:GKT_invariants}[\cite{GKT2}]
Fix $[l]$ and take a set of twists $\{(a_i, b_i)\}_{i \in [l]}$, $d_i \in \Z_{\ge0}$. There exists global multisections of the open FJRW bundles $\cW_1\oplus\cW_2\oplus\bigoplus_{i\in[l]}\CL_i^{\oplus d_i}$ over any moduli space $\M_{0, k_1, k_2, 1, \mathbf{A}}$ (as defined in Notation~\ref{not: GKT moduli}) where $k_1, k_2\in \Z_{\ge 0}$ and $\mathbf{A} \subseteq \{(a_i, b_i)\}_{i \in [l]}$ that satisfy the forgetful, canonical and compatibility boundary conditions above. That is, these global multisections form a canonical family of multisections. Moreover, they are nowhere vanishing on the boundary of $\M_{0, k_1, k_2, 1, \mathbf{A}}$ and thus define rooted intersection numbers
$$
\langle\tau^{a_{i_1},b_{i_1}}_{d_{i_1}}\ldots\tau^{a_{i_m},b_{i_m}}_{d_{i_m}}\sigma_1^{k_1}\sigma_2^{k_2}\sigma_{12};s\rangle^{(W,\mu_r\times\mu_s),o},
$$
where $\mathbf{A} =\{(a_{i_1},b_{i_1}), \dots, (a_{i_m}, b_{i_m})\}$.

These rooted intersection numbers depend on the choice of canonical family of multisections. However, one can construct explicit polynomials of rooted intersection numbers that are independent of the choice of compatibility boundary conditions (see Notation~\ref{nn:mathcal_A}). Lastly, given any collection of values $\lambda_{k_1, k_2, d_i, \mathbf{A}} \in \Q$ which satisfy all the polynomial constraints given in Theorem~\ref{thm:A_vanishes}, there exists a canonical family of multisections $s_{\lambda}$ so that 
$$
\langle\tau^{a_{i_1},b_{i_1}}_{d_{i_1}}\ldots\tau^{a_{i_m},b_{i_m}}_{d_{i_m}}\sigma_1^{k_1}\sigma_2^{k_2}\sigma_{12};s_{\lambda}\rangle^{(W,\mu_r\times\mu_s),o} = \lambda_{k_1, k_2, d_i, \mathbf{A}}.
$$
\end{thm}

\begin{definition}
    The rooted intersection numbers
$
\langle\tau^{a_{i_1},b_{i_1}}_{d_{i_1}}\ldots\tau^{a_{i_m},b_{i_m}}_{d_{i_m}}\sigma_1^{k_1}\sigma_2^{k_2}\sigma_{12};s\rangle^{(W,\mu_r\times\mu_s),o}
$ 
from Theorem~\ref{thm:GKT_invariants} are called \emph{open FJRW invariants}.
\end{definition}

In the upcoming work \cite{GKT3} this theorem is extended to the most general Fermat polynomial, corresponding to the LG model
$$
(x_1^{r_1}+\cdots+x_a^{r_a}, \mu_{r_1} \times \cdots \times \mu_{r_a}).
$$ 
Unlike the rank $a=2$ case, for $a>2$ additional boundary conditions are required for this theorem to hold and for the intersection theory to satisfy mirror symmetry. We remark that the first barrier that will demonstrate this requirement is quite quick to see: when one takes a boundary edge with $a=2$, then one has three cases: (i) one has a positive boundary half-edge, (ii)  one boundary half-edge has $\tw = \alt =0$ for both spin bundles, or (iii) both boundary half-edges are singly twisted. When $a>2$, this is clearly no longer the case.

\subsubsection{Balanced graphs, critical boundaries, and wall-crossing}
Open FJRW invariants are nonzero only when the rank of the FJRW bundle is equal to the dimension of the underlying moduli space of $W$-orbidisks. The real rank of the open FJRW bundle is given by summing the open Witten bundle components \eqref{eq:graded_open_rank} when there are no descendents. If there are descendents, then one adds two to the real rank for each $\psi$-class, as each $\psi$-class adds one complex dimension. In~\eqref{real dim of open Moduli of discs}, we have the real dimension of the moduli space of genus $g$ orbidisks with $k$ boundary markings and $l$ internal markings. 

We can then see that we only have nonzero intersection numbers when these two numbers are equal. We call a dual graph corresponding to this situation \emph{balanced}. In fact, the twists and descendent vector factor into the computation of the rank of the open FJRW bundle, while only the genus and number of marked points show up in the dimension of the moduli space, so they must ``balance'' in order to find a nonzero intersection number. This condition boils down to particular equations. For example, in the open $r$-spin case, the equation given in Theorem~\ref{thm:r-spin_nums_well_def} determines when an open $r$-spin dual graph is balanced.

In the rank one case, we then can show that the open $r$-spin invariants do not depend on this choice by showing that there exists a homotopy between any two choices of canonical multisection that does not introduce a change in the invariants. In the rank two case, this is often impossible! The culprit for this is the existence of so-called \emph{critical boundaries}. 

Consider a smooth, balanced, rooted $W$-spin disk and allow it to degenerate to a disk with a boundary node. Let the dual graph for this boundary be $\Gamma$. The rank of the FJRW bundle stays the same, but the moduli space $\CM_{\Gamma}$ has codimension 1. We now consider a homotopy between two different systems of canonical multisections defining open intersection numbers. In effect, the homotopy adds an additional dimension to the moduli space, so one may now have nonvanishing intersection numbers. To combat this, each theory has a technical lemma providing structure to the homotopy (see, e.g., \cite[Lemma 4.11]{BCT2}, \cite[Lemma 3.39]{GKT2}) where the homotopy can be found to be nonvanishing on the boundary in many cases. However, we cannot prove it in all cases. This is where wall-crossing arises.

Critical boundaries in rank 2 occur when both the half-nodes are singly twisted and the vertex that has the root is still balanced. Here, from the ``recursiveness'' of the system of canonical sections, the real rank of the FJRW bundle and the dimension of the moduli space corresponding to the rooted vertex are equal, and the real rank of the FJRW bundle on the non-rooted vertex is one greater than its corresponding moduli space's dimension. Such unrooted vertices are called \emph{critical}, and, in rank two, how they correspond with critical boundaries is fairly straightforward \cite[Proposition 2.37 and Notation 2.38]{GKT2}. We remark that one cannot see such a phenomenon occur in rank 1 as both half-nodes cannot be singly twisted, and hence one vertex's moduli space will `lose' an additional dimension when the untwisted half-node is forgotten by passing to the base.

We now give a few examples so one can see wall-crossing occur.

\begin{ex}\label{GKT Example of Wall Crossing}
The simplest example of wall crossing comes from considering $W= x_1^4 + x_2^4$. We look at the $W$-spin disk with two internal markings with twists $(2,2)$, four boundary marked points singly twisted of type $1$, and a root. With no descendents, this  corresponds to the rooted intersection number $\langle \tau_0^{2,2}\tau_0^{2,2}\sigma_1^4\sigma_{12}; s\rangle^{(W, \mu_4 \times \mu_4),o}$. 

We depict a generic $W$-spin disk with these properties in Figure~\ref{x4y4 example}(A). One can see that the dimension of the moduli space and the rank of the FJRW bundle are both 6. There is a critical boundary found by degenerating as seen in Figure~\ref{x4y4 example}(B). After normalizing, one can see that the disk on the right has two internal markings with twist $(2,2)$, one boundary marking singly twisted of type 1, and another of type 2 (which comes from the half-node). One can check that the dimension of the moduli space is 3 and the rank of the FJRW bundle is 4, hence it is a critical disk. On the other hand, the rooted disk has a moduli space of dimension 2 and the rank of the FJRW bundle is 2.

\begin{figure}
\begin{subfigure}{.45\textwidth}
  \centering
\begin{tikzpicture}[scale=0.5]
  \draw (0,0) circle (2cm);
\node at (1,.7) {$\bullet$};
   \node at (-.35,-.4) {$\bullet$};
    
  \node (c) at (-2,0) {$\times$};

  \node (r1) at (1.73,-1) {\tiny $\bullet$};
   \node (r1) at (0.347,-1.9696) {\tiny $\bullet$};
   \node (r1) at (1.36,-1.46) {\tiny $\bullet$};
   \node (r1) at (0.347,1.9696) {\tiny $\bullet$};
\end{tikzpicture}
      \caption{Smooth, balanced open FJRW disk}
\end{subfigure}
\begin{subfigure}{.45\textwidth}
\centering
\begin{tikzpicture}[scale=0.5]
  \draw (0,0) circle (2cm);
 \draw (4,0) circle (2cm);

\node at (5,.7) {$\bullet$};
   \node at (3.35,-.4) {$\bullet$};
   
  \node (c) at (-2,0) {$\times$};
  \node (a) at (2,0) {$\bullet$};

  \node (r1) at (1.73,-1) {\tiny $\bullet$};
   \node (r1) at (0.347,-1.9696) {\tiny $\bullet$};
   \node (r1) at (5.36,-1.46) {\tiny $\bullet$};
   \node (r1) at (0.347,1.9696) {\tiny $\bullet$};
\end{tikzpicture}
      \caption{Critical boundary}
\end{subfigure}
\caption{Example of critical boundary in the case $W= x_1^4 + x_2^4$. Here the $\times$ depicts the root, all interior markings have twist $(2,2)$, and all boundary markings are singly twisted of type $1$. }\label{x4y4 example}
\end{figure}
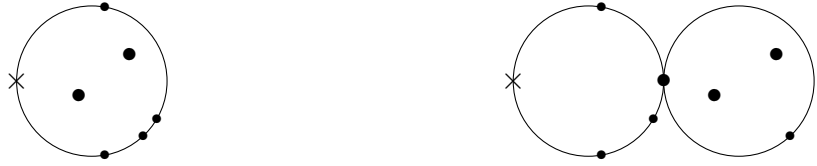

\end{ex}

A corollary to what we show in \cite{GKT2} is that, for any $c \in \Q$, there exists a canonical multisection $s_c$ so that 
$$
\langle \tau_0^{2,2}\tau_0^{2,2}\sigma_1^4\sigma_{12}; s_c\rangle^{(W, \mu_4 \times \mu_4),o} = c.
$$
This is because one can prove a lemma that allows one to build a new canonical multisection on critical disks that adds new zeros \cite[Lemma 5.16]{GKT2}.

This may seem troublesome, but it introduces a beautiful structure.  Since the canonical multisections must be built recursively on boundary strata, we can see that, as we change the canonical multisection for the critical disk, we also change canonical multisections for other balanced disks than the one above. 

In this particular example, one can see by symmetry that changing the multisection associated with this critical disk will also change the intersection number
$$
\langle \tau_0^{2,2}\tau_0^{2,2}\sigma_2^4\sigma_{12}; s_c\rangle^{(W, \mu_4 \times \mu_4),o}.
$$
Unraveling the notation, this intersection number considers the moduli space whose smooth representative is a disk that has two internal markings with twists $(2,2)$, four boundary marked points singly twisted of type $2$, and a root. A similar picture arises to that in Figure~\ref{x4y4 example} but with types of the singly twisted boundary markings flipped. One obtains the same critical disk, but this time the boundary marked point coming from the half-node is of type $1$. This changes the orientation of the gluing of the moduli space and consequently the sign of the change of intersection number coming from changing the multisection. Thus we obtain that the \emph{sum}
\begin{equation}\label{sum quantity for WC}
\langle \tau_0^{2,2}\tau_0^{2,2}\sigma_1^4\sigma_{12}; s\rangle^{(W, \mu_4 \times \mu_4),o} + \langle \tau_0^{2,2}\tau_0^{2,2}\sigma_2^4\sigma_{12}; s\rangle^{(W, \mu_4 \times \mu_4),o}
\end{equation}
is invariant with respect to the choice of system of canonical multisections. In \S\ref{sec:comp ms ih}, the quantities that are invariant will be made precise (see Notation~\ref{nn:mathcal_A}).

\subsection{\TZ}
The obstruction for allowing more general boundary twists in the \BCT-$r$-spin theory is that with twists lower than $r-2$ the positivity requirements cannot be fulfilled. In \cite{TZ2}, Tessler and Zhao overcome this problem by replacing the \emph{forgetful boundary conditions} by the \emph{point insertion boundary conditions}. The \TZ-$r$-spin theories, are defined over the moduli space of $\MU$-connected surfaces, and are labelled by $\mathfrak{h}=0,1,\ldots,\lfloor\tfrac{r}{2}\rfloor-1.$ The $\mathfrak{h}^{th}$ theory allows all internal twists of the form $(a,\ldots,a)$, and all boundary twists of the form $(b,\ldots,b)$ with $b\geq r-2\mathfrak{h}-2,~b=r\pmod 2$. 

We first consider the case with one spin bundle. On codimension $1$ boundary strata which correspond to contracted boundaries, or to a boundary edge whose non-alternating side has twist greater than $2\mathfrak{h}$, we impose the same positivity conditions.
Let $\Gamma$ be a $(W,\mathfrak{h})$ $\MU$-connected graph with a single boundary edge $e$ whose non-alternating side has twist $2a\leq 2\mathfrak{h},$ hence its alternating side has twist $r-2a-2\geq r-2\mathfrak{h}-2.$ Then $\Gamma$ is \emph{equivalent}, in the sense of $\MU$-connected graphs, to a unique graph $\Gamma'$ which has an additional $\MU$-edge between a new boundary tail twisted $(r-2a-2)$ and an internal tail twisted $a,$ see Figure~\ref{fig:rh surface} for the equivalence in the level of surfaces. For any moduli point in $\M_\Gamma$ there is a unique equivalent moduli point in $\M_{\Gamma'},$ and vice versa. Moreover, the fibers of the Witten bundle and relative cotangent lines are canonically identified.

A section $s$ is said to be \emph{canonical} at $\Gamma$ if
\begin{equation}\label{eq:point_insertion_bc}s|_{\M_{\Gamma}}=s|_{\M_{\Gamma'}}\end{equation}under the above identifications of moduli spaces and bundles.
An intuitive way to think about these \emph{point insertion boundary conditions} is that we glue moduli spaces of $\MU$-connected surfaces along boundary strata which represent equivalent objects. From this perspective, the glued boundaries cease to act as true boundaries. Thus, we need not impose boundary conditions as before and instead just need the multisections to allow gluing. 
This is exactly what \eqref{eq:point_insertion_bc} does. 

This construction extends word for word to the case of Fermat polynomials of the form $\sum_{i=1}^{2m+1}x_i^r$ and their minimal symmetry group $\mu_r.$ In this case all twists have the form $(t,\ldots,t)$, and boundary twists must satisfy $t\geq r-2\mathfrak{h}-2,~t= r \pmod{2}.$ The reason we consider only an odd number of summands $2m+1$ here is orientation: The sum of an even number of copies of Witten bundles can always be canonically oriented. A single Witten bundle is relatively canonically oriented, by Theorem \ref{thm:or}. Thus, the sum of $2m+1$ copies is also canonically relatively oriented. There is no analog for this theory, in general, for an even number of copies.

\begin{thm}\label{thm:TZ-r-spin_nums_well_def}[\cite{TZ2}]
    Let $g=0,~m\geq 0$, $a_1,\ldots, a_l,b_1,\ldots,b_k\in\Z_{\geq 0}$ which satisfy \eqref{eq:graded_open_rank},~\eqref{eq:parity}, and each $b_i\in\{r-2,r-4,\ldots,r-2\mathfrak{h}-2\}.$ Let $d_1,\ldots,d_l\in\Z_{\geq 0}$ satisfy \begin{equation}\label{eq:TZ_rk_vs_dim}\sum d_i+(2m+1)\frac{2\sum_{i=1}^l a_i+\sum_{i=1}^k b_i-(r-2)}{r}=2l+k-3,\end{equation}
    or equivalently
    \[\rank(\cW^{\oplus 2m+1}\oplus\bigoplus_{i\in[l]}\CL_i^{\oplus d_i})=\dim(\M_{g,\{b_1,\ldots,b_k\},\{a_1,\ldots,a_l\}}^{(W,\mathfrak{h})}).\]
   Then there exists a family of global multisections  for $$\cW^{\oplus 2m+1}\oplus\bigoplus_{i\in[l]}\CL_i^{\oplus d_i}\to\M_{g,\{b_1,\ldots,b_k\},\{a_1,\ldots,a_l\}}^{(W,\mathfrak{h})},$$
   that satisfy the canonical boundary conditions and are nowhere vanishing on the boundary of $\M_{g,\{b_1,\ldots,b_k\},\{a_1,\ldots,a_l\}}^{(W,\mathfrak{h})}$.
   Moreover, every choice of such family of global multisections induces the same intersection number. We denote this number by $\langle\tau^{a_1}_{d_1}\ldots\tau^{a_l}_{d_l}\sigma_{b_1}\ldots\sigma_{b_k}\rangle^{(W,\mathfrak{h},\mu_r),o}$.
\end{thm}

\begin{definition}
    The intersection numbers $\langle\tau^{a_1}_{d_1}\ldots\tau^{a_l}_{d_l}\sigma_{b_1}\ldots\sigma_{b_k}\rangle^{(W,\mathfrak{h},\mu_r),o}$ from Theorem~\ref{thm:TZ-r-spin_nums_well_def} are called open $(r,\mathfrak{h})$-spin invariants.
\end{definition}

\begin{rmk}
 An equivalent point of view is to define section-dependent intersection numbers without doing the point insertion procedure, and then we would have obtained that certain polynomial combinations of these intersection numbers give rise to invariants, as in Theorem \ref{thm:GKT_invariants}. 

Section 8 of \cite{TZ2} sketches how the point insertion idea can also be applied to open Gromov-Witten theories, in which the Lagrangian satisfies some properties. The topological recursion relation we shall see below for the \TZ-$r$-spin theories extend naturally as a universal recursion relation for all point-insertion theories in $g=0.$ 

 The $\mathfrak{h}=0$ $r$-spin theory is equivalent to the \BCT-$r$-spin theory in the sense that there is an invertible polynomial transformation between the intersection numbers of one and that of the other \cite[Theorem 5.6]{TZ2}. The geometric origin of this equivalence is that one can define the \BCT-$r$-spin intersection numbers as intersection numbers on the $\mathfrak{h}=0$ moduli space of $\MU$-connected surfaces, with the same boundary conditions, but a modified version $\widehat\CL_i$ of the relative cotangent lines $\CL_i$. These modified lines are pulled back from the moduli space obtained by forgetting all internal markings matched by $\MU.$ It can be shown that this definition makes sense, and that canonical sections in the \BCT~sense induce canonical sections for the modified intersection theory, with the same intersection numbers. See \cite[Section 5]{TZ2} for more details.
\end{rmk}
\subsubsection{The $g=0$ Fermat quintic $(W,G)=(x_1^5+\cdots+x_5^5,\mu_5)$}
An example of a higher Fermat intersection theory which involves the point insertion technique is the Fermat quintic with minimal admissible symmetry group $\mu_5.$ In this case all twists have the form $(a,a,a,a,a)$, where $a=0,\ldots,4.$ Boundary twists are more restricted: in the $\mathfrak{h}=0$ case, which is equivalent to the na\"ive attempt to construct the quintic theory using \BCT-$5$-spin techniques, all boundary twists must be $(3,3,3,3,3).$ The $\mathfrak{h}=1$ allows also $(1,1,1,1,1)$ boundary twists. Equation 
\eqref{eq:open_rank1_general} implies that primary invariants, those which do not involve descendents, must satisfy
\[5 \ \bigg| \ 2\sum_{i=1}^l{a_i}+\sum_{i=1}^kb_i-3=2l+k-3.\]
Using geometric arguments one can show that all possible non-trivial intersection numbers for $\mathfrak{h}=0$ vanish. Also, most intersection number for $\mathfrak{h}=1$ vanish. However, for
$\mathfrak{h}=1$ the intersection numbers $\langle\sigma_1^{5d+3}\rangle^{(W,\h=1,\mu_5),o}_0$ for $d$ positive and odd were conjectured not to vanish in \cite{TZ2}. This conjecture further suggests these intersection numbers are related to the open GW theory of the Calabi-Yau threefold studied in \cite{PSW} via an open analog of the LG/CY correspondence. See \cite{Melissa,Walcher} for this conjectural correspondence.

\subsubsection{Open potentials}
Once we have a genus $g$ definition of intersection numbers for an open FJRW theory, we can pack them in a generating function called the \emph{potential}, e.g, for \BCT-$r$-spin theory it reads
\begin{equation}\label{eq:open_pot}F_g^{x^r,o}=\sum_{\substack{k,l\geq 0\\2g-2+k+2l>0}}\frac{s^k}{k!l!}\sum_{\substack{a_1,\ldots,a_l\in\{0,\ldots,r-1\}\\d_1,\ldots,d_l\geq 0}}\left(\prod_{i=1}^lt_{d_i}^{a_i}\right)\large\langle\prod_{i=1}^l\tau_{d_i}^{a_i}\sigma^k\large\rangle^{x^r,o}\end{equation}
where $o$ stands for open.
In the \BCT-case, where all intersection numbers are defined, for all $g$ (with only $0$ twists) we also define the full potential \begin{equation}\label{eq:all_g_open_pot}F^{o}=\sum_{g\geq0}u^{g-1}F_{g}^{x^2,o|_{t_d^1=0,~d\geq 0}}.\end{equation}
\section{Computations, Mirror Symmetry and Integrable Hierarchies}
\label{sec:comp ms ih}
\subsection{Open topological recursion relations} Open topological recursion relations are used to compute open $r$-spin invariants or prove structural results like mirror symmetry about combinations of open FJRW invariants. We now will outline these relations in each of the above theories. In each theory, there are two types of open topological recursion relations in genus 0. Roughly speaking, they relate to different distinguished multisections of the cotangent line bundle on the moduli space of Riemann surfaces with boundary, and thus give two different relations that can be played off on one another and applied to computations.
\subsubsection{\PST~and \BCT~ theories}
The \PST~and \BCT~theories satisfy the following topological recursion relations in genus $0,1,$ which allow computing all $g=0,1$ numbers:
\begin{itemize}
\item[(a)] (Boundary marked point $g=0$ TRR) Suppose $l,k\geq 1$.  Then
\begin{align*}
\< \tau_{d_1+1}^{a_1}\prod_{i=2}^l\tau^{a_i}_{d_i}\sigma^k\>^{x^r,o}_0=&
\sum_{a=-1}^{r-2}\sum_{S \sqcup R = \{2,\ldots,l\}}\left\langle \tau_0^{a}\tau_{d_1}^{a_1}\prod_{i \in S}\tau_{d_i}^{a_i}\right\rangle^{x^r,{ext}}_0
\left\langle \tau_0^{r-2-a}\prod_{i\in R}\tau^{a_i}_{d_i}\sigma^k\right\rangle^{x^r,o}_0+\\
&+\sum_{\substack{S \sqcup R = \{2,\ldots,l\} \\ k_1 + k_2 = k-1}} \binom{k-1}{k_1} \left\langle \tau^{a_1}_{d_1}\prod_{i \in S} \tau^{a_i}_{d_i}\sigma^{k_1}\right\rangle^{x^r,o}_0 \left\langle \prod_{i \in R} \tau^{a_i}_{d_i} \sigma^{k_2+2}\right\rangle^{x^r, o}_0.
\end{align*}
\item[(b)] (Internal marked point $g=0$ TRR) Suppose $l\geq 2$.  Then
\begin{align*}
\<\tau_{d_1+1}^{a_1}\prod_{i=2}^l\tau^{a_i}_{d_i}\sigma^k\>^{x^r,o}_0=&
\sum_{a=-1}^{r-2}\sum_{S \sqcup R = \{3,\ldots,l\}}\left\langle \tau_0^{a}\tau_{d_1}^{a_1}\prod_{i \in S}\tau_{d_i}^{a_i}\right\rangle^{x^r,{ext}}_0
\left\langle \tau_0^{r-2-a}\tau^{a_2}_{d_2}\prod_{i\in R}\tau^{a_i}_{d_i}\sigma^k\right\rangle^{x^r,o}_0+\\
&+\sum_{\substack{S \sqcup R = \{3,\ldots,l\} \\ k_1 + k_2 = k}} \binom{k}{k_1} \left\langle \tau^{a_1}_{d_1} \prod_{i \in S} \tau^{a_i}_{d_i}\sigma^{k_1}\right\rangle^{x^r,o}_0 \left\langle \tau^{a_2}_{d_2}\prod_{i \in R} \tau^{a_i}_{d_i} \sigma^{k_2+1}\right\rangle^{x^r, o}_0.
\end{align*}
\item[(c)](Genus $1$ TRR)
\begin{equation}
\begin{split}
\left\langle \tau_{d_{1}+1}^{a_1}\prod_{i\in [l]\setminus \{1\}}\tau^{a_i}_{d_i}\sigma^k\right\rangle^{x^r,o}_1\hspace{-0.2cm}
=&\sum_{\substack{J_1 \sqcup J_2 = [l]\setminus\{1\}\\ -1\le a \le r-2}}\hspace{-0.1cm}\left\langle \tau_0^{a}\tau_{d_{1}}^{a_{1}}\prod_{i \in J_1}\tau_{d_i}^{a_i}\right\rangle^{x^r,\text{ext}}_0
\hspace{-0.1cm}\left\langle \tau_0^{r-2-a}\prod_{i\in J_2}\tau^{a_i}_{d_i}\sigma^k\right\rangle^{x^r,o}_1\\
&+\hspace{-0.1cm}\sum_{\substack{J_1 \sqcup J_2 =  [l]\setminus\{1\} \\ k_1+k_2=k}} \hspace{-0.1cm} \binom{k}{k_1}\left\langle \tau^{a_{1}}_{d_{1}}\prod_{i \in J_1} \tau^{a_i}_{d_i}\sigma^{k_1}\right\rangle^{x^r,o}_0 \hspace{-0.1cm}\left\langle \prod_{i \in J_2} \tau^{a_i}_{d_i} \sigma^{k_2+1}\right\rangle^{x^r, o}_1\\
&+\frac{1}{2}\left\langle \prod_{i\in [l]}\tau^{a_i}_{d_i}\sigma^{k+1}\right\rangle^{x^r,o}_0.
    \end{split}
\end{equation}\label{eq:g=1_trr}
\end{itemize}
The first two items were proven in \cite[Theorem 1.5]{PST14} for $r=2,$ no Ramond insertions, and in \cite[Theorem 4.1]{BCT2} for general $r$ and all insertions.
The third item is in \cite{TZ3}.

The proofs involve studying the zero count of an explicit global section, and comparing it to the non explicit canonical ones via homotopy arguments. 

\subsubsection{Relation with Solomon's Open WDVV}
In \cite{sol_owdvv} Solomon found that some $g=0$ open GW theories satisfy a uniform relation which he termed the \emph{Open WDVV (OWDVV)}. 
This relation is satisfied by a large class of open GW theories, which include  the OGW theory of $(\mathbb{CP}^n,\mathbb{RP}^n)$ for $n$ odd \cite{solomon2023relative}.
This relation also holds for $g=0$~\PST ~and \BCT ~theories, and is a formal consequence of the genus $0$ topological recursion relations presented above. Very roughly speaking, theories which satisfy Solomon's OWDVV are those theories which allow for two (torsionless) boundary states, one corresponds to the \emph{point class} (in the $r$-spin case it is the twist $r-2$) and one to the \emph{unit} (in the $r$-spin case it is the twist $0$). 

In \cite{alexandrov2023construction} Alexandrov, Basalaev and Buryak, following the program of Givental \cite{givental2001gromov,givental2001semisimple} study all genus \emph{F-CohFTs}, which are generalizations of CohFTs designed to axiomatize open intersection theories whose $g=0$ sector is subject to Solomon's OWDVV. They show the associated potentials, including descendents, satisfy certain universal equations: the \emph{open string} and \emph{open dilaton} equations for all genus, and $g=0,1$ \emph{topological recursion relations}. See also \cite{gomez2021open}. In light of their work and the Givental-Teleman reconstruction theorem in the closed setting \cite{Teleman,givental2001gromov,givental2001semisimple}, the following question is natural:
\begin{question}
    Do OGW and OFJRW theories whose $g=0$ part satisfies Solomon's OWDVV (and some semi-simplicity condition) have an all genus extension, which includes descendents, and whose potential is given by Alexandrov-Basalaev-Buryak's recipe?
\end{question}

\PST~theory and \BCT~theory are the only known constructions of open theories which satisfy Solomon's OWDVV in $g=1$. Moreover, it holds in all genera for \PST~theory. The examples of \TZ~and \GKT~
 theories below, which allow richer varieties of boundary states, are subject to different topological recursion relations, and also to wall-crossing phenomena for \GKT~theory, as we shall see below.

\subsubsection{\TZ~theories}
\TZ~theories satisfy very different types of topological recursion relations.
\begin{thm}\label{thm:TZ_TRR}
Let $W=x_1^r+\cdots+x_m^r,~G=\mu_r$ and $\h\in\{1,\ldots,\lfloor\frac{r}{2}\rfloor\}$. Then the following relations hold, whenever \eqref{eq:TZ_rk_vs_dim} holds:
\begin{itemize}
\item[(a)] If $l,k\ge 1$, then
\begin{equation*}
		\begin{split}
		&\left\langle
		\tau^{a_1}_{d_1+1}\tau^{a_2}_{d_2}\dots\tau^{a_l}_{d_l}\sigma^{b_1}\sigma^{b_2}\dots\sigma^{b_k}
		\right\rangle_0^{(W,\mathfrak{h}),o}
		\\=&\sum_{\substack{s\ge 0\\-1\le a \le r-2}}\sum_{\substack{0 \le t_i \le \h\\\coprod_{j=-1}^{s}R_j=\{2,3,\dots,l\}\\\coprod_{j=0}^{s}T_j=\{2,3,\dots,k\}\\ \{(R_j,T_j,t_j)\}_{1\le j \le s}\text{ unordered}}}(-1)^s\left\langle
		\tau^{a}_{0}\tau^{a_1}_{d_1}\prod_{i\in R_{-1}}\tau^{a_i}_{d_i}\prod_{j=1}^{s}\tau^{t_j}_{0}
		\right\rangle_0^{\frac{1}{r},\text{ext},m}\\
		&\cdot\left\langle
		\tau^{r-2-a}_{0}\sigma^{b_1}\prod_{i\in R_0}\tau^{a_i}_{d_i}\prod_{i\in T_0}\sigma^{b_i}
		\right\rangle_0^{(W,\mathfrak{h}),o}
		\prod_{j=1}^{s}\left\langle
		\sigma^{r-2-2t_j}\prod_{i\in R_j}\tau^{a_i}_{d_i}\prod_{i\in T_j}\sigma^{b_i}
		\right\rangle_0^{(W,\mathfrak{h}),o}.
		\end{split}
		\end{equation*}

\item[(b)] If $l\ge 2$, then
\begin{equation*}
		\begin{split}
		&\left\langle
		\tau^{a_1}_{d_1+1}\tau^{a_2}_{d_2}\dots\tau^{a_l}_{d_l}\sigma^{b_1}\sigma^{b_2}\dots\sigma^{b_k}
		\right\rangle_0^{(W,\mathfrak{h}),o}
		\\=&\sum_{\substack{s\ge 0\\-1\le a \le r-2}}\sum_{\substack{0 \le t_i \le \h\\ \coprod_{j=-1}^{s}R_j=\{3,4,\dots,l\}\\\coprod_{j=0}^{s}T_j=\{1,2,\dots,k\}\\ \{(R_j,T_j,t_j)\}_{1\le j \le s}\text{ unordered}}}(-1)^s\left\langle
		\tau^{a}_{0}\tau^{a_1}_{d_1}\prod_{i\in R_{-1}}\tau^{a_i}_{d_i}\prod_{j=1}^{s}\tau^{t_j}_{0}
		\right\rangle_0^{\frac{1}{r},\text{ext},m}\\
		&\cdot\left\langle
		\tau^{r-2-a}_{0}\tau^{a_2}_{d_2}\prod_{i\in R_0}\tau^{a_i}_{d_i}\prod_{i\in T_0}\sigma^{b_i}
		\right\rangle_0^{(W,\mathfrak{h}),o}
		\prod_{j=1}^{s}\left\langle
		\sigma^{r-2-2t_j}\prod_{i\in R_j}\tau^{a_i}_{d_i}\prod_{i\in T_j}\sigma^{b_i}
		\right\rangle_0^{(W,\mathfrak{h}),o}.
		\end{split}
		\end{equation*}
\end{itemize}
\end{thm}
This form of this TRR, proven in \cite{TZ2}, is very different than TRRs that appeared before in OGW theory, notably Solomon's Open WDVV \cite{sol_owdvv,solomon2023relative} and the TRRs of \cite{PST14,BCT2}. But it is argued in \cite[Section 7]{TZ2} that this TRR is the \emph{universal form} of TRRs for theories based on the point insertion technique.
The proof uses a construction of an explicit section of $\CL_1,$ this time canonical over the whole moduli of $\MU$-connected disks, and studies its zero count.
These recursions allow the calculation of all intersection numbers in the $r$-spin case $m=1,$ see \cite[Section 6]{TZ2}.

 \TZ-numbers are also subject to several vanishing theorems, also expected to be universal, see \cite[Propositions 1.3, 1.4]{TZ2}.
\subsubsection{\GKT~theories}
As mentioned above, in the general \GKT~theory intersection numbers do depend on choices, but there are invariant quantities. One might have thought that this is a bug of the theory, and that a better definition would yield fully invariant numbers. As we shall see in the discussion in mirror symmetry below, this is in fact a feature, which is necessary in order for mirror symmetry to hold.

In order to avoid making the notation too heavy, 
we will state the result in the case $(W,G)=(x^r+y^s,\mu_r\times\mu_s)$.
\begin{nn}\label{nn:mathcal_A}
We define $\mathcal{A}(\mathbf{A},\mathbf{d},\mathbf{s}^\bullet)$, where $\mathbf{A}=\{(a_1,b_1),\ldots,(a_l,b_l)\}$ is a multiset of twists for internal markings, $\mathbf{d}=\{d_1,\ldots,d_l\}$ is a vector of descendents, and $\mathbf{s}=\mathbf{s}^\bullet$ is a compatible canonical family boundary conditions for the open FJRW theory $(W,G)=(x^r+y^s,\mu_r\times\mu_s)$, to be
\begin{align*}
\mathcal{A}(\mathbf{A},\mathbf{d},\mathbf{s}^{\bullet}):=\sum_{h=1}^{l}\frac{1}{h!}\sum_{J_1\sqcup\cdots\sqcup J_h=\mathbf{A}}&\sum_{\substack{\{k_1(i)\}_{i=1}^{h},~k_1(i)=r(J_i)~(mod~r),\\
\{k_2(i)\}_{i=1}^{h},~k_2(i)=s(J_i)~(mod~s),\\sk_1(i)+rk_2(i)=m(J_i,\vecd)}}
\frac{\Gamma(\frac{1+\sum_{i=1}^hk_1(i)}{r}) \Gamma(\frac{1+\sum_{i=1}^h k_2(i)}{s})}{\Gamma(\frac{1+r(J)}{r}) \Gamma(\frac{1+s(J)}{s})}\cdot\\
&\cdot\prod_{i=1}^h 
\langle\prod_{j\in J_i}\tau_{d_j}^{(a_j,b_j)}\sigma_1^{k_1(i)}\sigma_2^{k_2(i)}\sigma_{12};{s}^{(J_i,k_1(i),k_2(i))}\rangle^{W,o}
\end{align*}
where the numbers $r(I)\in\{0,\ldots,r-1\},s(I)\in\{0,\ldots,s-1\}$, $m(I,\mathbf{d})$, $d(I, \mathbf{d})$, for $I \subseteq [l]$ are given by
\begin{align}\label{eq:r(I),s(I),m(I)}
\begin{split}
r(I) \equiv {} & \sum_{j\in I}a_j~(\text{mod}~r),\\
s(I)\equiv {} & \sum_{j\in I}b_j~(\text{mod}~s), \\
m(I,\vecd):= {} &rs + \sum_{j\in I} \left(sa_j+rb_j+rs(d_j-1)\right), \text{ and } \\
d(I,\vecd):= {} & {sr(I)+rs(I)-m(I,\vecd)\over rs}-1.
\end{split}
\end{align}
Here $\Gamma(\cdot)$ denotes the Gamma function.
For a multiset of twists $$\mathbf{B} = \{(a_i, b_i) \ | \ i \in I\subseteq[l]\}\subseteq \mathbf{A},$$ we similarly denote $r(\mathbf{B}):=r(I), s(\mathbf{B}):= s(I), d(\mathbf{B}, \mathbf{d})=d(I, \mathbf{d}),\text{ and } m(\mathbf{B}, \mathbf{d})=m(I,\mathbf{d})$.
Finally, ${s}^{(I, k_1, k_2)}$ is the canonical boundary condition for the moduli $\M_{0,k_1,k_2,1,\mathbf{A}_{I}}^{W,o}$ coming from the compatible family $\mathbf{s}^\bullet.$
\end{nn}
The main geometric result of \cite{GKT2} is
\begin{thm}\label{thm:A_vanishes}
The polynomial combination $\mathcal{A}(\mathbf{A},\mathbf{d},\mathbf{s})$ is independent of the canonical family $\mathbf{s}$, hence will be denoted $\mathcal{A}(\mathbf{A},\mathbf{d})$. In addition:
\begin{enumerate}
\item
If $\mathbf{A}= \{(a_1,b_1)\}$ is a singleton, then $\mathcal{A}(\mathbf{A},\mathbf{d}) = (-1)^{d_1}$.
\item Suppose $l\ge 2$.
 Then
\[
\mathcal{A}(\mathbf{A},\mathbf{d})=
\begin{cases}
0 & d(\mathbf{A},\vecd)<0,\\
\left\langle \tau_{d(\mathbf{A},\mathbf{d})}^{(r-r(I)-2,s-s(I)-2)}\prod_{i\in I}\tau_{d_i}^{(a_i,b_i)}
\right\rangle^{\textup{ext}}&d(\mathbf{A},\mathbf{d})\ge 0.
\end{cases}
\]
\end{enumerate}
Moreover, for every rational solution of the equations of (1) and (2), ranging over all $\mathbf{A}$ and $\mathbf{d}$, there exists a compatible family of canonical boundary conditions such that the resulting intersection numbers
 agree with this rational solution.\end{thm}

This theorem generalizes to the most general Fermat, as in \eqref{eq:Fermat W}, in the upcoming \cite{GKT3}.
The unfixed boundary conditions described in the previous section are responsible for the dependence of the intersection numbers on choices. Yet, the intricate form of $\mathcal{A}$ and the properties of canonical multisections guarantee that while different intersection numbers may change, $\mathcal{A}$ itself does not. Computing $\mathcal{A}$ involves a highly non-trivial generalization of the proofs of the other TRRs described in this section.

\begin{ex}
We now return to Example~\ref{GKT Example of Wall Crossing}. In this case, $\mathbf{A} = \{(2,2),(2,2)\}$ and $\vecd = \mathbf{0}$. One can compute that 
\begin{equation}\begin{aligned}
    \mathcal{A}(\mathbf{A},\mathbf{d}) &= \frac{\Gamma(\frac{5}{4}) \Gamma(\frac{1}{4})}{\Gamma(\frac{1}{4}) \Gamma(\frac{1}{4})} \cdot \langle\tau_{0}^{(2,2)} \tau_{0}^{(2,2)}\sigma_1^{4}\sigma_{12}; s\rangle^{W,o} + \frac{\Gamma(\frac{1}{4}) \Gamma(\frac{1}{4})}{\Gamma(\frac{5}{4}) \Gamma(\frac{1}{4})} \cdot \langle\tau_{0}^{(2,2)} \tau_{0}^{(2,2)}\sigma_2^{4}\sigma_{12}; s\rangle^{W,o} \\
        & \qquad + \frac{1}{2}\cdot \frac{\Gamma(\frac{5}{4}) \Gamma(\frac{5}{4})}{\Gamma(\frac{1}{4}) \Gamma(\frac{1}{4})} \cdot \left(\langle\tau_{0}^{(2,2)}\sigma_1^{2}\sigma_2^2\sigma_{12};s\rangle^{W,o}\right)^2 \\
        &= \frac{1}{4} (\langle\tau_{0}^{(2,2)} \tau_{0}^{(2,2)}\sigma_1^{4}\sigma_{12}; s\rangle^{W,o} + \langle\tau_{0}^{(2,2)} \tau_{0}^{(2,2)}\sigma_2^{4}\sigma_{12}; s\rangle^{W,o}) + \frac{1}{32} \left(\langle\tau_{0}^{(2,2)}\sigma_1^{2}\sigma_2^2\sigma_{12}; s\rangle^{W,o}\right)^2.
\end{aligned}\end{equation}
After computing $d(\mathbf{A}, \vecd)$ and applying Theorem~\ref{thm:A_vanishes}, we see $\mathcal{A}(\mathbf{A},\mathbf{d}) = 0$ and $\mathcal{A}(\{(2,2)\}, 0) = \langle\tau_{0}^{(2,2)}\sigma_1^{2}\sigma_2^2\sigma_{12}\rangle^{W,o} = 1$. Thus we obtain the relation
$$
\langle\tau_{0}^{(2,2)} \tau_{0}^{(2,2)}\sigma_1^{4}\sigma_{12}; s\rangle^{W,o} + \langle\tau_{0}^{(2,2)} \tau_{0}^{(2,2)}\sigma_2^{4}\sigma_{12}; s\rangle^{W,o} = - \frac{1}{8},
$$
yielding the quantity in Equation~\eqref{sum quantity for WC}.
\end{ex}

\subsection{Open FJRW theory with maximal symmetry group and mirror symmetry}
\label{marks mirror symmetry section}
We return to the mirror symmetry discussion of 
\S\ref{subsubsec:mirror symmetry}, and in particular continue with
the mirror correspondence discussed there with $a=1,2$. 
It is in practice very difficult to construct the
isomorphism of Frobenius manifolds described in \cite{HeLiShenWebb}.
However, the lesson learned in the case of mirror symmetry for toric
varieties as studied by Gross (in the case of $\mathbb{P}^2$) in \cite{GrossP2}
and
Fukaya et al in \cite{FOOO1} is that life becomes much simpler if one uses
open invariants to produce the potential. Instead of using an arbitrary 
universal unfolding of the potential $W$, the correct thing to do is
to use a perturbation of the potential which is a generating function for
open invariants. In the papers just cited, these are open Gromov-Witten
invariants, but in \cite{GKT,GKT2}, in the ranks $1$ and $2$ cases, we
use the open FJRW invariants whose construction we have described above.

We focus here on the statements in the rank $2$ case, the rank $1$
case being rather simpler and treated in \cite{GKT}. We also, for the simplicity of exposition, focus here on the generating function built from primary open invariants. We write as usual
$W=x^r+y^s$.
Choose a family of canonical multi-sections 
${\bf s}$.\footnote{To be more technically accurate, what we say
below requires what we call a family of 
\emph{symmetric} canonical multi-sections;
we send the reader to \cite[\S3.5]{GKT} for these details. They can be
for the most part ignored in this survey.}
Define the ring
\[
R:=\Q[t_{a,b}\,|\,0\le a\le r-2, 0\le b \le s-2].
\]
We may then define a perturbed potential
\[
W^{\bf{s}} =
\sum_{k_1,k_2\ge 0,l\ge 0}
\sum_{\mathbf{A}=\{(a_i,b_i)\}\in\mathcal{A}_l}
(-1)^{l-1}{\langle \prod_{i=1}^l \tau_0^{(a_j,b_j)}
\sigma_1^{k_1}\sigma_2^{k_2}\sigma_{12}\rangle^{W,{\bf s},o}
\over |\Aut(\mathbf{A})|}
x^{k_1}y^{k_2}\prod_{i=1}^l t_{a_i,b_i} \in R[[x,y]].
\]
Here $\mathcal{A}_l$ is the set of size $l$ multi-sets whose entries
are pairs $(a,b)\in \{0,\ldots,r-2\}\times \{0,\ldots,s-2\}$,
while if $\mathbf{A}=\{(a_i,b_i)\}$ is such a multi-set, then
the automorphism group $\Aut(\mathbf{A})$ is the set of permutations $\sigma$ of
$[l]$ such that $(a_i,b_i)=(a_{\sigma(i)},b_{\sigma(i)})$
for all $1\le i \le l$. We remark that this perturbed potential is different from the open potential in Equation~\eqref{eq:open_pot}.

One can show that modulo the ideal $\langle t_{a,b}\rangle^2$,
$W^{\bf s}$ agrees with
\[
x^r+y^s+\sum_{0\le a\le r-2, 0\le b\le s-2} t_{a,b}
x^{a}y^{b}.
\]
Hence $W^{\bf s}$, to first order, agrees with the standard universal 
unfolding of $W$ with parameters $t_{a,b}$. However, the higher
order terms in the $t$ variables are crucial for giving a simple form
for the Saito-Givental theory.

Here we express this in terms of oscillatory integrals. Following
\cite{HeLiShenWebb}, a \emph{good basis} for the relative homology
group $H^a(\C^a, \re(W/\hbar)\ll 0;\C)$ is a basis 
\[
\big\{\Xi_{\mu}\,|\,(\mu_1,\ldots,\mu_a)\in \prod_{i=1}^a \{0,\ldots,r_i-2\}
\big\}
\]
with the property that
\[
\int_{\Xi_{\mu}} x^{\mu'}e^{W/\hbar}dx_1\wedge\cdots dx_a
=
\delta_{\mu\mu'},
\]
where $x^{\mu'}:=\prod_i x_i^{\mu'_i}$, and the $\delta$ is the Kronecker
delta. The reason for what may be viewed as a quite puzzling definition
for the quantity $\mathcal{A}(\mathbf{A},\mathbf{d},\mathbf{s})$ is that
a direct calculation shows that, in the rank $2$ case,
\[
\int_{\Xi_{(a,b)}} e^{W^{\bf s}/\hbar} dx\wedge dy
=
\delta_{a,0}\delta_{b,0} +
\sum_{l\ge 1} \sum_{\mathbf{A}=\{(a_i,b_i)\}\in \mathcal{A}_l}
(-1)^l(-\hbar)^{-d(\mathbf{A})-2}\mathcal{A}(\mathbf{A},\mathbf{s})\delta_{r(\mathbf{A}),a}\delta_{s(\mathbf{A}),b}
\left({\prod_{j=1}^l t_{a_j,b_j}\over |\Aut(\mathbf{A})|}\right).
\]
From this formula and Theorem~\ref{thm:A_vanishes}, we obtain the main mirror
statement of \cite{GKT2}:

\begin{thm}[Open Mirror Symmetry]
\label{thm:main mirror thm}
\begin{align*}
\int_{\Xi_{a,b}} &e^{W^{\mathbf{s}}/\hbar}dx\wedge dy 
= {}
\delta_{a,0}\delta_{b,0} + t_{a,b}\hbar^{-1}
\\
{} & +\sum_{l\ge 2}
\sum_{\substack{\mathbf{A}\in \mathcal{A}_l\\
d(\mathbf{A})\ge 0}} (-1)^{l}
(-\hbar)^{-d(\mathbf{A})-2}{\delta_{r(\mathbf{A}),a}\delta_{s(\mathbf{A}),b}
\over |\Aut(\mathbf{A})|}
 \left\langle \tau_{d(\mathbf{A})}^{(r-r(\mathbf{A})-2,s-s(\mathbf{A})-2)}
\prod_{(a_i,b_i)
\in \mathbf{A}}\tau_0^{(a_i,b_i)}
\right\rangle^{\text{ext}}
\prod_{j=1}^l t_{a_j,b_j}.
\end{align*}
\end{thm}
The generating function for closed extended FJRW invariants appearing
on the right-hand side is the Givental $J$-function associated to the
theory. This can be viewed as giving an effective version of the genus zero results 
of \cite{HeLiShenWebb} in the case that $W$ is a rank $2$ Fermat polynomial. It is important to emphasize that the right-hand side of the equality in Theorem~\ref{thm:main mirror thm} does not depend on the family of multisections $\mathbf{s}$, while the potential $W^{\mathbf{s}}$ does. In \cite[Corollary 5.11]{GKT2}, we in fact give a stronger version of Theorem~\ref{thm:main mirror thm} which includes
descendents on all insertions, following \cite{Overholser} for the case
of mirror symmetry for $\P^2$. 

\subsubsection{The LG wall crossing group}\label{subsub:LG_WC_group}
In the case of $r$-spin invariants, where the analogous result to
Theorem~\ref{thm:main mirror thm} was proved, the potential $W^{\mathbf{s}}$
is well-defined as open $r$-spin invariants are well-defined independently
of choices of canonical multisections \cite{BCT2}. However, in rank $2$, many
open invariants are  not well-defined, which raises the obvious
question of the relationship between invariants defined by different families
of canonical multisections $\mathbf{s}$ and $\mathbf{s}'$. This is a
standard phenomenon across many types of geometric situations where a
moduli or enumerative problem depends on parameters. Here, analogously
to the Landau-Ginzburg models mirror to $\mathbb{P}^2$
studied in \cite{GrossP2}, we find that the relationship
between different sets of invariants is controlled by a \emph{wall-crossing
group}.

For a potential $W=x^r+y^s$, we define a subgroup $G\subseteq\Aut_R(R[[x,y]])$,
the latter being the group of continuous $R$-algebra automorphisms
of $R[[x,y]]$. This group is determined by specifying its Lie algebra, which
is done as follows. Recall that $S:=R[[x,y]]$ is an inverse limit of rings
\[
S_k:=R[x,y]/(x,y)^{k+1}.
\]
We have the module $\Theta_k$ of derviations of $S_k$ over $R$:
\[
\Theta_k:=S_k\partial_x \oplus S_k \partial_y.
\]
This module comes equipped with the standard Lie bracket, i.e.,
\[
[x^{k_1}y^{k_2}\partial_x, x^{\ell_1}y^{\ell_2}\partial_y]
=\ell_1 x^{k_1+\ell_1-1}y^{k_2+\ell_2}\partial_y- k_2
x^{k_1+\ell_1}y^{k_2+\ell_2-1}\partial_x.
\]
We set
\begin{align*}
\mathfrak{g}_{R,k}:= {} & 
\bigoplus_{\substack{(k_1,k_2)\in\Z_{\ge 0}^2\setminus \{0\}\\ k_1+k_2\le k}}
\mathfrak{g}_{R,(k_1,k_2)}\\
:= {} &
\bigoplus_{\substack{(k_1,k_2)\in\Z_{\ge 0}^2\setminus \{0\}\\ k_1+k_2\le k}}
\mathfrak{m}\cdot\big(x^{k_1}y^{k_2}((k_2+1)x\partial_x-(k_1+1)y\partial_y)\big).
\end{align*}
Here $\mathfrak{m}$ is the ideal of $R$ generated by the $t_{a,b}$'s.
As $\mathfrak{g}_k$ is a nilpotent Lie algebra, we may define a group
$G_{R,k}:=\exp(\mathfrak{g}_{R,k})$. This is a group whose underlying
set is $\mathfrak{g}_{R,k}$ but where multiplication is given by the 
Baker-Campbell-Hausdorff formula. We may then define $G_R$ as the
pro-nilpotent Lie group given as the inverse limit of the $G_{R,k}$,
identified as a set with $\mathfrak{g}_R$, the inverse limit of the
$\mathfrak{g}_{R,k}$.

Note that $G_R$ acts on $R$ via automorphisms: given $v\in \mathfrak{g}_R$,
$\exp(v)$ acts on elements $f$ of $R[[x,y]]$ via
\[
f\mapsto \sum_{n=0}^{\infty} {v^n(f)\over n!},
\]
where $v^n$ denotes differentiating with respect to the vector field
$n$ times.

One checks that these automorphisms preserve the holomorphic $2$-form
$dx\wedge dy$ and are the identity modulo $\mathfrak{m}$.
One can also check (see \cite[Lem.~4.20]{GKT2}) that if
$v=gx^{k_1}y^{k_2}((k_2+1)x\partial_x-(k_1+1)y\partial_y)$ for some 
$g\in\mathfrak{m}$,
then we have the following possibilities describing the action of
$\exp(v)$. If $k_1\not=k_2$, then $\exp(v)$ acts by 
\[
\begin{aligned}
x &\longmapsto x(1+(k_2-k_1) g x^{k_1}y^{k_2})^{(k_2+1)/(k_2-k_1)}\\
y &\longmapsto y(1+(k_2-k_1) g x^{k_1}y^{k_2})^{(k_1+1)/(k_1-k_2)},
\end{aligned}
\] 
while if $k_1=k_2$, $\exp(v)$ acts by 
\[
\begin{aligned}
x &\longmapsto x\exp((k_1+1) g x^{k_1}y^{k_2})\\
y &\longmapsto y\exp(-(k_1+1) g x^{k_1}y^{k_2}).
\end{aligned}
\]
For the given potential $W=x^r+y^s$, we identify a group $G^{r,s}_R$,
a subgroup of $G_R$, which serves as our wall-crossing group. This is
done by defining a subspace $\mathfrak{g}_{R,(k_1,k_2)}^{r,s}
\subseteq \mathfrak{g}_{R,(k_1,k_2)}$ via
\[
\mathfrak{g}_{R,(k_1,k_2)}^{r,s}
=
\bigoplus_{\substack{l>0, \mathbf{A}\in \mathcal{A}_l\\
k_1=r(\mathbf{A})\bmod r, \quad k_2=s(\mathbf{A})\bmod s\\
sk_1+rk_2=m(\mathbf{A})-rs}}
\left(\prod_{(a_i,b_i)\in \mathbf{A}} t_{a_i,b_i}\right)
x^{k_1}y^{k_2}\big((k_2+1)x\partial_x-(k_1+1)y\partial_y\big)\Q.
\]
We then define
\[
\mathfrak{g}^{r,s}_{R,k}=\bigoplus_{\substack{k_1,k_2\ge 0\\ 0<k_1+k_2\le k}}
\mathfrak{g}_{R,(k_1,k_2)}^{r,s}.
\]
As before, this then yields groups $G_{R,k}^{r,s}$ and $G_{R}^{r,s}$.

We then obtain:

\begin{thm} 
Let $W=x^r+y^s$.
\begin{enumerate}
\item
Let $\mathbf{s},\mathbf{s}'$ be two families of canonical multisections
for $W$. Then there exists an element $g\in G_R^{r,s}$ such that
\[
W^{\mathbf{s}'}=g(W^{\mathbf{s}}).
\]
\item
Let $\mathbf{s}$ be a family of canonical multisections for $W$ and
$g\in G^{r,s}_R$. Then there exists a family of canonical multisections
$\mathbf{s}'$ for $W$ with
\[
W^{\mathbf{s}'}=g(W^{\mathbf{s}}).
\]
\item Let $\mathbf{s}$ be a family of canonical multisections for $W$ and $g\in G^{r,s}_R$.    If $W^{\mathbf{s}}=g(W^{\mathbf{s}})$, then $g$ is the identity.
\end{enumerate}
\end{thm}
\noindent Morally, this tells us the wall-crossing group $G^{r,s}_A$ acts faithfully
and transitively on the set of all possible systems of open FJRW invariants.
This theorem underlies the proof of the last statement of Theorem~\ref{thm:A_vanishes}.
This theorem should not be viewed as saying the invariants can be chosen to be
anything
one wants: it is extremely hard to find, a priori, a set of open invariants
for which Theorem \ref{thm:main mirror thm} holds!

\begin{rmk}
While it is extremely hard to find these set of open invariants, for $a=2$ for low degree it has been classically computed as flat coordinates for the Frobenius manifold. We refer the reader to \cite{Noumi, NoumiYamada} for simple and elliptic singularities and \cite{Maher} for a modern treatment using the open enumerative geometry viewpoint. 
\end{rmk}

\subsection{Open $r$-spin and $r$-KdV hierarchy for $r\geq2$}
With $L$ as in Subsection \ref{sub:int_hierarchies_closed}, let~$\Phi(T_*,u)$ be the solution of the system of equations
\begin{gather}\label{eq:wave_func_eqs}
\frac{\partial\Phi}{\partial T_n}=u^{n-1}(L^{n/r})_+\Phi,\quad n\geq 1,
\end{gather}
which satisfies the initial condition
$$
\left.\Phi\right|_{T_{\geq 2}=0}=1.
$$
The system of differential equations~\eqref{eq:wave_func_eqs} for the function~$\Phi$ coincides with the system of differential equations for the wave function of the KP hierarchy.
This function $\Phi$, for $r=2$ was first studied in \cite{Bur16}. For general $r$ it was first studied in~\cite{BY15}, where the authors found an explicit formula for $\Phi$ in terms of the wave function.

Let $\phi:=\log\Phi$ and expand it as a power series in $u:$
$$
\phi=\sum_{g\in\mathbb{Z}}u^{g-1}\phi_g,\quad \phi_g\in\C[[T_*]].
$$
While $F^{x^r,c}_0$ of \eqref{eq:open_pot} depends only on the variables $t^0_d,\ldots,t^{r-2}_d$, the function~$F^{x^r,o}_0$ depends also on $t^{r-1}_d$ and $s$. We relate $T_{mr}$ and $t^{r-1}_{m-1}$ via
\begin{gather}\label{eq:r-1_coord_change}
T_{mr}=\frac{1}{(-r)^{\frac{m(r-2)}{2(r+1)}}m!r^m}t^{r-1}_{m-1},\quad m\ge 1.
\end{gather}

In the \PST~case $r=2,$ with all twists $0,$ it holds that
\begin{thm}\label{thm:openKdV_thm}
Let $F^o$ be the all genus open potential for $r=2$ and with no Ramond insertions, given in \eqref{eq:all_g_open_pot}.
Then $F^o=\left(\frac{1}{\sqrt{-2}}\right)^{g-1}\phi|_{t^1_d\mapsto{-\sqrt{-2}}\delta_{d,0}s}$, where the notation $|_{t^1_d\mapsto{-\sqrt{-2}}\delta_{d0}s}$ means substituting $-\sqrt{-2}\delta_{d,0}s$ for $t_d^1$.
\end{thm}
In $g=0$ an equivalent formulation in terms of the open KdV and Virasoro equations is proven in \cite{PST14}. The all genus case is a combination of results from (i) \cite{Tes15} which calculates a Feynman-sum formula for open intersection numbers, generalizing Kontsevich's formula \cite{Kontsevich} to the open setting, (ii) \cite{Bur15} which proved that the open KdV and open Virasoro equations are equivalent to the wave function statement above, and (iii) \cite{BT17} which proved that the numbers calculated in \cite{Tes15} satisfy the open Virasoro relations.

Safnuk in \cite{safnuk2016topological} studies Eynard-Orantin topological recursion for 
the above theory. A physics perspective on the subject can be found in \cite{dijkgraaf2018developments}.
In \cite{alexandrov2017refined,wang2025identification}, the authors study refined intersection numbers, which are filtered both by the genus $g$ and the number $h$ of boundary components of the topological surface, and relate them to the Kontsevich-Penner matrix model. Lastly, the open $r=2$ theory is related to higher Airy structures in \cite{borot2024higher}, and the authors further conjecture a generalization for all $r$. 

In \cite{BCT2,TZ2}, the open $r$-spin version of Witten's conjecture is proven in genus zero and one:
\begin{thm}\label{thm:main_BCT2TZ2}
It holds that
\begin{gather}\label{eq:main_res_BCT}
F^{x^r,o}_0=\frac{1}{\sqrt{-r}}\phi_0\big|_{t^{r-1}_d\mapsto \frac{1}{\sqrt{-r}}(t^{r-1}_d-r\delta_{d,0}s)}-\frac{1}{\sqrt{-r}}\phi_0\big|_{t^{r-1}_d\mapsto\frac{1}{\sqrt{-r}}t^{r-1}_d}.
\end{gather}
\begin{equation}\label{eq:main_res_TZ}
F^{x^r,o}_1=\left.\phi_1\right|_{t^{r-1}_d\mapsto\frac{1}{\sqrt{-r}}(t^{r-1}_d-\delta_{d,0}rs)}.
\end{equation}
\end{thm} 
\begin{conj}[\cite{BCT3}] Also for $g\geq 2$ there exists a geometric definition of the open $r$-spin potential $F_g^{x^r,o}$such that
\begin{equation}\label{eq:bct_conj}
F^{x^r,o}_g=\left.(-r)^{\frac{g-1}{2}}\phi_g\right|_{t^{r-1}_d\mapsto\frac{1}{\sqrt{-r}}(t^{r-1}_d-\delta_{d,0}rs)}.
\end{equation}
\end{conj}
We should point out that at the moment no analogous results are known or conjectured for the $\mathfrak{h}=1,2,\ldots,\lfloor r/2\rfloor-1$ open $r$-spin theories constructed in \cite{TZ1,TZ2}, even though all $g=0$ intersection numbers of these theories are fully calculated in \cite{TZ2}.
\begin{rmk}
In open Gromov--Witten theory, unlike the closed theory, not much is known or even conjectured about higher-genus invariants. The conjecture above is one of the few conjectures that describes a full, all-genus open Gromov--Witten theory, and also one of the few conjectures that relate the potential to an integrable hierarchy.\footnote{As mentioned, an analogous conjecture was made in \cite{PST14} for the $r=2$ case, and it was proven in~\cite{BT17}. In~\cite{BPTZ}, the full stationary open Gromov--Witten theory for maps to $(\mathbb{CP}^1,\mathbb{RP}^1)$ is conjectured.}
\end{rmk}
Finally, \PST~and \BCT~theories satisfy open string and dilaton equations.
\begin{thm}[\cite{PST14,BT17,BCT2,BCT3,TZ3}]
The following string and dilaton equations hold for the \PST~theory in all $g,$ and the \BCT~ theory in $g=0,1$ for which the theory is constructed:
\begin{itemize}
    \item (Open String)
Assume $2g-2+k+2l>0$. Then
\begin{equation}\label{eq:open_string}
\left\langle\tau^0_0\prod_{i=1}^l\tau^{a_i}_{d_i}\sigma^k\right\rangle^{x^r,o}_g=\sum_{j=1}^l\left\langle\tau^{a_j}_{d_j-1}\prod_{i=1,i\neq j}^l\tau^{a_i}_{d_i}\sigma^k\right\rangle^{x^r,o}_g,
\end{equation}
where an intersection number which includes $\tau^*_{-1}$ is defined to be $0.$
\item (Open Dilaton)
Assume $2g-2+k+2l>0$. Then
\begin{equation}\label{eq:open_dilaton}
\left\langle\tau^1_0\prod_{i=1}^l\tau^{a_i}_{d_i}\sigma^k\right\rangle^{x^r,o}_g=(g+l+k-1)\left\langle\prod_{i=1}^l\tau^{a_i}_{d_i}\sigma^k\right\rangle^{x^r,o}_g.
\end{equation}
\end{itemize}
\end{thm}

These relations are algebraic consequences of the relations between the open potentials and the corresponding wave functions of the $r$-KdV hierarchy, and the fact that these equations hold for the coefficients of the wave functions, see \cite[Section 3.3]{BCT3}. Alternatively they can be proven geometrically by using the behaviour of the relative cotangent lines and Witten bundles under the forgetful maps which forget an untwisted internal marking which does not carry descendents. This path was taken in \cite[Sections 4.1, 4.2]{PST14}, and the proof generalizes easily to all cases for which the intersection numbers are defined.

\section{Open questions and problems}
We finish this survey with several open questions and problems:

\begin{itemize}
\item {\bf Virtual techniques.} A key technical difficulty in open theories is the absence of the virtual fundamental class machinery that revolutionized the study of closed Gromov-Witten and FJRW theories. Since open theories are defined at the chain, rather than cycle level, a very important open problem is to construct a virtual fundamental chain formalism that could somehow incorporate boundary conditions. Such machinery seems to be necessary for constructing open $r$-spin theories in $g>1,$ as well as open FJRW theories beyond the concave setting.
\item {\bf Higher genus open $r$-spin and FJRW intersection numbers.}
With virtual techniques at hand, one can go on and try to compute the resulting numbers, and to test relations with mirror symmetry and integrable hierarchies. A question of particular importance is whether the all genus open $r$-spin potential yields the $r$-KdV wave function, as conjectured in \cite{BCT3}. One can also ask if, for higher-dimensional Landau-Ginzburg models, the algebraic construction in \cite{alexandrov2023construction} be extended to the open FJRW setting to produce an all genus conjecture, and if such conjecture can be realized geometrically. 
\item {\bf OFJRW theories with non-maximal symmetry group and non-Fermat potentials.}  If the obstacle of creating appropriate virtual fundamental class machinery can be handled, then one could look towards constructing open theories for more general Landau-Ginzburg models. While \cite{TZ2} handles this for minimal symmetry group for certain Fermat polynomials, this does not cover all Landau-Ginzburg models. As seen in \cite{FJR}, one may want to consider a loop or chain polynomials which arise in the Kreuzer-Skarke classification of invertible polynomials, or consider other non-maximal symmetry groups on the $A$-model. Even more ambitiously, one could then move towards more complex gauged linear sigma models, as done in \cite{FJRGLSM, FaveroKim}.
\item {\bf The open LG/CY correspondence and an open Chiodo class.} 
In \cite{PSW}, Pandharipande, Solomon, and Walcher constructed an OGW theory for the Fermat quintic. In turn, Aleshkin, Liu, and Walcher  conjectured an extension of the LG/CY correspondence to the open setting \cite{Walcher,Melissa}. The recent construction of \cite{TZ2} defines an open FJRW theory of the Fermat quintic with minimal symmetry group, which is a candidate LG dual side for the OGW construction of \cite{PSW}. However, it is not known at the moment how can one calculate the intersection numbers in the OFJRW side. In the closed setting, this question was first settled using the construction of the \emph{Chiodo class} \cite{ChiodoWitten}. Can such a construction be found in the open case? Can the open LG/CY be proven in this case, and in other cases for which the closed analogue has been proven?
\item {\bf A tropical version of OFJRW. } Recently, there has been interest in studying FJRW theory and $r$-spin theory through the lens of tropical geometry. For example, Abreu, Pacini and Secca investigated constructing a tropical counterpart to Jarvis's moduli space construction for $r$-roots \cite{APS}. Also, Cavalieri, Kelly, and Silversmith provides a tropical perspective on the closed $g=0$ $r$-spin theory on the way towards determining a closed formula for closed $g=0$ primary $r$-spin invariants \cite{cavalieri2024genus}. It is natural to look for an analogous tropical perspective on more general closed and open FJRW theories. It would also be interesting to provide a tropical geometry \emph{construction} of these theories.

\item {\bf Universal relations from open moduli.}
Pandharipande, Pixton and Zvonkine in \cite{PPZ} use Witten's $3$-spin theory to prove that Pixton's relations hold for the tautological ring of $\M_{g,n}$. We ask if open FJRW theory can be used as a guiding principle to understand relations in the moduli space of Riemann surfaces with boundary. While the moduli space of Riemann surfaces with boundary does not admit an algebraic intersection theory or tautological ring, this structure could present itself as relations that open Gromov-Witten invariants and open FJRW invariants universally satisfy.

\item {\bf Topological recursion and higher Airy structures.} In \cite{safnuk2016topological}, Safnuk  following Alexandrov \cite{Alexandrov} establishes topological recursion for $r=2$-spin open intersection numbers.  In \cite{borot2024higher}, the authors further underpin this structure using higher Airy structures. Can such theory be extended to higher $r$-spin theory for $r\geq 2$? Moreover, does there exist a theory of topological recursion that underpins open FJRW theory for Landau-Ginzburg models such as $(x_1^{r_1} + x_2^{r_2}, \mu_{r_1} \times \mu_{r_2})$ where open WDVV fails? 
\end{itemize}

\bibliography{OpenBiblio}

@article {LiLiSaitoShen,
    AUTHOR = {Li, Changzheng and Li, Si and Saito, Kyoji and Shen, Yefeng},
     TITLE = {Mirror symmetry for exceptional unimodular singularities},
   JOURNAL = {J. Eur. Math. Soc. (JEMS)},
  FJOURNAL = {Journal of the European Mathematical Society (JEMS)},
    VOLUME = {19},
      YEAR = {2017},
    NUMBER = {4},
     PAGES = {1189--1229},
      ISSN = {1435-9855,1435-9863},
   MRCLASS = {14J33},
  MRNUMBER = {3626554},
MRREVIEWER = {Eduardo\ A.\ Gonzalez},
       DOI = {10.4171/JEMS/691},
       URL = {https://doi.org/10.4171/JEMS/691},
}

@article {Overholser,
    AUTHOR = {Overholser, Peter},
     TITLE = {A descendent tropical {L}andau-{G}inzburg potential for {$\Bbb
              P^2$}},
   JOURNAL = {Commun. Number Theory Phys.},
  FJOURNAL = {Communications in Number Theory and Physics},
    VOLUME = {10},
      YEAR = {2016},
    NUMBER = {4},
     PAGES = {739--803},
      ISSN = {1931-4523},
   MRCLASS = {14N35 (14J33 14T05)},
  MRNUMBER = {3636674},
MRREVIEWER = {Amin Gholampour},
       DOI = {10.4310/CNTP.2016.v10.n4.a3},
       URL = {https://doi.org/10.4310/CNTP.2016.v10.n4.a3},
}

@article {GrossP2,
    AUTHOR = {Gross, Mark},
     TITLE = {Mirror symmetry for {$\Bbb P^2$} and tropical geometry},
   JOURNAL = {Adv. Math.},
  FJOURNAL = {Advances in Mathematics},
    VOLUME = {224},
      YEAR = {2010},
    NUMBER = {1},
     PAGES = {169--245},
      ISSN = {0001-8708}, 
   MRCLASS = {14J33 (14N35 14T05)},
  MRNUMBER = {2600995},
MRREVIEWER = {Jake Philip Solomon},
       DOI = {10.1016/j.aim.2009.11.007},
       URL = {https://doi.org/10.1016/j.aim.2009.11.007},
}

@article {FOOO1,
    AUTHOR = {Fukaya, Kenji and Oh, Yong-Geun and Ohta, Hiroshi and Ono,
              Kaoru},
     TITLE = {Lagrangian {F}loer theory on compact toric manifolds. {I}},
   JOURNAL = {Duke Math. J.},
  FJOURNAL = {Duke Mathematical Journal},
    VOLUME = {151},
      YEAR = {2010},
    NUMBER = {1},
     PAGES = {23--174},
      ISSN = {0012-7094},
     CODEN = {DUMJAO},
   MRCLASS = {53D40 (53D12)}, 
  MRNUMBER = {MR2573826},
       DOI = {10.1215/00127094-2009-062},
       URL = {http://dx.doi.org/10.1215/00127094-2009-062},

}

@article{sol_owdvv,
    author = {J.P. Solomon},
    title = {A differential equation for the open {G}romov-{W}itten potential},
    journal = {preprint},
    year = {October 2007}
}

@article {alexandrov2023construction,
    AUTHOR = {Alexandrov, Alexander and Basalaev, Alexey and Buryak,
              Alexandr},
     TITLE = {A construction of open descendant potentials in all genera},
   JOURNAL = {Int. Math. Res. Not. IMRN},
  FJOURNAL = {International Mathematics Research Notices. IMRN},
      YEAR = {2023},
    NUMBER = {17},
     PAGES = {14840--14889},
      ISSN = {1073-7928,1687-0247},
   MRCLASS = {14N35 (16S32 22E67 35Q76)},
  MRNUMBER = {4637453},
MRREVIEWER = {Miguel\ Moreira},
       DOI = {10.1093/imrn/rnac240},
       URL = {https://doi.org/10.1093/imrn/rnac240},
}

@article{fjrw_survey,
  title={A brief survey of {F}{J}{R}{W} theory},
  author={Francis, A and Jarvis, T and Priddis, Nathan},
  journal={Primitive Forms and Related Subjects (Kavli IPMU, 2014), Advanced Studies in Pure Mathematics},
  volume={83},
  pages={19--53},
  year={2014},
MRNUMBER = {4384380},
}

@article {chiodo2011lg,
    AUTHOR = {Chiodo, Alessandro and Ruan, Yongbin},
     TITLE = {L{G}/{CY} correspondence: the state space isomorphism},
   JOURNAL = {Adv. Math.},
  FJOURNAL = {Advances in Mathematics},
    VOLUME = {227},
      YEAR = {2011},
    NUMBER = {6},
     PAGES = {2157--2188},
      ISSN = {0001-8708,1090-2082},
   MRCLASS = {14J33 (14J32 55N32)},
  MRNUMBER = {2807086},
MRREVIEWER = {Yasuhiro\ Goto},
       DOI = {10.1016/j.aim.2011.04.011},
       URL = {https://doi.org/10.1016/j.aim.2011.04.011},
}

@article {fjrw_and_drinfeld_sokolov,
    AUTHOR = {Liu, Si-Qi and Ruan, Yongbin and Zhang, Youjin},
     TITLE = {B{CFG} {D}rinfeld-{S}okolov hierarchies and {FJRW}-theory},
   JOURNAL = {Invent. Math.},
  FJOURNAL = {Inventiones Mathematicae},
    VOLUME = {201},
      YEAR = {2015},
    NUMBER = {2},
     PAGES = {711--772},
      ISSN = {0020-9910,1432-1297},
   MRCLASS = {81T45 (14N35 37K10 53D45)},
  MRNUMBER = {3370624},
MRREVIEWER = {Xiaobin\ Li},
       DOI = {10.1007/s00222-014-0559-3},
       URL = {https://doi.org/10.1007/s00222-014-0559-3},
}

@article {BerglundHubsch,
    AUTHOR = {Berglund, Per and H\"{u}bsch, Tristan},
     TITLE = {A generalized construction of mirror manifolds},
   JOURNAL = {Nuclear Phys. B},
  FJOURNAL = {Nuclear Physics. B. Theoretical, Phenomenological, and
              Experimental High Energy Physics. Quantum Field Theory and
              Statistical Systems},
    VOLUME = {393},
      YEAR = {1993},
    NUMBER = {1-2},
     PAGES = {377--391},
      ISSN = {0550-3213,1873-1562},
   MRCLASS = {14J30 (32G81 32J17 81T40)},
  MRNUMBER = {1214325},
       DOI = {10.1016/0550-3213(93)90250-S},
       URL = {https://doi.org/10.1016/0550-3213(93)90250-S},
}

@book {Krawitz,
    AUTHOR = {Krawitz, Marc},
     TITLE = {F{JRW} rings and {L}andau-{G}inzburg mirror symmetry},
      NOTE = {Thesis (Ph.D.)--University of Michigan},
 PUBLISHER = {ProQuest LLC, Ann Arbor, MI},
      YEAR = {2010},
     PAGES = {67},
      ISBN = {978-1124-28080-6},
   MRCLASS = {Thesis},
  MRNUMBER = {2801653},
       URL =
              {http://gateway.proquest.com/openurl?url_ver=Z39.88-2004&rft_val_fmt=info:ofi/fmt:kev:mtx:dissertation&res_dat=xri:pqdiss&rft_dat=xri:pqdiss:3429350},
}

@book {Maher,
    AUTHOR = {Maher, Rob},
     TITLE = {Predictions in Open {F}an-{J}arvis-{R}uan-{W}itten theory via mirror symmetry, modularity, and wall-crossing},
      NOTE = {Thesis (Ph.D.)--University of Birmingham},
      YEAR = {2025},
}

@article {FJRSpin,
    AUTHOR = {Fan, Huijun and Jarvis, Tyler and Ruan, Yongbin},
     TITLE = {Quantum singularity theory for {$A_{(r-1)}$} and {$r$}-spin
              theory},
   JOURNAL = {Ann. Inst. Fourier (Grenoble)},
  FJOURNAL = {Universit\'{e} de Grenoble. Annales de l'Institut Fourier},
    VOLUME = {61},
      YEAR = {2011},
    NUMBER = {7},
     PAGES = {2781--2802},
      ISSN = {0373-0956,1777-5310},
   MRCLASS = {14D23 (14B05 53D45 57R56)},
  MRNUMBER = {3112508},
MRREVIEWER = {Hsian-Hua\ Tseng},
       DOI = {10.5802/aif.2794},
       URL = {https://doi.org/10.5802/aif.2794},
}

@article {Noumi,
    AUTHOR = {Noumi, Masatoshi},
     TITLE = {Expansion of the solutions of a {G}auss-{M}anin system at a
              point of infinity},
   JOURNAL = {Tokyo J. Math.},
  FJOURNAL = {Tokyo Journal of Mathematics},
    VOLUME = {7},
      YEAR = {1984},
    NUMBER = {1},
     PAGES = {1--60},
      ISSN = {0387-3870},
   MRCLASS = {32C40},
  MRNUMBER = {752110},
MRREVIEWER = {G.-M.\ Greuel},
       DOI = {10.3836/tjm/1270152991},
       URL = {https://doi.org/10.3836/tjm/1270152991},
}

@incollection {NoumiYamada,
    AUTHOR = {Noumi, Masatoshi and Yamada, Yasuhiko},
     TITLE = {Notes on the flat structures associated with simple and simply
              elliptic singularities},
 BOOKTITLE = {Integrable systems and algebraic geometry ({K}obe/{K}yoto,
              1997)},
     PAGES = {373--383},
 PUBLISHER = {World Sci. Publ., River Edge, NJ},
      YEAR = {1998},
      ISBN = {981-02-3266-7},
   MRCLASS = {32S25 (32S30 81T40)},
  MRNUMBER = {1672069},
MRREVIEWER = {Aleksandr\ G.\ Aleksandrov},
}

@article{solomon2006intersection,
  title={Intersection theory on the moduli space of holomorphic curves with Lagrangian boundary conditions},
  author={Solomon, Jake P},
  journal={arXiv preprint math/0606429},
  year={2006}
}

@article {solomon2023relative,
    AUTHOR = {Solomon, Jake P. and Tukachinsky, Sara B.},
     TITLE = {Relative quantum cohomology},
   JOURNAL = {J. Eur. Math. Soc. (JEMS)},
  FJOURNAL = {Journal of the European Mathematical Society (JEMS)},
    VOLUME = {26},
      YEAR = {2024},
    NUMBER = {9},
     PAGES = {3497--3573},
      ISSN = {1435-9855,1435-9863},
   MRCLASS = {53D37 (14N10 14N35 53D12 53D45)},
  MRNUMBER = {4767498},
MRREVIEWER = {Hsian-Hua\ Tseng},
       DOI = {10.4171/jems/1337},
       URL = {https://doi.org/10.4171/jems/1337},
}

@article{katz2006enumerative,
  title={Enumerative geometry of stable maps with Lagrangian boundary conditions and multiple covers of the disc},
  author={Katz, Sheldon and Liu, Chiu-Chu Melissa},
  journal={Geometry \& Topology Monographs},
  volume={8},
  pages={1--47},
  year={2006}
}

@article{BY15,
	AUTHOR = {Bertola, Marco and Yang, Di},
     TITLE = {The partition function of the extended {$r$}-reduced
              {K}adomtsev-{P}etviashvili hierarchy},
   JOURNAL = {J. Phys. A},
  FJOURNAL = {Journal of Physics. A. Mathematical and Theoretical},
    VOLUME = {48},
      YEAR = {2015},
    NUMBER = {19},
     PAGES = {195205, 20},
      ISSN = {1751-8113,1751-8121},
   MRCLASS = {35Q53},
  MRNUMBER = {3342766},
       DOI = {10.1088/1751-8113/48/19/195205},
       URL = {https://doi.org/10.1088/1751-8113/48/19/195205},
}

@article {polishchuk2011matrix,
    AUTHOR = {Polishchuk, Alexander and Vaintrob, Arkady},
     TITLE = {Matrix factorizations and cohomological field theories},
   JOURNAL = {J. Reine Angew. Math.},
  FJOURNAL = {Journal f\"{u}r die Reine und Angewandte Mathematik. [Crelle's
              Journal]},
    VOLUME = {714},
      YEAR = {2016},
     PAGES = {1--122},
      ISSN = {0075-4102,1435-5345},
   MRCLASS = {14N35 (14H10)},
  MRNUMBER = {3491884},
MRREVIEWER = {J\'{e}r\'{e}my\ C\'{e}dric\ Gu\'{e}r\'{e}},
       DOI = {10.1515/crelle-2014-0024},
       URL = {https://doi.org/10.1515/crelle-2014-0024},
}

@article{Bur15,
	AUTHOR = {Buryak, Alexandr},
     TITLE = {Equivalence of the open {K}d{V} and the open {V}irasoro
              equations for the moduli space of {R}iemann surfaces with
              boundary},
   JOURNAL = {Lett. Math. Phys.},
  FJOURNAL = {Letters in Mathematical Physics},
    VOLUME = {105},
      YEAR = {2015},
    NUMBER = {10},
     PAGES = {1427--1448},
      ISSN = {0377-9017,1573-0530},
   MRCLASS = {14H10 (35Q53)},
  MRNUMBER = {3395226},
MRREVIEWER = {Milagros\ Izquierdo},
       DOI = {10.1007/s11005-015-0789-3},
       URL = {https://doi.org/10.1007/s11005-015-0789-3},
}

@article{Bur16,
	AUTHOR = {Buryak, A.},
     TITLE = {Open intersection numbers and the wave function of the {K}d{V}
              hierarchy},
   JOURNAL = {Mosc. Math. J.},
  FJOURNAL = {Moscow Mathematical Journal},
    VOLUME = {16},
      YEAR = {2016},
    NUMBER = {1},
     PAGES = {27--44},
      ISSN = {1609-3321,1609-4514},
   MRCLASS = {35Q53 (14H10)},
  MRNUMBER = {3470575},
MRREVIEWER = {Wen-Xiu\ Ma},
       DOI = {10.17323/1609-4514-2016-16-1-27-44},
       URL = {https://doi.org/10.17323/1609-4514-2016-16-1-27-44},
}

@article{Moc06,
	AUTHOR = {Mochizuki, Takuro},
     TITLE = {The virtual class of the moduli stack of stable {$r$}-spin
              curves},
   JOURNAL = {Comm. Math. Phys.},
  FJOURNAL = {Communications in Mathematical Physics},
    VOLUME = {264},
      YEAR = {2006},
    NUMBER = {1},
     PAGES = {1--40},
      ISSN = {0010-3616,1432-0916},
   MRCLASS = {14H10 (14N35)},
  MRNUMBER = {2211733},
MRREVIEWER = {Dan\ Abramovich},
       DOI = {10.1007/s00220-006-1538-3},
       URL = {https://doi.org/10.1007/s00220-006-1538-3},
}

@article{BT17,
	AUTHOR = {Buryak, Alexandr and Tessler, Ran J.},
     TITLE = {Matrix models and a proof of the open analog of {W}itten's
              conjecture},
   JOURNAL = {Comm. Math. Phys.},
  FJOURNAL = {Communications in Mathematical Physics},
    VOLUME = {353},
      YEAR = {2017},
    NUMBER = {3},
     PAGES = {1299--1328},
      ISSN = {0010-3616,1432-0916},
   MRCLASS = {14H10 (14H81 14N35)},
  MRNUMBER = {3652492},
MRREVIEWER = {Guangbo\ Xu},
       DOI = {10.1007/s00220-017-2899-5},
       URL = {https://doi.org/10.1007/s00220-017-2899-5},
}

@article{FSZ10,
	AUTHOR = {Faber, Carel and Shadrin, Sergey and Zvonkine, Dimitri},
     TITLE = {Tautological relations and the {$r$}-spin {W}itten conjecture},
   JOURNAL = {Ann. Sci. \'{E}c. Norm. Sup\'{e}r. (4)},
  FJOURNAL = {Annales Scientifiques de l'\'{E}cole Normale Sup\'{e}rieure.
              Quatri\`eme S\'{e}rie},
    VOLUME = {43},
      YEAR = {2010},
    NUMBER = {4},
     PAGES = {621--658},
      ISSN = {0012-9593,1873-2151},
   MRCLASS = {14N35 (14H10 53D45)},
  MRNUMBER = {2722511},
MRREVIEWER = {Hsian-Hua\ Tseng},
       DOI = {10.24033/asens.2130},
       URL = {https://doi.org/10.24033/asens.2130},
}

@article{pixton2012conjectural,
  title={Conjectural relations in the tautological ring of 
$\overline{M}_{g, n}$},
  author={Pixton, Aaron},
  journal={arXiv preprint arXiv:1207.1918},
  year={2012}
}

@article{FaveroKim,
  title={General {G}{L}{S}{M} Invariants and Their Cohomological Field Theories},
  author={Favero, David and Kim, Bumsig},
  journal={arXiv preprint arXiv:2006.12182},
  year={2020}
}

@article {CFFGKS,
    AUTHOR = {Ciocan-Fontanine, Ionut and Favero, David and Gu\'{e}r\'{e},
              J\'{e}r\'{e}my and Kim, Bumsig and Shoemaker, Mark},
     TITLE = {Fundamental factorization of a {GLSM} {P}art {I}:
              {C}onstruction},
   JOURNAL = {Mem. Amer. Math. Soc.},
  FJOURNAL = {Memoirs of the American Mathematical Society},
    VOLUME = {289},
      YEAR = {2023},
    NUMBER = {1435},
     PAGES = {iv+96},
      ISSN = {0065-9266,1947-6221},
      ISBN = {978-1-4704-6543-8; 978-1-4704-7590-1},
   MRCLASS = {14N35 (14F08 53D45)},
  MRNUMBER = {4632308},
MRREVIEWER = {Sergiy\ Koshkin},
       DOI = {10.1090/memo/1435},
       URL = {https://doi.org/10.1090/memo/1435},
}

@article{givental2001gromov,
  title={Gromov-Witten invariants and quantization of quadratic Hamiltonians},
  author={Givental, Alexander B},
  journal={arXiv preprint math/0108100},
  year={2001}
}

@article {givental2001semisimple,
    AUTHOR = {Givental, Alexander B.},
     TITLE = {Semisimple {F}robenius structures at higher genus},
   JOURNAL = {Internat. Math. Res. Notices},
  FJOURNAL = {International Mathematics Research Notices},
      YEAR = {2001},
    NUMBER = {23},
     PAGES = {1265--1286},
      ISSN = {1073-7928,1687-0247},
   MRCLASS = {53D45 (14N35)},
  MRNUMBER = {1866444},
MRREVIEWER = {Gilberto\ Bini},
       DOI = {10.1155/S1073792801000605},
       URL = {https://doi.org/10.1155/S1073792801000605},
}

@article {behrend1997intrinsic,
    AUTHOR = {Behrend, K. and Fantechi, B.},
     TITLE = {The intrinsic normal cone},
   JOURNAL = {Invent. Math.},
  FJOURNAL = {Inventiones Mathematicae},
    VOLUME = {128},
      YEAR = {1997},
    NUMBER = {1},
     PAGES = {45--88},
      ISSN = {0020-9910,1432-1297},
   MRCLASS = {14F99 (14C15 14D20)},
  MRNUMBER = {1437495},
MRREVIEWER = {Tohru\ Nakashima},
       DOI = {10.1007/s002220050136},
       URL = {https://doi.org/10.1007/s002220050136},
}

@article{FJR,
	AUTHOR = {Fan, Huijun and Jarvis, Tyler and Ruan, Yongbin},
     TITLE = {The {W}itten equation, mirror symmetry, and quantum
              singularity theory},
   JOURNAL = {Ann. of Math. (2)},
  FJOURNAL = {Annals of Mathematics. Second Series},
    VOLUME = {178},
      YEAR = {2013},
    NUMBER = {1},
     PAGES = {1--106},
      ISSN = {0003-486X,1939-8980},
   MRCLASS = {81T70 (14B05 14N35 35Q53)},
  MRNUMBER = {3043578},
MRREVIEWER = {Johannes\ Walcher},
       DOI = {10.4007/annals.2013.178.1.1},
       URL = {https://doi.org/10.4007/annals.2013.178.1.1},
}

@article{FJR2,
	AUTHOR = {Fan, Huijun and Jarvis, Tyler J. and Ruan, Yongbin},
     TITLE = {Geometry and analysis of spin equations},
   JOURNAL = {Comm. Pure Appl. Math.},
  FJOURNAL = {Communications on Pure and Applied Mathematics},
    VOLUME = {61},
      YEAR = {2008},
    NUMBER = {6},
     PAGES = {745--788},
      ISSN = {0010-3640,1097-0312},
   MRCLASS = {32G13 (14D21 32W50)},
  MRNUMBER = {2400605},
MRREVIEWER = {Fabio\ Perroni},
       DOI = {10.1002/cpa.20246},
       URL = {https://doi.org/10.1002/cpa.20246},
}

@unpublished{Nill,
Author ={S. Nill}, 
Title={Extended {F}{J}{R}{W} theory of the quintic threefold in genus zero},
Note={talk in S\'eminaire de G\'eom\'etrie Enum\'erative, 2 December 2021}}

@article {alexandrov2017refined,
    AUTHOR = {Alexandrov, Alexander and Buryak, Alexandr and Tessler, Ran
              J.},
     TITLE = {Refined open intersection numbers and the
              {K}ontsevich-{P}enner matrix model},
   JOURNAL = {J. High Energy Phys.},
  FJOURNAL = {Journal of High Energy Physics},
      YEAR = {2017},
    NUMBER = {3},
     PAGES = {123, front matter+40},
      ISSN = {1126-6708,1029-8479},
   MRCLASS = {81T30 (30F20 32G15)},
  MRNUMBER = {3650694},
       DOI = {10.1007/JHEP03(2017)123},
       URL = {https://doi.org/10.1007/JHEP03(2017)123},
}

@unpublished{FJR3,
	Author = {Fan, H. and Jarvis, T. J. and Ruan, Y.},
	Note = {arXiv preprint arXiv:0712.4025},
	Title = {The {W}itten equation and its virtual fundamental cycle},
	Year = {2007}}

@article{FJRGLSM,
	AUTHOR = {Fan, Huijun and Jarvis, Tyler J. and Ruan, Yongbin},
     TITLE = {A mathematical theory of the gauged linear sigma model},
   JOURNAL = {Geom. Topol.},
  FJOURNAL = {Geometry \& Topology},
    VOLUME = {22},
      YEAR = {2018},
    NUMBER = {1},
     PAGES = {235--303},
      ISSN = {1465-3060,1364-0380},
   MRCLASS = {14N35 (14D23 14J32 14L24 53D45 81T40 81T60)},
  MRNUMBER = {3720344},
MRREVIEWER = {Hsian-Hua\ Tseng},
       DOI = {10.2140/gt.2018.22.235},
       URL = {https://doi.org/10.2140/gt.2018.22.235},
}

@article{ChiodoStable,
	AUTHOR = {Chiodo, Alessandro},
     TITLE = {Stable twisted curves and their {$r$}-spin structures},
   JOURNAL = {Ann. Inst. Fourier (Grenoble)},
  FJOURNAL = {Universit\'{e} de Grenoble. Annales de l'Institut Fourier},
    VOLUME = {58},
      YEAR = {2008},
    NUMBER = {5},
     PAGES = {1635--1689},
      ISSN = {0373-0956,1777-5310},
   MRCLASS = {14H10 (14H60)},
  MRNUMBER = {2445829},
MRREVIEWER = {Arvid\ Siqveland},
       URL = {http://aif.cedram.org/item?id=AIF_2008__58_5_1635_0},
}

@article{ChiodoWitten,
	 AUTHOR = {Chiodo, Alessandro},
     TITLE = {The {W}itten top {C}hern class via {$K$}-theory},
   JOURNAL = {J. Algebraic Geom.},
  FJOURNAL = {Journal of Algebraic Geometry},
    VOLUME = {15},
      YEAR = {2006},
    NUMBER = {4},
     PAGES = {681--707},
      ISSN = {1056-3911,1534-7486},
   MRCLASS = {14H10 (14C35 14N35)},
  MRNUMBER = {2237266},
MRREVIEWER = {Dan\ Abramovich},
       DOI = {10.1090/S1056-3911-06-00444-9},
       URL = {https://doi.org/10.1090/S1056-3911-06-00444-9},
}

@incollection{PV,
	AUTHOR = {Polishchuk, Alexander and Vaintrob, Arkady},
     TITLE = {Algebraic construction of {W}itten's top {C}hern class},
 BOOKTITLE = {Advances in algebraic geometry motivated by physics ({L}owell,
              {MA}, 2000)},
    SERIES = {Contemp. Math.},
    VOLUME = {276},
     PAGES = {229--249},
 PUBLISHER = {Amer. Math. Soc., Providence, RI},
      YEAR = {2001},
      ISBN = {0-8218-2810-X},
   MRCLASS = {14C17 (14H10 14N35)},
  MRNUMBER = {1837120},
MRREVIEWER = {Tyler\ J.\ Jarvis},
       DOI = {10.1090/conm/276/04523},
       URL = {https://doi.org/10.1090/conm/276/04523},
}

@article{CLL,
	AUTHOR = {Chang, Huai-Liang and Li, Jun and Li, Wei-Ping},
     TITLE = {Witten's top {C}hern class via cosection localization},
   JOURNAL = {Invent. Math.},
  FJOURNAL = {Inventiones Mathematicae},
    VOLUME = {200},
      YEAR = {2015},
    NUMBER = {3},
     PAGES = {1015--1063},
      ISSN = {0020-9910,1432-1297},
   MRCLASS = {14D23 (14H99)},
  MRNUMBER = {3348143},
MRREVIEWER = {Stefan\ Schr\"{o}er},
       DOI = {10.1007/s00222-014-0549-5},
       URL = {https://doi.org/10.1007/s00222-014-0549-5},
}

@unpublished{Melissa,
  title={{Open/closed Correspondence and Extended {LG/CY} Correspondence for Quintic Threefolds}},
  author={Aleshkin, Konstantin and Liu, Chiu-Chu Melissa},
  note={arXiv preprint arXiv:2309.14628},
  year={2023}
}

@article {Walcher,
    AUTHOR = {Walcher, Johannes},
     TITLE = {Evidence for tadpole cancellation in the topological string},
   JOURNAL = {Commun. Number Theory Phys.},
  FJOURNAL = {Communications in Number Theory and Physics},
    VOLUME = {3},
      YEAR = {2009},
    NUMBER = {1},
     PAGES = {111--172},
      ISSN = {1931-4523,1931-4531},
   MRCLASS = {14J33 (14J32 14N35 81T45 81T50)},
  MRNUMBER = {2504755},
MRREVIEWER = {Jake\ Philip\ Solomon},
       DOI = {10.4310/CNTP.2009.v3.n1.a3},
       URL = {https://doi.org/10.4310/CNTP.2009.v3.n1.a3},
}

@unpublished{ST_unpublished,
    author = {Solomon, Jake P. and Tessler, Ran J.},
    note = {Unpublished} 
}

@book{FOOO_ii,
  title={Lagrangian intersection Floer theory: anomaly and obstruction, Part II},
  author={Fukaya, Kenji and Oh, Yong-Geun and Ohta, Hiroshi and Ono, Kaoru},
  volume={2},
  year={2010},
  publisher={American Mathematical Soc.}
}

@article{Teleman,
	AUTHOR = {Teleman, Constantin},
     TITLE = {The structure of 2{D} semi-simple field theories},
   JOURNAL = {Invent. Math.},
  FJOURNAL = {Inventiones Mathematicae},
    VOLUME = {188},
      YEAR = {2012},
    NUMBER = {3},
     PAGES = {525--588},
      ISSN = {0020-9910,1432-1297},
   MRCLASS = {57R56 (18D10 53D45)},
  MRNUMBER = {2917177},
MRREVIEWER = {Julia\ Bergner},
       DOI = {10.1007/s00222-011-0352-5},
       URL = {https://doi.org/10.1007/s00222-011-0352-5},
}

@article {lee2014mirror,
    AUTHOR = {Lee, Yuan-Pin and Shoemaker, Mark},
     TITLE = {A mirror theorem for the mirror quintic},
   JOURNAL = {Geom. Topol.},
  FJOURNAL = {Geometry \& Topology},
    VOLUME = {18},
      YEAR = {2014},
    NUMBER = {3},
     PAGES = {1437--1483},
      ISSN = {1465-3060,1364-0380},
   MRCLASS = {14N35 (14J33 53D45)},
  MRNUMBER = {3228456},
MRREVIEWER = {Jake\ Philip\ Solomon},
       DOI = {10.2140/gt.2014.18.1437},
       URL = {https://doi.org/10.2140/gt.2014.18.1437},
}

@article {priddis2016landau,
    AUTHOR = {Priddis, Nathan and Shoemaker, Mark},
     TITLE = {A {L}andau-{G}inzburg/{C}alabi-{Y}au correspondence for the
              mirror quintic},
   JOURNAL = {Ann. Inst. Fourier (Grenoble)},
  FJOURNAL = {Universit\'{e} de Grenoble. Annales de l'Institut Fourier},
    VOLUME = {66},
      YEAR = {2016},
    NUMBER = {3},
     PAGES = {1045--1091},
      ISSN = {0373-0956,1777-5310},
   MRCLASS = {14N35 (14J33 32G20 53D45)},
  MRNUMBER = {3494166},
MRREVIEWER = {Emily\ Clader},
       DOI = {10.5802/aif.3031},
       URL = {https://doi.org/10.5802/aif.3031},
}

@article {priddis2014proof,
    AUTHOR = {Lee, Yuan-Pin and Priddis, Nathan and Shoemaker, Mark},
     TITLE = {A proof of the {L}andau-{G}inzburg/{C}alabi-{Y}au
              correspondence via the crepant transformation conjecture},
   JOURNAL = {Ann. Sci. \'{E}c. Norm. Sup\'{e}r. (4)},
  FJOURNAL = {Annales Scientifiques de l'\'{E}cole Normale Sup\'{e}rieure.
              Quatri\`eme S\'{e}rie},
    VOLUME = {49},
      YEAR = {2016},
    NUMBER = {6},
     PAGES = {1403--1443},
      ISSN = {0012-9593,1873-2151},
   MRCLASS = {14N35 (14J33 53D45)},
  MRNUMBER = {3592361},
MRREVIEWER = {Felix\ Janda},
       DOI = {10.24033/asens.2312},
       URL = {https://doi.org/10.24033/asens.2312},
}

@article{krawitz2011landau,
  title={Landau-{G}inzburg/{C}alabi-{Y}au Correspondence of all Genera for Elliptic Orbifold $\mathbb{P}^1$},
  author={Krawitz, Marc and Shen, Yefeng},
  journal={arXiv preprint arXiv:1106.6270},
  year={2011}
}

@article {Alexandrov,
    AUTHOR = {Alexandrov, Alexander},
     TITLE = {Open intersection numbers, {K}ontsevich-{P}enner model and
              cut-and-join operators},
   JOURNAL = {J. High Energy Phys.},
  FJOURNAL = {Journal of High Energy Physics},
      YEAR = {2015},
    NUMBER = {8},
     PAGES = {028, front matter+24},
      ISSN = {1126-6708,1029-8479},
   MRCLASS = {81R05},
  MRNUMBER = {3402137},
       DOI = {10.1007/JHEP08(2015)028},
       URL = {https://doi.org/10.1007/JHEP08(2015)028},
}

@article {APS,
    AUTHOR = {Abreu, Alex and Pacini, Marco and Secco, Matheus},
     TITLE = {On moduli spaces of roots in algebraic and tropical geometry},
   JOURNAL = {J. Algebra},
  FJOURNAL = {Journal of Algebra},
    VOLUME = {634},
      YEAR = {2023},
     PAGES = {832--872},
      ISSN = {0021-8693,1090-266X},
   MRCLASS = {14T20 (14D20 14H10 14H40)},
  MRNUMBER = {4633697},
       DOI = {10.1016/j.jalgebra.2023.06.015},
       URL = {https://doi.org/10.1016/j.jalgebra.2023.06.015},
}

@article{milanov2011gromov,
  title={Gromov-Witten theory of elliptic orbifold $\mathbb{P}^1$ and quasi-modular forms},
  author={Milanov, Todor and Ruan, Yongbin},
  journal={arXiv preprint arXiv:1106.2321},
  year={2011}
}

@article {milanov2016global,
    AUTHOR = {Milanov, Todor and Shen, Yefeng},
     TITLE = {Global mirror symmetry for invertible simple elliptic
              singularities},
   JOURNAL = {Ann. Inst. Fourier (Grenoble)},
  FJOURNAL = {Universit\'{e} de Grenoble. Annales de l'Institut Fourier},
    VOLUME = {66},
      YEAR = {2016},
    NUMBER = {1},
     PAGES = {271--330},
      ISSN = {0373-0956,1777-5310},
   MRCLASS = {14N35 (14B05 14J33)},
  MRNUMBER = {3477877},
MRREVIEWER = {Hsian-Hua\ Tseng},
       DOI = {10.5802/aif.3012},
       URL = {https://doi.org/10.5802/aif.3012},
}

@article {clader2017landau,
    AUTHOR = {Clader, Emily},
     TITLE = {Landau-{G}inzburg/{C}alabi-{Y}au correspondence for the
              complete intersections {$X_{3,3}$} and {$X_{2,2,2,2}$}},
   JOURNAL = {Adv. Math.},
  FJOURNAL = {Advances in Mathematics},
    VOLUME = {307},
      YEAR = {2017},
     PAGES = {1--52},
      ISSN = {0001-8708,1090-2082},
   MRCLASS = {14J33 (14N35)},
  MRNUMBER = {3590512},
MRREVIEWER = {Tyler\ L.\ Kelly},
       DOI = {10.1016/j.aim.2016.11.010},
       URL = {https://doi.org/10.1016/j.aim.2016.11.010},
}

@article {clader2018sigma,
    AUTHOR = {Clader, Emily and Ross, Dustin},
     TITLE = {Sigma models and phase transitions for complete intersections},
   JOURNAL = {Int. Math. Res. Not. IMRN},
  FJOURNAL = {International Mathematics Research Notices. IMRN},
      YEAR = {2018},
    NUMBER = {15},
     PAGES = {4799--4851},
      ISSN = {1073-7928,1687-0247},
   MRCLASS = {14N35},
  MRNUMBER = {3842378},
MRREVIEWER = {Xiaobin\ Li},
       DOI = {10.1093/imrn/rnx029},
       URL = {https://doi.org/10.1093/imrn/rnx029},
}

@article {cavalieri2024genus,
    AUTHOR = {Cavalieri, Renzo and Kelly, Tyler L. and Silversmith, Rob},
     TITLE = {Genus-zero {$r$}-spin theory},
   JOURNAL = {Moduli},
  FJOURNAL = {Moduli},
    VOLUME = {1},
      YEAR = {2024},
     PAGES = {Paper No. e7, 35},
      ISSN = {2949-7647,2977-1382},
      ISBN = {},
   MRCLASS = {14N10 (14C17 14H10 14T15 14T90)},
  MRNUMBER = {4892593},
       DOI = {10.1112/mod.2024.2},
       URL = {https://doi.org/10.1112/mod.2024.2},
}

@article {zhao2022landau,
    AUTHOR = {Zhao, Yizhen},
     TITLE = {Landau-{G}inzburg/{C}alabi-{Y}au correspondence for a complete
              intersection via matrix factorizations},
   JOURNAL = {Int. Math. Res. Not. IMRN},
  FJOURNAL = {International Mathematics Research Notices. IMRN},
      YEAR = {2022},
    NUMBER = {15},
     PAGES = {11796--11863},
      ISSN = {1073-7928,1687-0247},
   MRCLASS = {14J33 (14N35)},
  MRNUMBER = {4458566},
       DOI = {10.1093/imrn/rnab044},
       URL = {https://doi.org/10.1093/imrn/rnab044},
}

@article{acosta2014asymptotic,
  title={Asymptotic expansion and the {LG}/({F}ano, general type) correspondence},
  author={Acosta, Pedro},
  journal={arXiv preprint arXiv:1411.4162},
  year={2014}
}

@incollection{Witten2DGravity,
	AUTHOR = {Witten, Edward},
     TITLE = {Two-dimensional gravity and intersection theory on moduli
              space},
 BOOKTITLE = {Surveys in differential geometry ({C}ambridge, {MA}, 1990)},
     PAGES = {243--310},
 PUBLISHER = {Lehigh Univ., Bethlehem, PA},
      YEAR = {1991},
      ISBN = {0-8218-0168-6},
   MRCLASS = {32G15 (14C17 14H15 32G81 58F07 81T40)},
  MRNUMBER = {1144529},
MRREVIEWER = {Steven\ Rosenberg},
}

@incollection{Witten93,
	AUTHOR = {Witten, Edward},
     TITLE = {Algebraic geometry associated with matrix models of
              two-dimensional gravity},
 BOOKTITLE = {Topological methods in modern mathematics ({S}tony {B}rook,
              {NY}, 1991)},
     PAGES = {235--269},
 PUBLISHER = {Publish or Perish, Houston, TX},
      YEAR = {1993},
   MRCLASS = {32G15 (14H15 81T40)},
  MRNUMBER = {1215968},
MRREVIEWER = {Claude\ Itzykson},
}

@article{Jarvis,
	AUTHOR = {Jarvis, Tyler J.},
     TITLE = {Geometry of the moduli of higher spin curves},
   JOURNAL = {Internat. J. Math.},
  FJOURNAL = {International Journal of Mathematics},
    VOLUME = {11},
      YEAR = {2000},
    NUMBER = {5},
     PAGES = {637--663},
      ISSN = {0129-167X,1793-6519},
   MRCLASS = {14H10},
  MRNUMBER = {1780734},
MRREVIEWER = {Jos\'{e}\ M.\ Mu\~{n}oz Porras},
       DOI = {10.1142/S0129167X00000325},
       URL = {https://doi.org/10.1142/S0129167X00000325},
}

@article {PST14,
    AUTHOR = {Pandharipande, Rahul and Solomon, Jake P. and Tessler, Ran J.},
     TITLE = {Intersection theory on moduli of disks, open {K}d{V} and
              {V}irasoro},
   JOURNAL = {Geom. Topol.},
  FJOURNAL = {Geometry \& Topology},
    VOLUME = {28},
      YEAR = {2024},
    NUMBER = {6},
     PAGES = {2483--2567},
      ISSN = {1465-3060,1364-0380},
   MRCLASS = {14H15 (14N35 32G15 37K20 53D45)},
  MRNUMBER = {4817469},
       DOI = {10.2140/gt.2024.28.2483},
       URL = {https://doi.org/10.2140/gt.2024.28.2483},
}

@article{fukaya2009lagrangian,
  title={Lagrangian intersection Floer theory—anomaly and obstruction, Parts I \& II},
  author={Fukaya, Kenji and Oh, Yong-Geun and Ohta, Hiroshi and Ono, Kaoru},
  journal={AMS/IP Studies in Advanced Mathematics},
  volume={46},
  number={46.2},
  year={2009}
}

@article{BCT3,
  AUTHOR = {Buryak, Alexandr and Clader, Emily and Tessler, Ran J.},
     TITLE = {Open {$r$}-spin theory {III}: {A} prediction for higher genus},
   JOURNAL = {J. Geom. Phys.},
  FJOURNAL = {Journal of Geometry and Physics},
    VOLUME = {192},
      YEAR = {2023},
     PAGES = {Paper No. 104960, 12},
      ISSN = {0393-0440,1879-1662},
   MRCLASS = {14H10},
  MRNUMBER = {4629749},
       DOI = {10.1016/j.geomphys.2023.104960},
       URL = {https://doi.org/10.1016/j.geomphys.2023.104960},
}

@article{KontsManin,
  AUTHOR = {Kontsevich, M. and Manin, Yu.},
     TITLE = {Gromov-{W}itten classes, quantum cohomology, and enumerative
              geometry},
   JOURNAL = {Comm. Math. Phys.},
  FJOURNAL = {Communications in Mathematical Physics},
    VOLUME = {164},
      YEAR = {1994},
    NUMBER = {3},
     PAGES = {525--562},
      ISSN = {0010-3616,1432-0916},
   MRCLASS = {14N10 (53C15 58D10 58F05)},
  MRNUMBER = {1291244},
MRREVIEWER = {Dietmar\ A.\ Salamon},
       URL = {http://projecteuclid.org/euclid.cmp/1104270948},
}

@unpublished{Zernik,
	Author = {{Netser Zernik}, A.},
	Note = {arXiv preprint arXiv:1709.07402},
	Title = {{M}oduli of {O}pen {S}table {M}aps to a {H}omogeneous {S}pace},
	Year = {2017}}

@article{JKV,
	AUTHOR = {Jarvis, Tyler J. and Kimura, Takashi and Vaintrob, Arkady},
     TITLE = {Moduli spaces of higher spin curves and integrable
              hierarchies},
   JOURNAL = {Compositio Math.},
  FJOURNAL = {Compositio Mathematica},
    VOLUME = {126},
      YEAR = {2001},
    NUMBER = {2},
     PAGES = {157--212},
      ISSN = {0010-437X,1570-5846},
   MRCLASS = {14H70 (14H10 14H81 37K20)},
  MRNUMBER = {1827643},
MRREVIEWER = {Gilberto\ Bini},
       DOI = {10.1023/A:1017528003622},
       URL = {https://doi.org/10.1023/A:1017528003622},
}

@incollection{JKV2,
	AUTHOR = {Jarvis, Tyler J. and Kimura, Takashi and Vaintrob, Arkady},
     TITLE = {Gravitational descendants and the moduli space of higher spin
              curves},
 BOOKTITLE = {Advances in algebraic geometry motivated by physics ({L}owell,
              {MA}, 2000)},
    SERIES = {Contemp. Math.},
    VOLUME = {276},
     PAGES = {167--177},
 PUBLISHER = {Amer. Math. Soc., Providence, RI},
      YEAR = {2001},
      ISBN = {0-8218-2810-X},
   MRCLASS = {14H10 (14H70 14H81 14N35)},
  MRNUMBER = {1837117},
MRREVIEWER = {Dan\ Abramovich},
       DOI = {10.1090/conm/276/04520},
       URL = {https://doi.org/10.1090/conm/276/04520},
}

@article{BPTZ,
	AUTHOR = {Buryak, Alexandr and Zernik, Amitai Netser and Pandharipande,
              Rahul and Tessler, Ran J.},
     TITLE = {Open {$\mathbb{CP}^1$} descendent theory {I}: {T}he stationary
              sector},
   JOURNAL = {Adv. Math.},
  FJOURNAL = {Advances in Mathematics},
    VOLUME = {401},
      YEAR = {2022},
     PAGES = {Paper No. 108249, 93},
      ISSN = {0001-8708,1090-2082},
   MRCLASS = {14D21 (14N35 53D45)},
  MRNUMBER = {4387849},
MRREVIEWER = {Ritwik\ Mukherjee},
       DOI = {10.1016/j.aim.2022.108249},
       URL = {https://doi.org/10.1016/j.aim.2022.108249},
}

@unpublished{Liu,
	Author = {Liu, Chiu-Chu Melissa},
	Note = {arXiv preprint arXiv:math/0210257},
	Title = {Moduli of {J}-Holomorphic {C}urves with {L}agrangian {B}oundary {C}onditions and {O}pen {G}romov-{W}itten {I}nvariants for an
{$S^1$}-{E}quivariant {P}air}}

@article {gomez2021open,
    AUTHOR = {Brauer Gomez, Oscar and Buryak, Alexandr},
     TITLE = {Open topological recursion relations in genus 1 and integrable
              systems},
   JOURNAL = {J. High Energy Phys.},
  FJOURNAL = {Journal of High Energy Physics},
      YEAR = {2021},
    NUMBER = {1},
     PAGES = {Paper No. 048, 14},
      ISSN = {1126-6708,1029-8479},
   MRCLASS = {81R12 (14H81 14N35 53D45 81T45)},
  MRNUMBER = {4258256},
MRREVIEWER = {Hsian-Hua\ Tseng},
       DOI = {10.1007/jhep01(2021)048},
       URL = {https://doi.org/10.1007/jhep01(2021)048},
}

@article {borot2024higher,
    AUTHOR = {Borot, Ga\"{e}tan and Bouchard, Vincent and Chidambaram, Nitin
              K. and Creutzig, Thomas and Noshchenko, Dmitry},
     TITLE = {Higher {A}iry structures, {$\mathcal W$} algebras and topological
              recursion},
   JOURNAL = {Mem. Amer. Math. Soc.},
  FJOURNAL = {Memoirs of the American Mathematical Society},
    VOLUME = {296},
      YEAR = {2024},
    NUMBER = {1476},
     PAGES = {v+108},
      ISSN = {0065-9266,1947-6221},
      ISBN = {978-1-4704-6906-1; 978-1-4704-7813-1},
   MRCLASS = {81R10 (14H81 14N35)},
  MRNUMBER = {4744795},
MRREVIEWER = {Reinier\ Kramer},
       DOI = {10.1090/memo/1476},
       URL = {https://doi.org/10.1090/memo/1476},
}

@article {safnuk2016topological,
    AUTHOR = {Safnuk, Brad},
     TITLE = {Topological recursion for open intersection numbers},
   JOURNAL = {Commun. Number Theory Phys.},
  FJOURNAL = {Communications in Number Theory and Physics},
    VOLUME = {10},
      YEAR = {2016},
    NUMBER = {4},
     PAGES = {833--857},
      ISSN = {1931-4523,1931-4531},
   MRCLASS = {14H10 (14H81 14N35)},
  MRNUMBER = {3636676},
MRREVIEWER = {Felix\ Janda},
       DOI = {10.4310/CNTP.2016.v10.n4.a5},
       URL = {https://doi.org/10.4310/CNTP.2016.v10.n4.a5},
}

@article{wang2025identification,
  title={The identification of the extended refined open partition function and the {K}ontsevich-{P}enner matrix model},
  author={Wang, Gehao},
  journal={arXiv preprint arXiv:2511.16919},
  year={2025}
}

@article {dijkgraaf2018developments,
    AUTHOR = {Dijkgraaf, Robbert and Witten, Edward},
     TITLE = {Developments in topological gravity},
   JOURNAL = {Internat. J. Modern Phys. A},
  FJOURNAL = {International Journal of Modern Physics A. Particles and
              Fields. Gravitation. Cosmology. Astrophysics. Accelerator
              Physics},
    VOLUME = {33},
      YEAR = {2018},
    NUMBER = {30},
     PAGES = {1830029, 63},
      ISSN = {0217-751X,1793-656X},
   MRCLASS = {32G15 (14H15 14H81 81T45 83C45)},
  MRNUMBER = {3875934},
MRREVIEWER = {Peter\ R.\ Law},
       DOI = {10.1142/S0217751X18300296},
       URL = {https://doi.org/10.1142/S0217751X18300296},
}

@article {Tes15,
    AUTHOR = {Tessler, Ran J.},
     TITLE = {The combinatorial formula for open gravitational descendents},
   JOURNAL = {Geom. Topol.},
  FJOURNAL = {Geometry \& Topology},
    VOLUME = {27},
      YEAR = {2023},
    NUMBER = {7},
     PAGES = {2497--2648},
      ISSN = {1465-3060,1364-0380},
   MRCLASS = {53D45},
  MRNUMBER = {4645483},
MRREVIEWER = {Alex\ Massarenti},
       DOI = {10.2140/gt.2023.27.2497},
       URL = {https://doi.org/10.2140/gt.2023.27.2497},
}

@article{BCT_Closed_Extended,
	AUTHOR = {Buryak, Alexandr and Clader, Emily and Tessler, Ran J.},
     TITLE = {Closed extended {$r$}-spin theory and the {G}elfand-{D}ickey
              wave function},
   JOURNAL = {J. Geom. Phys.},
  FJOURNAL = {Journal of Geometry and Physics},
    VOLUME = {137},
      YEAR = {2019},
     PAGES = {132--153},
      ISSN = {0393-0440,1879-1662},
   MRCLASS = {14H10 (35Q53 37K10)},
  MRNUMBER = {3894288},
MRREVIEWER = {Dmitry\ Zakharov},
       DOI = {10.1016/j.geomphys.2018.11.007},
       URL = {https://doi.org/10.1016/j.geomphys.2018.11.007},
}

@article{BCT1,
	AUTHOR = {Buryak, Alexandr and Clader, Emily and Tessler, Ran J.},
     TITLE = {Open {$r$}-spin theory {I}: {F}oundations},
   JOURNAL = {Int. Math. Res. Not. IMRN},
  FJOURNAL = {International Mathematics Research Notices. IMRN},
      YEAR = {2022},
    NUMBER = {14},
     PAGES = {10458--10532},
      ISSN = {1073-7928,1687-0247},
   MRCLASS = {14H10},
  MRNUMBER = {4458550},
MRREVIEWER = {Alex\ Massarenti},
       DOI = {10.1093/imrn/rnaa345},
       URL = {https://doi.org/10.1093/imrn/rnaa345},
}

@article{PPZ,
	 AUTHOR = {Pandharipande, Rahul and Pixton, Aaron and Zvonkine, Dimitri},
     TITLE = {Relations on {$\overline{\mathcal M}_{g,n}$} via {$3$}-spin
              structures},
   JOURNAL = {J. Amer. Math. Soc.},
  FJOURNAL = {Journal of the American Mathematical Society},
    VOLUME = {28},
      YEAR = {2015},
    NUMBER = {1},
     PAGES = {279--309},
      ISSN = {0894-0347,1088-6834},
   MRCLASS = {14H10 (14N35)},
  MRNUMBER = {3264769},
MRREVIEWER = {Shengmao\ Zhu},
       DOI = {10.1090/S0894-0347-2014-00808-0},
       URL = {https://doi.org/10.1090/S0894-0347-2014-00808-0},
}

@article{Kontsevich,
	AUTHOR = {Kontsevich, Maxim},
     TITLE = {Intersection theory on the moduli space of curves and the
              matrix {A}iry function},
   JOURNAL = {Comm. Math. Phys.},
  FJOURNAL = {Communications in Mathematical Physics},
    VOLUME = {147},
      YEAR = {1992},
    NUMBER = {1},
     PAGES = {1--23},
      ISSN = {0010-3616,1432-0916},
   MRCLASS = {32G15 (14H15 58F07 81T40)},
  MRNUMBER = {1171758},
MRREVIEWER = {Claude\ Itzykson},
       URL = {http://projecteuclid.org/euclid.cmp/1104250524},
}

@article {GKT,
    AUTHOR = {Gross, Mark and Kelly, Tyler L. and Tessler, Ran J.},
     TITLE = {Mirror symmetry for open {$r$}-spin invariants},
   JOURNAL = {Pure Appl. Math. Q.},
  FJOURNAL = {Pure and Applied Mathematics Quarterly},
    VOLUME = {20},
      YEAR = {2024},
    NUMBER = {2},
     PAGES = {1005--1024},
      ISSN = {1558-8599,1558-8602},
   MRCLASS = {14J33 (53D37)},
  MRNUMBER = {4734888},
MRREVIEWER = {Sergiy\ Koshkin},
       DOI = {10.4310/pamq.2024.v20.n2.a9},
       URL = {https://doi.org/10.4310/pamq.2024.v20.n2.a9},
}

@article {BCT2,
    AUTHOR = {Buryak, Alexandr and Clader, Emily and Tessler, Ran J.},
     TITLE = {Open {$r$}-spin theory {II}: {T}he analogue of {W}itten's
              conjecture for {$r$}-spin disks},
   JOURNAL = {J. Differential Geom.},
  FJOURNAL = {Journal of Differential Geometry},
    VOLUME = {128},
      YEAR = {2024},
    NUMBER = {1},
     PAGES = {1--75},
      ISSN = {0022-040X,1945-743X},
   MRCLASS = {14N35},
  MRNUMBER = {4773181},
MRREVIEWER = {Hsian-Hua\ Tseng},
       DOI = {10.4310/jdg/1721075259},
       URL = {https://doi.org/10.4310/jdg/1721075259},
}

@unpublished{GKT2,
  title={Open {FJRW} {T}heory and {M}irror {S}ymmetry},
  author={Gross, Mark and Kelly, Tyler L and Tessler, Ran J},
  note={arXiv preprint arXiv:2203.02435},
  year={2022}
}

@article{ChiodoRuan,
 AUTHOR = {Chiodo, Alessandro and Ruan, Yongbin},
     TITLE = {Landau-{G}inzburg/{C}alabi-{Y}au correspondence for quintic
              three-folds via symplectic transformations},
   JOURNAL = {Invent. Math.},
  FJOURNAL = {Inventiones Mathematicae},
    VOLUME = {182},
      YEAR = {2010},
    NUMBER = {1},
     PAGES = {117--165},
      ISSN = {0020-9910,1432-1297},
   MRCLASS = {14N35 (14J17 14J33 32G20 53D45)},
  MRNUMBER = {2672282},
MRREVIEWER = {Jake\ Philip\ Solomon},
       DOI = {10.1007/s00222-010-0260-0},
       URL = {https://doi.org/10.1007/s00222-010-0260-0},
}

@article{PSW,
 AUTHOR = {Pandharipande, R. and Solomon, J. and Walcher, J.},
     TITLE = {Disk enumeration on the quintic 3-fold},
   JOURNAL = {J. Amer. Math. Soc.},
  FJOURNAL = {Journal of the American Mathematical Society},
    VOLUME = {21},
      YEAR = {2008},
    NUMBER = {4},
     PAGES = {1169--1209},
      ISSN = {0894-0347,1088-6834},
   MRCLASS = {14N35 (14J32 14N10 53D45)},
  MRNUMBER = {2425184},
MRREVIEWER = {Hans\ U.\ Boden},
       DOI = {10.1090/S0894-0347-08-00597-3},
       URL = {https://doi.org/10.1090/S0894-0347-08-00597-3},
}

@unpublished{TZ1,
  title={The point insertion technique and open $r$-spin theories {I}: moduli and orientations},
  author={Tessler, Ran J. and Zhao, Yizhen},
  note={arXiv preprint arXiv:2310.13185},
  year={2023}
}

@unpublished{TZ2,
	Author = {Tessler, Ran J. and Zhao, Yizhen},
	note = {arXiv preprint arXiv:2311.11779},
	Title = {The point insertion technique and open $r$-spin theories {II}: intersection theories in genus-zero},
	Year = {2023}}

@article {HeLiShenWebb,
    AUTHOR = {He, Weiqiang and Li, Si and Shen, Yefeng and Webb, Rachel},
     TITLE = {Landau-{G}inzburg mirror symmetry conjecture},
   JOURNAL = {J. Eur. Math. Soc. (JEMS)},
  FJOURNAL = {Journal of the European Mathematical Society (JEMS)},
    VOLUME = {24},
      YEAR = {2022},
    NUMBER = {8},
     PAGES = {2915--2978},
      ISSN = {1435-9855,1435-9863},
   MRCLASS = {14J33 (14B05 14N35)},
  MRNUMBER = {4416593},
MRREVIEWER = {Helge\ Ruddat},
       DOI = {10.4171/jems/1155},
       URL = {https://doi.org/10.4171/jems/1155},
}

@incollection {Hori,
    AUTHOR = {Hori, Kentaro},
     TITLE = {Boundary {RG} flows of {$N=2$} minimal models},
 BOOKTITLE = {Mirror symmetry. {V}},
    SERIES = {AMS/IP Stud. Adv. Math.},
    VOLUME = {38},
     PAGES = {381--404},
 PUBLISHER = {Amer. Math. Soc., Providence, RI},
      YEAR = {2006},
      ISBN = {978-0-8218-4251-5; 0-8218-4251-X},
   MRCLASS = {81T40 (32S30 81T17)},
  MRNUMBER = {2282968},
MRREVIEWER = {Johannes\ Walcher},
       DOI = {10.1090/amsip/038/17},
       URL = {https://doi.org/10.1090/amsip/038/17},
}

@unpublished{Horiprivate,
	Author = {Hori, Kentaro},
	Note = {private communication},
	Title = {},
	Year = {}}

@unpublished{TZ3,
  title={Open $r$-spin  theory in genus one, and the {G}elfand-{D}ikii wave function},
  note = {arXiv preprint arXiv:2601.04114},
  author={Tessler, Ran J. and Zhao, Yizhen},
	Year = {2026}}

@article {BerglundHenningson,
    AUTHOR = {Berglund, Per and Henningson, M\aa ns},
     TITLE = {Landau-{G}inzburg orbifolds, mirror symmetry and the elliptic
              genus},
   JOURNAL = {Nuclear Phys. B},
  FJOURNAL = {Nuclear Physics. B. Theoretical, Phenomenological, and
              Experimental High Energy Physics. Quantum Field Theory and
              Statistical Systems},
    VOLUME = {433},
      YEAR = {1995},
    NUMBER = {2},
     PAGES = {311--332},
      ISSN = {0550-3213,1873-1562},
   MRCLASS = {58G10 (11F11 14J32 58G40 81T40)},
  MRNUMBER = {1310310},
MRREVIEWER = {Bruce\ Hunt},
       DOI = {10.1016/0550-3213(94)00389-V},
       URL = {https://doi.org/10.1016/0550-3213(94)00389-V},
}

@unpublished{polyvector,
  title={Primitive forms via polyvector fields},
  note = {arXiv preprint arXiv:1311.1659},
  author={Li, Changzheng and Li, Si and Saito, Kyoji},
	Year = {2026}}

@unpublished{GKT3,
  title={Open {F}ermat theory and mirror symmetry},
  author={Kelly, Tyler L. and Tessler, Ran J.},
  Note = {In preparation},
}
\bibliographystyle{amsalpha}

\end{document}